\documentclass[a4paper]{article}

\usepackage{graphicx} 
\usepackage{amssymb,amsfonts,amsthm,amsmath,mathtools,cleveref,comment,mathrsfs}
\usepackage{pgf}
\usepackage[margin=25mm,a4paper]{geometry}
\usepackage{xcolor}
\usepackage{paralist} 
\usepackage{tikz-cd}
\usepackage{todonotes}
\usepackage[backend=biber,style=alphabetic,sorting=nyt,maxbibnames=99,maxalphanames=5,doi=false,isbn=false,url=false]{biblatex}
\usepackage[normalem]{ulem}
\usepackage{float}
\usepackage{thm-restate}
\usepackage{epstopdf}
\usepackage{listings}
\usepackage{soul,cancel}
\usepackage{enumitem}

\usetikzlibrary{arrows}
\newcommand{\midarrow}{\tikz \draw[-angle 90] (0,0) -- +(.1,0);}
\newcommand{\midrevarrow}{\tikz \draw[-angle 90 reversed] (0,0) -- +(.1,0);}

\definecolor{codegreen}{rgb}{0,0.6,0}
\definecolor{codegray}{rgb}{0.5,0.5,0.5}
\definecolor{codepurple}{rgb}{0.58,0,0.82}
\definecolor{backcolour}{rgb}{0.95,0.95,0.92}

\lstdefinestyle{mystyle}{
    backgroundcolor=\color{backcolour},   
    commentstyle=\color{codegreen},
    keywordstyle=\color{magenta},
    numberstyle=\tiny\color{codegray},
    stringstyle=\color{codepurple},
    basicstyle=\ttfamily\footnotesize,
    breakatwhitespace=false,         
    breaklines=true,                 
    captionpos=b,                    
    keepspaces=true,                 
    numbers=left,                    
    numbersep=5pt,                  
    showspaces=false,                
    showstringspaces=false,
    showtabs=false,                  
    tabsize=2
}

\newtheorem{theorem}{Theorem}[section]
\newtheorem{lemma}[theorem]{Lemma}
\newtheorem{proposition}[theorem]{Proposition}
\newtheorem{corollary}[theorem]{Corollary}
\theoremstyle{definition}
\newtheorem{definition}[theorem]{Definition}
\theoremstyle{remark}
\newtheorem{remark}[theorem]{Remark}

\newtheorem{question}[theorem]{Question}

\newtheorem{example}[theorem]{Example}

\numberwithin{table}{section}

\newcommand{\defn}[1]{\ul{#1}}

\newcommand\cA{{\mathcal A}}

\newcommand\cD{{\mathcal D}}
\newcommand\cE{{\mathcal E}}
\newcommand\cF{{\mathcal F}}

\newcommand\cH{{\mathcal H}}

\newcommand\cM{{\mathcal M}}

\newcommand\cQ{{\mathcal Q}}
\newcommand\cV{{\mathcal V}}

\newcommand\BB{{\mathbb B}}
\newcommand\CC{{\mathbb C}}

\newcommand\FF{{\mathbb F}}

\newcommand\KK{{\mathbb K}}

\newcommand\PP{{\mathbb P}}

\newcommand\RR{{\mathbb R}}

\newcommand\TT{{\mathbb T}}

\newcommand\ZZ{{\mathbb Z}}

\newcommand\bbG{{\mathbb G}}
\newcommand\bbP{{\mathbb P}}

\newcommand\fe{{\mathfrak e}}
\newcommand\fg{{\mathfrak g}}

\newcommand\fh{{\mathfrak h}}

\newcommand\fl{{\mathfrak l}}
\newcommand\fS{{\mathfrak S}}
\newcommand{\twedge}{{\textstyle{\bigwedge}}}
\newcommand{\Hom}{{\mathrm{Hom}}}

\newcommand\scP{{\mathscr{P}}}

\newcommand{\cat}{\circ}
\newcommand{\rem}[2]{{#1}\!\setminus\!\{#2\}}
\newcommand{\remS}[2]{{#1}\!\setminus\!#2}
\newcommand{\ord}[1]{\overrightarrow{#1}}

\renewcommand{\st}{\colon} 
\newcommand{\set}[2]{\left\{#1 \,\st\, #2\right\}} 
\newcommand{\symDiff}{\mathbin\triangle} 
\newcommand{\InnerProd}[2]{\left( #1,#2\right)}
\newcommand{\scaledInnerProd}[2]{\left\langle #1,#2\right\rangle}
\newcommand{\groupGen}[1]{\left\langle #1\right\rangle}
\renewcommand{\emptyset}{\varnothing}

\DeclarePairedDelimiter\order{\lvert}{\rvert}%
\newcommand{\rev}[1]{{#1}^{T}}

\newcommand{\half}{\frac{1}{2}}
\newcommand{\tshalf}{{\textstyle\half}}
\newcommand{\floor}[1]{\left\lfloor #1 \right\rfloor}
\newcommand{\ceil}[1]{\left\lceil #1 \right\rceil}

\renewcommand{\det}{\mathrm{det}}

\DeclareMathOperator{\proj}{proj}
\DeclareMathOperator{\conv}{conv}
\DeclareMathOperator{\sgn}{sgn}
\DeclareMathOperator{\End}{End}
\DeclareMathOperator{\Cl}{Cl} 
\DeclareMathOperator{\Lie}{Lie} 
\newcommand{\vectorSpace}{{\cV}}
\DeclareMathOperator{\SO}{\mathrm{SO}}

\DeclareMathOperator{\Spin}{\mathrm{Spin}}
\DeclareMathOperator{\Par}{\mathrm{Par}}

\newcommand{\Sp}{\mathcal{S}}
\renewcommand{\so}{\mathfrak{so}}

\newcommand{\polyOf}[1]{P(#1)}
\newcommand{\ambientPoly}{\operatorname{conv}\{W\cdot \omega_J\}}
\newcommand{\cube}[2]{\square_{#1, #2}} 
\newcommand{\cubeVertex}[1]{\varepsilon_{#1}}
\newcommand{\face}{Q}

\newcommand{\crossPoly}{\lozenge}

\newcommand{\rootSystem}{\Phi} 
\newcommand{\simpleRoots}{\Pi} 
\newcommand{\maximalRoots}[1]{{S\setminus{\{s_{#1}\}}}}
\newcommand{\parabolic}[1]{P_{#1}} 
\newcommand{\maximalParabolic}[1]{\parabolic{\maximalRoots{#1}}} 
\newcommand{\matroid}{M}
\newcommand{\cosets}[1]{W^{#1}}
\newcommand{\maximalCosets}[1]{{W^{\maximalRoots{#1}}}}

\DeclareMathOperator{\trop}{trop}
\DeclareMathOperator{\ex}{ex}

\newcommand{\lng}[1]{\ell\left(#1\right)}
\newcommand{\lngS}[1]{\lng{\sigma_{#1}}}

\begin{document}
\title{Quadratic exchange equations for Coxeter matroids}
\author{Kieran Calvert\thanks{Lancaster University, E-mail: \texttt{kieran.calvert@lancaster.ac.uk}} \and 
Aram Dermenjian\thanks{University of Sevilla, E-mail: \texttt{aram.dermenjian.math@gmail.com}} \and 
Alex Fink\thanks{Queen Mary University of London, E-mail: \texttt{a.fink@qmul.ac.uk}} \and 
Ben Smith\thanks{Lancaster University, E-mail: \texttt{b.smith9@lancaster.ac.uk}}}

\maketitle

\begin{abstract}
Tropicalisation (with trivial coefficients) is a process that turns a polynomial equation into a combinatorial predicate on subsets of the set of variables.
We show that for each minuscule representation of a simple reductive group, there is a set of quadratic equations cutting out the orbit of the highest weight vector whose tropicalisation characterises the set of Coxeter matroids for that representation which satisfy the strong exchange property.
\end{abstract}

\section{Introduction}

Let $\KK$ be a field.
A linear space of dimension~$r$ over $\KK^n$
represents a matroid of rank~$r$ on the ground set $\{1,\ldots,n\}$,
and these linear spaces are parametrised by $\KK$-valued points of the Grassmannian $\mathrm{Gr}(r,n)$.
This relationship between the Grassmannian and matroids is used to great profit 
in the overlap of algebraic geometry and combinatorics, for example in tropical geometry.
Tropical linear spaces are defined to be in bijection with valuated matroids,
and the adjective ``valuated'' disappears if one works only with tropical linear spaces that are fans.
The parameter space for tropical linear spaces, known as the \defn{Dressian},
contains the tropicalisation of the Grassmannian.
The two do not agree, for one because the tropicalisation of the Grassmannian depends on~$\KK$, whereas defining a tropical linear space does not.
However, the ideal of the Grassmannian -- over $\ZZ$, and therefore over each $\KK$ --
contains a privileged set of quadrics, the \defn{quadratic (Grassmann--)Pl\"ucker relations},
and tropicalising only these quadrics cuts out exactly the Dressian.

Many have studied (see \Cref{ssec:litreview}) how well this happy state of affairs replicates for other quotients
$\bbG/\bbP$ of a reductive group by a parabolic subgroup.
In these studies, the first choice to be made 
is an embedding of~$\bbG/\bbP$ in a projective space with a prescribed coordinate basis,
which is required to define the tropicalisation.
Our objective in this paper is to do so uniformly in~$\bbG$.
Representation theory gives a uniform procedure
to define a projective embedding of~$\bbG/\bbP$ in a projectivised $\bbG$-representation $\proj(V_\lambda)$, 
in which it is cut out by quadrics (\Cref{t:eqfromcas}).
As for the coordinate basis, for Grassmannians we can concisely express the choice
as the basis of homogeneous linear forms in the weight grading for $\bbG$;
since the weight spaces are one-dimensional, this uniquely specifies our basis up to scalars.
The same approach works for $\bbG/\bbP$ 
when $V_\lambda$ is \defn{minuscule} (we also call $\bbP$ minuscule in this case),
i.e.\ all the weights of $V_\lambda$ are in a single Weyl group orbit.
This implies that all weight spaces are one-dimensional.
So it is the minuscule case we study here. 
For convenience in the representation theory, our paper will set $\KK=\CC$
(although see \Cref{rem:coefficients}).

On the other side of the connection, 
\defn{Coxeter matroids} \cite{BorovikGelfandWhite2003} are to $\bbG/\bbP$ as matroids are to the Grassmannian.
A Coxeter matroid is any subpolytope of the orbit polytope associated to~$\bbP$ of the Weyl group $W(\bbG)$
whose edges are parallel to roots of~$\bbG$ (see \Cref{ssec:Coxeter matroids} for a careful definition).
With this, we have every piece but one of our main theorem:

\begin{theorem}\label{thm:A}
Let $\bbG$ be a simply connected complex Lie group and $\bbP$ a minuscule parabolic subgroup.
There exists a set $Q$ spanning the quadrics cutting out $\bbG/\bbP\subset\proj(V_\lambda)$ 
such that a set of points is the vertex set of a strong Coxeter matroid if and only if it satisfies the tropicalisations of~$Q$.
\end{theorem}

The one word yet to be defined is ``strong''.
We say that a Coxeter matroid $P$ 
is \defn{strong} if for any two vertices $v,w$ of~$P$, either $\operatorname{conv}\{v,w\}$ is parallel to a root
or it is the diagonal of a parallelogram all four of whose vertices are vertices of~$P$ and with an edge parallel to a root.
(Such Coxeter matroids are also said to satisfy the \defn{strong exchange property}. For more details see \Cref{ssec:Coxeter matroids} below.) 
The condition of being strong will emerge from the structure of tropicalised quadrics homogeneous in the weight grading.
In particular, if $P$ is \defn{realizable}
in the sense that there is a point $x\in\bbG/\bbP$ 
so that a vertex $v$ of the orbit polytope is present in~$P$ if and only if the weight-$v$ coordinate of~$x$ is nonzero, 
then $P$ satisfies every tropical equation of $\bbG/\bbP$, so $P$ is strong.
All (ordinary) matroids are strong, but not, for example, all $\Delta$-matroids (\Cref{ssec:delta-matroids}).
For Coxeter matroids that are not strong, we are unable to suggest any characterisation by equations of tropical hypersurfaces, whatever their degree.

Borovik, Gelfand and White recognised the importance of minuscule representations in Open Problem 6.16.1 of \cite{BorovikGelfandWhite2003}.
Our theorem can be read as morally confirming the ``if'' direction of their expectation in that open problem,
strengthened by the connection to equations of~$\bbG/\bbP$.

Despite the statement of \Cref{thm:A} being type-agnostic, the proof is an analysis of each individual type in the classification of minuscule representations,
outside type~$A$ (usual Grassmannians) which is well known.
We sketch the structure of this paper. 
\Cref{s:Coxeter} is dedicated to background on Coxeter matroids.
Our type-by-type work begins in \Cref{s::charCoxMatroids}, where we characterise strong Coxeter matroids
in terms of the distribution of pairs of antipodal vertices on faces of the orbit polytope.
In \Cref{s:eqsfromGrass} we introduce the quadrics we need in each type, although their calculation is deferred to \Cref{app:equations},
and we make some use of computer algebra systems and the code in \Cref{app:code}.
We show the agreement of their tropicalisation with the characterisation that we obtained in \Cref{s::charCoxMatroids}, proving the cases of~\Cref{thm:A}.
\Cref{table:where} lists the classification of minuscule varieties and gives references to the results where the proof is achieved.
We have kept the subsectioning of Sections \ref{s::charCoxMatroids} and~\ref{s:eqsfromGrass} parallel,
so the reader interested in a particular type can easily pass between its treatment in the two sections.
\begin{table}
\begin{center}
\begin{tabular}{r|l}
  \Cref{t:Anmain}  & Type $A_n$ with minuscule variety associated to any root 
  \\
  \Cref{t:BnSpinmain}&Type $B_n$ with minuscule variety associated to the last root 
  \\
   \Cref{t:Dnspinmain} & Type $D_n$ with minuscule variety associated to either of the last two roots
  \\
   Lemma  \ref{l:Dncrossmain} & Type $D_n$ with the minuscule variety associated to the first root
 \\
   \Cref{l:Cncrossmain} &  Type $C_n$ with the minuscule variety associated to the first root 
     \\
    \Cref{p:mainthmE6} & Type $E_6$ and the two minuscule varieties associated to the first and last root  \\
    \Cref{p:mainthmE7} & Type $E_7$ and the Freudenthal variety 
\end{tabular}
\end{center}
\caption{Individual minuscule types and the statement of \Cref{thm:A} in each case.}
\label{table:where}
\end{table}

\paragraph*{Acknowledgements} 
The authors thank Chris Eur for his part in the initial impetus for this project and for pointing us to \cite{Kumar2002}, 
Matt Larson and Igor Makhlin for comments on the isomorphism between type $B_n$ and $D_{n+1}$ spinor Grassmannians,
Mauricio Velasco for pointers to the history of the Wick relations in the literature,
and Alessio D'Al\`\i, George Balla, Donggyu Kim, Dante Luber, Steve Noble, Felipe Rinc\'on, and Cynthia Vinzant for helpful conversations.
In the course of this project
the second author was supported by Heilbronn Institute for Mathematical Research.
The third author received support from the Engineering and Physical Sciences Research Council [grant number EP/X001229/1],
as well as the Institute for Advanced Study School of Mathematics
and the Fields Institute for Research in Mathematical Sciences.
The fourth author received support from the Engineering and Physical Sciences Research Council [grant number EP/X036723/1].

\subsection{Comparison to other work}\label{ssec:litreview}
The theorem that $\bbG/\bbP\subset\proj(V_\lambda)$ is cut out by quadrics obtained from the Casimir operator
was published at similar dates by Lichtenstein \cite{Lichtenstein:1982} and Garfinkle in her PhD thesis \cite{Garfinkle1982}, the latter attributing it to Kostant.
The quadrics can also be obtained through standard monomial theory as developed by Seshadri et al.\ \cite{seshadri1978geometry}. None of these three references gives the equations as explicitly as we need.

The first reference we know of that chooses generators of this space of quadric equations uniformly in $\bbG$ and $\bbP$ and makes their support explicit is Kumar's book \cite{Kumar2002}, in the proof of Corollary~10.1.11.
These agree with our equations in some types, but notably not in the nontrivial cases for classical groups, where our equations have smaller support than Kumar's.
In brief, in the equation of Kumar's whose ``middle'' monomial encodes a pair $v,w$ of vertices of~$P$,
there is another monomial for every pair $v+\alpha,w-\alpha$ where $\alpha$ is a root.
The prototypical shape of a root for a classical group is $\alpha=\pm e_i\pm e_j$.
The familiar exchange properties for matroids and their generalisations begin ``for all elements $i$, there exists an element $j$\ldots''.
Our equations model this structure, and thus a given one includes just the monomials for $v+\alpha,w-\alpha$ for a fixed $i$, with only $j$ varying.
(The monomial then ``in the middle'', by which we index the equation in this paper, is the monomial associated to $v\pm e_i,w\mp e_i$.)

This difference notwithstanding, we believe that when Kumar's equations are tropicalised
they determine the same class of Coxeter matroids as ours.
We plan to write a followup paper \cite{peerless_antipodes} proving this for the classical groups, by giving various equivalent characterisations of matroids and strong delta-matroids in terms of antipodes in faces of their polytopes.
It surprised us to learn that, in type $E_7$, one cannot use the equations of minimal support for \Cref{thm:A}; these impose a strictly stronger condition which excludes some strong $E_7$-matroids (Remark~\ref{rem:do+not+minimise+support}).

As this shows, the support of our equations is a type-dependent consideration, 
so our further discussion of precedents will be according to type.
We are aware of no relevant work on Coxeter matroids for exceptional groups.

Our main theorem in Type~$A$, \Cref{t:Anmain},
was known to Rota by 1971 \cite{RotaBowdoin}; see \cite[Section 3.3]{Kung1995geometricApproach} for further exposition.
Rota's interest in the Grassmannian and its coordinates was motivated by
his conjecture \cite{Rota1971} that realizability over any finite field was characterised by finitely many excluded minors.
Although the proof of this conjecture that has now been announced proceeds by the methods of structural combinatorics rather than algebraic geometry \cite{GeelenGerardsWhittle},
Rota's investigations did lead to the definition and study of the bracket ring \cite{White1975,White1977}.

There has been much subsequent interest in interactions between (type~$A$) matroids and Grassmannians, \cite[Section 4.1]{Ardila2018} being a survey.
When later authors have paid attention to the precise shape of the Grassmann--Pl\"ucker equations
it has often been because they were generalising matroids, whether by replacing the Grassmannian, as below,
or the coefficients.
Particularly clear is the discussion in the introduction of~\cite{DressWenzel1992}, the work introducing \defn{valuated matroids}, where the coefficients are real numbers representing the images of elements of $\KK$ under a nonarchimedean valuation.
Valuated matroids found a home in the field of tropical geometry, where they parametrise tropical linear spaces \cite[Remark 2.1]{Speyer2008},
and the coefficient object is interpreted as the tropical semifield $\TT=\RR\cup\{\infty\}$.
It's because of this line of work that we have adopted the language of tropicalising equations,
even though the solutions to tropical equations that we consider are valued in $\{0,\infty\}$ rather than in all of~$\TT$.

Valuated matroids can also be defined to be in bijection with regular subdivisions of matroid polytopes into matroid polytopes.
We will not use this perspective, but it appears in several of the works below,
and its generalisation to strong Coxeter matroids for minuscule representations is the theme of \cite{Frankfurt}.

In the other classical types $B$, $C$, and~$D$, there are fewer minuscule representations than in type~$A$.
The interesting ones come from the long root in type~$B$ and the two roots at the forked end of the Dynkin diagram in type~$D$.
The Coxeter matroids for both these types were introduced by Bouchet \cite{Bouchet1987} as \defn{$\Delta$-matroids},
as well as at similar dates by other authors (see \cite[Section 1]{Noble2025} for a survey).
What a realisation of a $\Delta$-matroid is, and therefore how to coordinatise $\Delta$-matroids, is a hairier question than for matroids;
from our perspective this is because of the potential to associate the geometry in any of the types $B$, $C$, or~$D$, 
or even other alternatives, like replacing the bilinear forms in these geometries with sesquilinear ones.
Declaring a realisation to be a point of $\bbG/\bbP$, as we do in this paper, is the approach of~\cite{BorovikGelfandWhite2003}, followed e.g.\ by \cite{Booth2001} introducing oriented $\Delta$-matroids.
In the structural literature, realisations are mostly taken to be symmetric or skew-symmetric square matrices, as in \cite{GeelenThesis}.
Many authors, following \cite{BouchetDuchamp}, further restrict attention to~$\FF_2$, symmetric and skew-symmetric matrices then being no different, at least off the diagonal.
From our perspective, these square matrices should be thought of as obtained from a Lagrangian subspace for a skew-symmetric or symmetric quadratic form (such subspaces being what $\bbG/\bbP$ parametrises) by taking a matrix spanning the subspace and deleting a maximal identity submatrix.

The first tropical investigations outside type~$A$ were made by Rinc\'on \cite{Rincon}, who considered even $\Delta$-matroids (type~$D$).
Here the \defn{Wick relations} generate the quadric equations of $\bbG/\bbP$.
The history of the Wick relations in the literature is complicated;
see Remark~\ref{rem:B+D+Gr+iso} for our attempts to trace it.
In \cite{Rincon} the tropical Wick relations appear as the definition of valuated even $\Delta$-matroid, without a direct comparison to the unvaluated case.
\cite[Theorem 5.4]{Rincon} is the regular subdivision statement connecting even $\Delta$-matroids to the Wick relations,
but our \Cref{p:strexcheqstrop} is not a consequence of Rinc\'on's theorem because he takes the support being a $\Delta$-matroid as an assumption.
Rinc\'on also states results about the tropicalisation of $\bbG/\bbP$ (i.e.\ of not just its quadrics), 
which from the $\Delta$-matroidal point of view is a question of realisability,
and about a definition of ``isotropic'' applying to tropical linear spaces.
Balla and Olarte \cite{BallaOlarte} investigate these latter two themes in all symplectic Grassmannians, the case of maximal parabolics in type~$C$.

In \cite{CCLV}, 
motivated by a notion of representability by Hermitian matrices that admits some non-strong $\Delta$-matroids,
Cheung, Chin, Liu and Vinzant carry out the regular subdivision programme for $\Delta$-matroids that need not be strong. 
They give a tropical characterisation of the height vectors 
that induce regular subdivisions of the cube into $\Delta$-matroids.
Not all their conditions on the vector are what we call ``tropical equations'' here, i.e.\ equations of tropical hypersurfaces, and their examples can be read as suggesting that a good characterisation by tropical hypersurfaces does not exist.
Instead, their conditions include assertions that the height vector lies in certain connected components of the complement of a tropical hypersurface.

Tracts \cite{BakerBowler} are algebraic structures providing a vast generalisation of $\TT$
and admitting a generalisation of tropical equations.
They therefore embrace not just ordinary matroids and valuated matroids but many other classes and variants of matroids from the literature.
Jin and Kim in \cite{JinKimO} show that the Wick relations achieve a good theory of even $\Delta$-matroids over tracts,
similarly rich to the theory \cite{BakerBowler} build for ordinary matroids over tracts using Pl\"ucker relations;
Jin and Kim's work, together with background on the Wick relations, \emph{does} imply our \Cref{p:strexcheqstrop}.
A sequel by Kim \cite{KimSp} treats the case of $\mathrm{SpGr}(n,2n)$ (type $C$, the maximal parabolic for the root $\alpha_n$),
using the coordinates of \cite{BDAKS}.
These coordinates can be interpreted as arising from $V_{\lambda_1}\otimes V_{\lambda_{n-1}}$, of which $V_{\lambda_n}$ is a submodule.
For some purposes these coordinates are better behaved than directly embedding into $\proj(V_{\lambda_n})$:
for instance they allow the formulation of a theory of fundamental circuits for these Coxeter matroids.

Finally, a non-minuscule family of Coxeter matroids that has been the subject of substantial literature is the \defn{flag matroids}.
Closest to our perspective is \cite{BrandtEurZhang};
see also the earlier \cite{Haque}, in which Proposition 3 is incorrect. 
Flag matroids are classically worked with \cite{Brylawski_1986} as tuples of matroids $M_1,\ldots,M_k$ that pairwise form matroid quotients.
A flag matroid is in fact a single type~$A$ Coxeter matroid, 
but the corresponding parabolic subgroup is not maximal, i.e.\ not minuscule.
For example, for \defn{full} flag matroids, $\bbP$ is a Borel subgroup $B$ in $\bbG=\mathrm{SL}_n$.
Accordingly, tropical work on flag matroids
has used not the projective coordinates of any embedding $\bbG/B\subset\proj(V_\lambda)$ directly,
but the \emph{multiprojective} coordinates on the factors of the embedding
\[\bbG/B\hookrightarrow\bbG/\bbP_1\times\cdots\times\bbG/\bbP_{n-1}
\subset\proj(V_{\lambda_1})\times\cdots\times\proj(V_{\lambda_{n-1}})\]
where $\bbP_1,\ldots,\bbP_{n-1}$ are all the maximal parabolics containing~$B$.
Each $\bbG/\bbP_i$ is a Grassmannian corresponding to one of the matroids in the flag.


\section{Coxeter matroids}\label{s:Coxeter}
In this section we give background on Coxeter matroids and some of their special cases.
We start by reviewing Coxeter groups, their associated hyperplane arrangements and the Coxeter complex.
We then describe Coxeter matroids and polytopes associated to them.
At the end of this section we describe the Coxeter matroids for certain parabolic subgroups in types $A$, $B$, and~$D$.
For a more thorough introduction, the interested reader is invited to look at \cite{BorovikGelfandWhite2003}.

Since Coxeter matroids of several types are main characters in this paper,
when we wish to refer to matroids in the standard sense we will often write ``ordinary matroids'' .
Ordinary matroids are one of several kinds of Coxeter matroids 
which can be modelled as \defn{set systems}: 
a set system on a finite set~$E$ is a set of subsets of~$E$.
In the polyhedral perspective we introduce below,
we will treat a Coxeter matroid $\matroid$ as a set system
when its polytope $\polyOf{\matroid}$ is a subpolytope of a cube.

\subsection{Coxeter groups}
In this section we review Coxeter groups and related structures.
The interested reader will find more details in \cite{Humphreys1990}, \cite{Humphreys1972}, \cite{Fulton+Harris:2004} and \cite{GoodmanWallach2010}.

\subsubsection{Coxeter groups, root systems and Weyl groups}\label{sssec:coxeter-groups-root-systems-and-weyl-groups}
Let $\vectorSpace$ be an $n$-dimensional Euclidean real vector space with inner product $(\cdot,\cdot)$.
For a given hyperplane $H\subset\vectorSpace$, we let $s_H$ denote the reflection which fixes $H$ pointwise and sends a normal vector of $H$ to its opposite.
A finite \defn{Coxeter group} $W$ is a finite group generated by a set of reflections in the orthogonal group $O(\vectorSpace)$.
A \defn{reflecting hyperplane} for~$W$ is a hyperplane $H$ such that $s_H\in W$.
The set of reflecting hyperplanes of $W$ forms the \defn{Coxeter arrangement of $W$}.

For a given $W$ we consider a \defn{root system} $\rootSystem\subset\vectorSpace$, i.e.\ a set of vectors closed under the action of~$W$, containing one opposite pair of normals to each of the reflecting hyperplanes of $W$.
Any choice of linear functional on $\vectorSpace$ that's nonzero on every root
separates the root system into a positive part $\rootSystem^+$ and a negative part $\rootSystem^-$.
The elements in these parts are called \defn{positive roots} and \defn{negative roots} respectively.
We let $\simpleRoots \subseteq \rootSystem^+$ be a minimal set with the property that every positive root is a positive linear combination of roots in $\simpleRoots$. The elements of $\simpleRoots$ are called \defn{simple roots}.
We let $S\subseteq W$ consist of the reflections which have a simple root as a normal vector to its reflecting hyperplane and name these the \defn{simple reflections}.
Then $S$ generates $W$.

We will require that our root system also have the property
\begin{equation} \label{eq:inner_int}
    \frac{2(\alpha, \beta)}{(\beta, \beta)} \in \ZZ \qquad \forall \alpha,\, \beta \in \rootSystem\, .
\end{equation}
Since we use this notation frequently, we let $\scaledInnerProd{\alpha}{\beta} \coloneq \frac{2(\alpha, \beta)}{(\beta, \beta)}$.
A \defn{Weyl group} is a Coxeter group together with the data of a fixed choice of root system satisfying \Cref{eq:inner_int}.

A root system can be associated to the type $B_n$ Coxeter group in two distinct ways so that \Cref{eq:inner_int} is satisfied.
This produces two different Weyl groups, type $B_n$ and type $C_n$, with isomorphic underlying Coxeter group.
The Coxeter group is sometimes accordingly called $BC_n$.
We do not use Coxeter groups that cannot be given the structure of a Weyl group,
but for completeness, in the classification of \defn{irreducible} Coxeter groups
i.e.\ those which are not direct products of two nontrivial Coxeter groups,
the Coxeter groups which are not Weyl groups make up one infinite family and two further examples,
known as  $I_2(m)$ $(m\ne3,4,6)$, $H_3$, and $H_4$.

The \defn{Dynkin diagram} of a Weyl group encodes the above data as a graph.
The vertex set of the Dynkin diagram is indexed by the simple roots. 
Edges are drawn between distinct $\alpha_i$ and $\alpha_j$ if $(\alpha_i, \alpha_j)\ne0$.
The number of such edges is given by $\scaledInnerProd{\alpha_i}{\alpha_j}\scaledInnerProd{\alpha_j}{\alpha_i}$ and, in addition, we draw an arrow pointing to the smaller of the two roots if they have different lengths.
Finiteness of~$W$ implies that $\scaledInnerProd{\alpha_i}{\alpha_j}\scaledInnerProd{\alpha_j}{\alpha_i}<4$,
which means that the Dynkin diagram determines all of the pairings $\scaledInnerProd{\alpha_i}{\alpha_j}\in\ZZ$.
The following is a list of all Dynkin diagrams for irreducible Weyl groups.

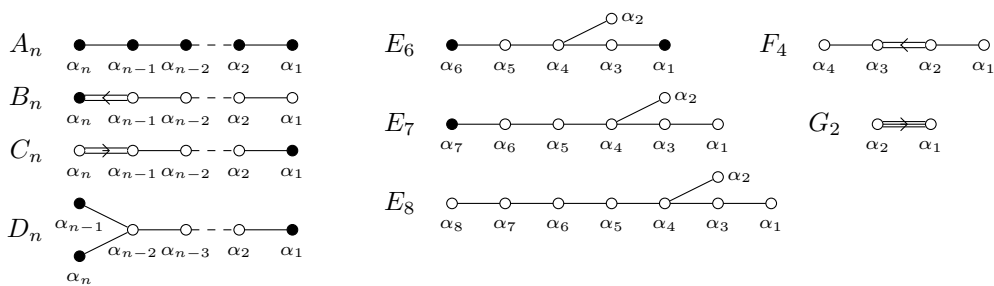
\begin{figure}[H]
    \centering
        \begin{tikzpicture}[scale=0.7]
        \begin{scope}[shift={(0,0)}]
            \draw (-0.5,0) node[anchor=east]  {$A_n$};
            
            \draw[fill=black] (0,0) circle(.1cm) node[below=.1cm] {\scriptsize$\alpha_n$};
            \draw[fill=black] (1,0) circle(.1cm) node[below=.1cm] {\scriptsize$\alpha_{n-1}$};
            \draw[fill=black] (2,0) circle(.1cm) node[below=.1cm] {\scriptsize$\alpha_{n-2}$};
            \draw[fill=black] (3,0) circle(.1cm) node[below=.1cm] {\scriptsize$\alpha_{2}$};
            \draw[fill=black] (4,0) circle(.1cm) node[below=.1cm] {\scriptsize$\alpha_1$};
            
            \draw (0.1, 0) -- +(0.8, 0);
            \draw (1.1, 0) -- +(0.8, 0);
            \draw[dashed] (2.1,0) -- +(0.8,0);
            \draw (3.1,0) -- +(0.8,0);
        \end{scope}
        \begin{scope}[shift={(0,-1)}]
            \draw (-0.5,0) node[anchor=east]  {$B_n$};
            
            \draw[fill=black] (0,0) circle(.1cm) node[below=.1cm] {\scriptsize$\alpha_n$};
            \draw (1,0) circle(.1cm) node[below=.1cm] {\scriptsize$\alpha_{n-1}$};
            \draw (2,0) circle(.1cm) node[below=.1cm] {\scriptsize$\alpha_{n-2}$};
            \draw (3,0) circle(.1cm) node[below=.1cm] {\scriptsize$\alpha_2$};
            \draw (4,0) circle(.1cm) node[below=.1cm] {\scriptsize$\alpha_1$};
            
            \draw (0.1, -0.05) -- +(0.8, 0);
            \draw[color=white] (0.1,0) -- node[color=black] {\midrevarrow} +(0.8,0);
            \draw (0.1, 0.05) -- +(0.8, 0);
            \draw (1.1, 0) -- +(0.8, 0);
            \draw[dashed] (2.1,0) -- +(0.8,0);
            \draw (3.1,0) -- +(0.8,0);
        \end{scope}
        \begin{scope}[shift={(0,-2)}]
            \draw (-0.5,0) node[anchor=east]  {$C_n$};
            
            \draw (0,0) circle(.1cm) node[below=.1cm] {\scriptsize$\alpha_n$};
            \draw (1,0) circle(.1cm) node[below=.1cm] {\scriptsize$\alpha_{n-1}$};
            \draw (2,0) circle(.1cm) node[below=.1cm] {\scriptsize$\alpha_{n-2}$};
            \draw (3,0) circle(.1cm) node[below=.1cm] {\scriptsize$\alpha_2$};
            \draw[fill=black] (4,0) circle(.1cm) node[below=.1cm] {\scriptsize$\alpha_1$};
            
            \draw (0.1, -0.05) -- +(0.8, 0);
            \draw[color=white] (0.1,0) -- node[color=black] {\midarrow} +(0.8,0);
            \draw (0.1, 0.05) -- +(0.8, 0);
            \draw (1.1, 0) -- +(0.8, 0);
            \draw[dashed] (2.1,0) -- +(0.8,0);
            \draw (3.1,0) -- +(0.8,0);
        \end{scope}

        \begin{scope}[shift={(0,-3.5)}]
            \draw (-0.5,0) node[anchor=east]  {$D_n$};
            
            \draw[fill=black] (0,0.5) circle(.1cm) node[below=.1cm] {\scriptsize$\alpha_{n-1}$};
            \draw[fill=black] (0,-0.5) circle(.1cm) node[below=.1cm] {\scriptsize$\alpha_n$};
            \draw (1,0) circle(.1cm) node[below=.1cm] {\scriptsize$\alpha_{n-2}$};
            \draw (2,0) circle(.1cm) node[below=.1cm] {\scriptsize$\alpha_{n-3}$};
            \draw (3,0) circle(.1cm) node[below=.1cm] {\scriptsize$\alpha_2$};
            \draw[fill=black] (4,0) circle(.1cm) node[below=.1cm] {\scriptsize$\alpha_1$};
            
            \draw (0.1, -0.45) -- +(0.8, 0.4);
            \draw (0.1, 0.45) -- +(0.8, -0.4);
            \draw (1.1, 0) -- +(0.8, 0);
            \draw[dashed] (2.1,0) -- +(0.8,0);
            \draw (3.1,0) -- +(0.8,0);
        \end{scope}
        \begin{scope}[shift={(7,0)}]
            \draw (-0.5,0) node[anchor=east]  {$E_6$};
            
            \draw[fill=black] (0,0) circle(.1cm) node[below=.1cm] {\scriptsize$\alpha_6$};
            \draw (1,0) circle(.1cm) node[below=.1cm] {\scriptsize$\alpha_5$};
            \draw (2,0) circle(.1cm) node[below=.1cm] {\scriptsize$\alpha_4$};
            \draw (3,0) circle(.1cm) node[below=.1cm] {\scriptsize$\alpha_3$};
            \draw (3,0.5) circle(.1cm) node[right] {\scriptsize$\alpha_2$};
            \draw[fill=black] (4,0) circle(.1cm) node[below=.1cm] {\scriptsize$\alpha_1$};
                                
            \draw (0.1,0) -- +(0.8,0);
            \draw (1.1, 0) -- +(0.8, 0);
            \draw (2.1, 0) -- +(0.8, 0);
            \draw (2.1, 0.05) -- +(0.8, 0.4);
            \draw (3.1, 0) -- +(0.8, 0);
        \end{scope}
        \begin{scope}[shift={(7,-1.5)}]
            \draw (-0.5,0) node[anchor=east]  {$E_7$};
            
            \draw[fill=black] (0,0) circle(.1cm) node[below=.1cm] {\scriptsize$\alpha_7$};
            \draw (1,0) circle(.1cm) node[below=.1cm] {\scriptsize$\alpha_6$};
            \draw (2,0) circle(.1cm) node[below=.1cm] {\scriptsize$\alpha_5$};
            \draw (3,0) circle(.1cm) node[below=.1cm] {\scriptsize$\alpha_4$};
            \draw (4,0.5) circle(.1cm) node[right] {\scriptsize$\alpha_2$};
            \draw (4,0) circle(.1cm) node[below=.1cm] {\scriptsize$\alpha_3$};
            \draw (5,0) circle(.1cm) node[below=.1cm] {\scriptsize$\alpha_1$};
                                
            \draw (0.1,0) -- +(0.8,0);
            \draw (1.1, 0) -- +(0.8, 0);
            \draw (2.1, 0) -- +(0.8, 0);
            \draw (3.1, 0.05) -- +(0.8, 0.4);
            \draw (3.1, 0) -- +(0.8, 0);
            \draw (4.1, 0) -- +(0.8, 0);
        \end{scope}
        \begin{scope}[shift={(7,-3)}]
            \draw (-0.5,0) node[anchor=east]  {$E_8$};
            
            \draw (0,0) circle(.1cm) node[below=.1cm] {\scriptsize$\alpha_8$};
            \draw (1,0) circle(.1cm) node[below=.1cm] {\scriptsize$\alpha_7$};
            \draw (2,0) circle(.1cm) node[below=.1cm] {\scriptsize$\alpha_6$};
            \draw (3,0) circle(.1cm) node[below=.1cm] {\scriptsize$\alpha_5$};
            \draw (4,0) circle(.1cm) node[below=.1cm] {\scriptsize$\alpha_4$};
            \draw (5,0) circle(.1cm) node[below=.1cm] {\scriptsize$\alpha_3$};
            \draw (5,0.5) circle(.1cm) node[right] {\scriptsize$\alpha_2$};
            \draw (6,0) circle(.1cm) node[below=.1cm] {\scriptsize$\alpha_1$};
                                
            \draw (0.1,0) -- +(0.8,0);
            \draw (1.1, 0) -- +(0.8, 0);
            \draw (2.1, 0) -- +(0.8, 0);
            \draw (3.1, 0) -- +(0.8, 0);
            \draw (4.1, 0) -- +(0.8, 0);
            \draw (4.1, 0.05) -- +(0.8, 0.4);
            \draw (5.1, 0) -- +(0.8, 0);
        \end{scope}
        \begin{scope}[shift={(14,0)}]
            \draw (-0.5,0) node[anchor=east]  {$F_4$};
            
            \draw (0,0) circle(.1cm) node[below=.1cm] {\scriptsize$\alpha_4$};
            \draw (1,0) circle(.1cm) node[below=.1cm] {\scriptsize$\alpha_3$};
            \draw (2,0) circle(.1cm) node[below=.1cm] {\scriptsize$\alpha_2$};
            \draw (3,0) circle(.1cm) node[below=.1cm] {\scriptsize$\alpha_1$};
            
            \draw (0.1, 0) -- +(0.8, 0);
            \draw (1.1, -0.05) -- +(0.8, 0);
            \draw[color=white] (1.1,0) -- node[color=black] {\midrevarrow} +(0.8,0);
            \draw (1.1, 0.05) -- +(0.8, 0);
            \draw (2.1, 0) -- +(0.8, 0);
        \end{scope}
        \begin{scope}[shift={(15,-1.5)}]
            \draw (-0.5,0) node[anchor=east]  {$G_2$};
            
            \draw (0,0) circle(.1cm) node[below=.1cm] {\scriptsize$\alpha_2$};
            \draw (1,0) circle(.1cm) node[below=.1cm] {\scriptsize$\alpha_1$};
            
            \draw (0.1, -0.05) -- +(0.8, 0);
            \draw (0.1, 0) -- node {\midarrow} +(0.8, 0) ;
            \draw (0.1, 0.05) -- +(0.8, 0);
        \end{scope}
    \end{tikzpicture}
        \caption{The Dynkin diagrams for all irreducible Weyl groups. The simple roots associated to minuscule fundamental weights are highlighted in black.}
        \label{fig:dynkin-diagrams}
\end{figure}

Note that if we forget which roots are longer, the diagrams for types $B_n$ and $C_n$, which give the same Coxeter group, become identical. 

\subsubsection{Weights and representations}\label{sssec:weights-and-representations}
For a given root system $\rootSystem$ satisfying \Cref{eq:inner_int} with a fixed choice of simple roots $\simpleRoots = \{\alpha_1, \ldots, \alpha_n\}$, let $\fg$ be its associated complex semisimple Lie algebra.
Let $\fh$ denote a Cartan subalgebra of $\fg$, i.e.\ a maximal subalgebra of $\fg$ consisting of semisimple elements.
A \defn{representation of $\fg$} is a pair $V_\pi \coloneqq (\pi, V)$ where $V$ is a complex vector space and $\pi : \fg \to \fg\fl(V)$ is a Lie algebra homomorphism.
We say $V_\pi$ is \defn{finite-dimensional} if $V$ is finite-dimensional and say it is \defn{irreducible} if there are no (non-zero) $\fg$-invariant proper subspaces.

For a representation $V_\pi$ and $\lambda \in \RR\rootSystem$ set
\[
    V(\lambda) = \set{v \in V}{\pi(h)v = \scaledInnerProd{\lambda}{h} v \text{ for all } h \in \fh}.
\]
Whenever $V(\lambda) \neq 0$ we call $\lambda$ a \defn{weight} of $V_\pi$ and $V(\lambda)$ a \defn{weight space}.
We let 
\[
    \Lambda = \set{\lambda \in \RR\rootSystem}{\scaledInnerProd{\lambda}{h}\in \ZZ \text{ for all } h \in \fh}
\]
denote the set of all \defn{weights}.
Letting $\{\lambda_i\}_{i \in [n]}$ denote the dual basis of the vectors $\left\{\frac{2\alpha_i}{(\alpha_i, \alpha_i)}\right\}_{i \in [n]}$,  the $\lambda_i$ are known as the \defn{fundamental weights}.
The set $\Lambda^+ = \set{\lambda \in \Lambda}{\lambda = \sum c_i \lambda_i,\; c_i \geq 0}$ is known as the set of \defn{dominant weights}.

It is well-known that a representation $V$ of $\fg$ can be decomposed such that $V = \bigoplus V(\lambda)$ where each $V(\lambda)$ is a weight space associated to a weight $\lambda$.
For every irreducible finite dimensional representation $V_\pi$ there is a unique dominant (integral) weight which is maximal under the order: $\mu \prec \nu$ if and only if $\nu - \mu$ is a sum of positive roots.
We call this maximal weight the \defn{highest weight} of $V_\pi$. 
Conversely, it is known that for every dominant (integral) weight $\lambda$ there is an irreducible finite-dimensional representation of $\fg$ whose highest weight is $\lambda$. 
In other words, without loss of generality, when we use the notations $V_\lambda$ for an irreducible finite-dimensional representation of $\fg$ we will assume that $\lambda$ is the unique highest weight of $V_\lambda$ unless otherwise specified.

\subsubsection{Parabolic subgroups and minuscule representations}
For a given subset $J \subseteq S = \{s_1, \ldots, s_n\}$, we let $\parabolic{J} = \groupGen{J}$ be the \defn{standard parabolic subgroup of $W$ generated by $J$} and $\cosets{J} = W / \parabolic{J}$ denote the \defn{standard parabolic cosets}.
The \defn{maximal parabolic subgroups} are precisely the standard parabolic subgroups of $W$  where $J = S \setminus \{s_i\}$ for some $i \in [n]$.
The maximal parabolic subgroups are an important part of our study due to their relationship with minuscule representations.

Suppose that $W$ is a Weyl group with a (fixed) root system $\rootSystem$ and simple roots $\simpleRoots$ and let $\fg$ be its associated semisimple Lie algebra. In the last subsection we presented a bijection between dominant (integral) weights and irreducible finite dimensional representations of $\fg$.
Given a representation $V_\lambda$ of $\fg$, let $J = \set{s \in S}{s(\lambda) = \lambda}$. 
A fundamental weight $\lambda_i$ is said to be \defn{minuscule} if $\scaledInnerProd{\lambda_i}{\alpha} \leq 1$ for all $\alpha \in \rootSystem^+$.
Equivalently, and crucially, $\lambda_i$ is minuscule if and only if the weights of $V_{\lambda_i}$ are in one Weyl group orbit. 
If $\lambda_i$ is minuscule, we say that the representation $V_{\lambda_i}$ is \defn{minuscule}.
 Similarly, we say that the maximal parabolic subgroup $\parabolic{J}$ for $J = S \setminus\{s_i\}$ is \defn{minuscule}.
 In \Cref{fig:dynkin-diagrams}, the simple roots associated to minuscule fundamental weights are highlighted in black.

The following is a list of all $J$ for which $\parabolic{J}$ is minuscule.
\begin{align}
    A_n&:\: S\setminus \{s_i\}\quad (i \in [n])
    &D_n&:\: S\setminus\{s_1\},\,S\setminus\{s_{n-1}\},\,\text{and }S\setminus\{s_n\}\notag\\
    B_n&:\: S\setminus \{s_n\}
    &E_6&:\: S\setminus\{s_6\} \text{ and }S \setminus \{s_1\} \label{eq:minuscule}\\
    C_n&:\: S\setminus\{s_1\}
    &E_7&:\: S\setminus\{s_7\}\notag
\end{align}
The simple roots associated to minuscule fundamental weights are shown in black in \Cref{fig:dynkin-diagrams}.
Note that Lie algebras of type $E_8$, $F_4$ and $G_2$ do not have any minuscule weights.

\subsection{Ordered (multi)sets and (multi)permutations}
\label{sssec:Symmetric-group}

Throughout this paper, in particular when prescribing the signs in our equations,
we will sometimes regard a set $A\subseteq [n]$ as an \defn{ordered set} $A = \{A_1, A_2, \ldots, A_m\}$, fixing which element $A_k$ appears at index~$k$ for each $k\in[m]$.
In fact we will need \defn{ordered multisets}, ordered sets whose elements might repeat.
Said otherwise, an ordered multiset is an arbitrary tuple, endowed with the multiset of its entries as its \defn{underlying} multiset.
An ordered set is an ordered multiset without repeated elements.

For two ordered multisets $A$ and $B$, we let $A \cat B$ denote their concatenation as an ordered multiset.
For an ordered multiset $A = \{A_1, \ldots, A_m\}$  we let $\rev{A} = \{A_m, \ldots, A_1\}$ denote the reversal of~$A$.
For an ordered set $A$ and a set $B$, we let $\remS{A}{B}$ be the ordered set obtained by deleting the elements of $A \cap B$ from~$A$, keeping the relative order of the rest.
As an abuse of notation, we use $\remS{A}{i}$ to mean $\rem{A}{i}$.
Additionally, if $i \in A$ we let $A(i)$ denote the index of $i$ in~$A$.
Note that $\rem{A}{i}$ and $A(i)$ are not (well-)defined for general ordered multisets since $i$ might appear more than once.

A \defn{permutation} $\sigma$ of an ordered multiset $A = \{A_1, A_2, \ldots, A_m\}$ is a bijection $[m]\to[m]$ on its set of indices.
When we wish to highlight that $A$ has repetitions, we say $\sigma$ is a \defn{multipermutation}.
Permutations act on ordered multisets by permuting indices: $\sigma(A)_k = A_{\sigma(k)}$.
Note that $A$ and $\sigma(A)$ have the same underlying multiset.
An \defn{ordering transposition} is a transposition of an ordered multiset which swaps two adjacent indices $i$ and $i+1$ such that $A_i > A_{i+1}$.
We emphasise that the inequality is strict; ordering transpositions may never swap adjacent elements that are equal.
If acting successively on an ordered multiset $A$ by an ordering transposition $\ell\ge0$ times 
transforms it to a weakly increasing ordered multiset $A'=\{A'_1,\ldots,A'_m\}$ with $A'_1\le\cdots\le A'_m$,
we write $\sigma_A$ for the composition of these ordering transpositions,
and call $\ell$ its \defn{length}%
\footnote{In the computer science literature, sorting a list this way,
without ever swapping equal elements, is called \defn{stable} sorting.}.
By abuse of notation, we will speak of the length as a property of $A$ or of~$\sigma_A$ indifferently,
and write $\lng{A}=\lng{\sigma_A}:=\ell$.
As the notation suggests, $\sigma_A$ and $\ell(\sigma_A)$ always exist and are independent of how ordering transpositions are chosen.

%
Though this notation suffices throughout, we do distinguish a special family of multipermutations.
Given two ordered sets $A$ and $B$, for readability we write $\sigma_{A,B}:=\sigma_{A\cat B}$ for the multipermutation that sorts $A \cat B$ into weakly increasing order,
and accordingly $\lng{\sigma_{A,B}} = \lng{A \cat B}$ for its length.
If we write $\sigma_{A,B}$ for sets $A$ and~$B$ to which we have not given an ordering,
we mean $A$ and $B$ to be increasingly ordered.

\begin{example}
    Let $n = 4$ and consider the ordered sets $I = \{1, 3, 4, 2\}$, $A = \{3, 2\}$ and $B = \{1, 4\}$.
    We have $\order{A} = \order{B} = 2$ and $\order{I} = 4$.
    Concatenating $A$ with $B$ gives $A\cat B = \{3, 2, 1, 4\}$ and $\rev{A} \cat B = \{2, 3, 1, 4\}$.
    We have that $\lng{\sigma_I} = 2$ as we need two ordering transpositions to reach the sorted ordered set:
    \[
        1342 \to 1324 \to 1234\, .
    \]
    We have $\remS{I}{3} = \{1, 4, 2\}$ and $\rem{I}{3, 2} = \remS{I}{A} = \{1, 4\} = B$.
    Finally, $I(2) = 4$, $I(3) = 2$ and $A(3) = 1$.
    Note that $B(3)$ is not defined since $3 \notin B$.

    Concatenating $I$ with $A$ gives the ordered multiset $I \cat A = \{1, 3, 4, 2, 3, 2\}$.
    Then the length of the permutation which sends $I \cat A$ to a weakly increasing ordered set is $\lng{\sigma_{I, A}} = 6$.
    The six ordering transpositions can be chosen to first push the left-hand $2$ to the left, then the right-hand $2$ to the left, and then finally the right-hand $3$ to the left:
    \[
        134232 \to 132432 \to 123432 \to 123423 \to 123243 \to 122343 \to 122334\, .
    \]
\end{example}

\subsection{Coxeter matroids}\label{ssec:Coxeter matroids}
\subsubsection{Coxeter matroids from polytopes}\label{sssec:coxeter-matroids-from-polytopes}
Given a Coxeter group $W$, 
a set of simple roots $\simpleRoots = \left\{ \alpha_1, \ldots, \alpha_n\right\}$ and a standard parabolic subgroup $\parabolic{J}$, we choose a point $\omega_J \in V$ such that
\begin{align} \label{eq:omega_J}
    \scaledInnerProd{\omega_J}{\alpha_i}  \begin{cases}
        < 0 &\text{if } s_i \notin J,\\
        = 0 &\text{otherwise.}
    \end{cases}
\end{align}
That is, $\omega_J$ is in the intersection of all hyperplanes of the Coxeter arrangement restricted to the parabolic subgroup and is in the negative half-space of the rest.
The choice of $\omega_J$ allows us to define an injective map from cosets $\cosets{J}$ to the vector space $\vectorSpace$ under which $w\parabolic{J} \mapsto w\cdot \omega_J$.
This map is well defined, as the choice of $\omega_J$ ensures that $w \cdot \omega_J = w' \cdot \omega_J$ for all $w,w' \in w\parabolic{J}$.
As such, we will not specify a representative of the coset and simply write $A \cdot \omega_J$ for the point associated to the coset $A \in W^J$.
Moreover, this bijection between points and cosets preserves the action of $W$, i.e., $w \cdot (A\cdot \omega_J) = (wA) \cdot \omega_J$.

Given a subset $M \subseteq W^J$, we let $\polyOf{\matroid}$ be the polytope constructed by taking the convex hull of $A \cdot \omega_J$ for all $A \in \matroid \subseteq \cosets{J}$.
When $M = W^J$, we call $P(M) = \ambientPoly$ the \defn{ambient polytope of $W^J$}.
We say that $M$ is a \defn{Coxeter matroid} if and only if every edge of $P(M)$ is parallel to a root in~$\rootSystem$.

We remark that the definition of Coxeter matroids adopted by \cite{BorovikGelfandWhite2003} is a characterisation using the Bruhat order on the subsets of~$\cosets{J}$.
Its equivalence to our definition is \cite[Theorem 6.3.1]{BorovikGelfandWhite2003}.

\begin{remark}\label{rem:combinatorial+equivalence}
Though the polytope $\polyOf{\matroid}$ depends on the choice of $\omega_J$, from a combinatorial perspective the precise choice is unimportant subject to the conditions in \eqref{eq:omega_J}.
Explicitly, given two different vectors $\omega_J$ and $\omega_J'$ satisfying \eqref{eq:omega_J}, the polytopes $\conv(w \omega_J \st w\parabolic{J} \in \matroid)$ and $\conv(w \omega_J' \st w\parabolic{J} \in \matroid)$ are combinatorially equivalent, i.e., have isomorphic face lattices, and corresponding faces under the isomorphism are parallel.
For the representation theory of Weyl groups, however, we will require $\omega_J$ to be a weight and work in the representation $V_{-\omega_J}$ of highest weight $-\omega_J$, which does depend on the choice.
\end{remark}

\begin{example}
    \label{ex:matroidPolys}
    Let $W$ be the type $B_3$ Coxeter group generated by $\{s_1, s_2, s_3\}$ where $(s_1 s_2)^3 = (s_2s_3)^4 = (s_1s_3)^2 = e$.
    Figure \ref{fig:B3-Coxeter-Complex} displays the Coxeter arrangement of~$W$ intersected by a $3$-cube (drawn solid, so only the front side is visible).
    The label on an edge is the generator of the parabolic subgroup associated to that edge.
    \begin{figure}[ht]
        \centering
        \begin{tikzpicture}
    [
        scale=3,
    ]
    \def\cubeL{1}
    \def\cubeX{0.6}
    \def\cubeY{0.7}
    \coordinate (bl1) at (0,0);
    \coordinate (bm1) at ({\cubeL/2}, 0);
    \coordinate (br1) at ({\cubeL},0);
    \coordinate (ml1) at (0, {\cubeL/2});
    \coordinate (mm1) at ({\cubeL/2}, {\cubeL/2});
    \coordinate (mr1) at ({\cubeL}, {\cubeL/2});
    \coordinate (tl1) at (0,{\cubeL});
    \coordinate (tm1) at ({\cubeL/2}, {\cubeL});
    \coordinate (tr1) at ({\cubeL},{\cubeL});

    \coordinate (bl2) at ({\cubeX/2}, {\cubeY/2});
    \coordinate (bm2) at ({\cubeX/2 + \cubeL/2}, {\cubeY/2});
    \coordinate (br2) at ({\cubeX/2 + \cubeL},{\cubeY/2});
    \coordinate (ml2) at ({\cubeX/2}, {\cubeY/2 + \cubeL/2});
    \coordinate (mm2) at ({\cubeX/2 + \cubeL/2}, {\cubeY/2 + \cubeL/2});
    \coordinate (mr2) at ({\cubeX/2 + \cubeL}, {\cubeY/2 + \cubeL/2});
    \coordinate (tl2) at ({\cubeX/2}, {\cubeY/2 + \cubeL});
    \coordinate (tm2) at ({\cubeX/2 + \cubeL/2}, {\cubeY/2 + \cubeL});
    \coordinate (tr2) at ({\cubeX/2 + \cubeL},{\cubeY/2 + \cubeL});

    \coordinate (bl3) at ({\cubeX},{\cubeY});
    \coordinate (bm3) at ({\cubeX + \cubeL/2}, {\cubeY});
    \coordinate (br3) at ({\cubeX + \cubeL},{\cubeY});
    \coordinate (ml3) at ({\cubeX}, {\cubeY + \cubeL/2});
    \coordinate (mm3) at ({\cubeX + \cubeL/2}, {\cubeY + \cubeL/2});
    \coordinate (mr3) at ({\cubeX + \cubeL}, {\cubeY + \cubeL/2});
    \coordinate (tl3) at ({\cubeX},{\cubeY + \cubeL});
    \coordinate (tm3) at ({\cubeX + \cubeL/2}, {\cubeY + \cubeL});
    \coordinate (tr3) at ({\cubeX + \cubeL},{\cubeY + \cubeL});

    \draw[thick] (bl1) -- (br1) -- (tr1) -- (tl1) -- (bl1);

    \draw[blue, thick, dotted] (br2) -- (tr2) -- (tl2);
    \draw[thick] (br3) -- (tr3) -- (tl3);

    \draw[thick] (tl1) -- (tl3);
    \draw[thick] (br1) -- (br3);
    \draw[thick] (tr1) -- (tr3);

    \draw[blue, thick, dotted] (tm3) -- (tm1) -- (bm1);
    \draw[red, thick, dashed] (tl1) -- (br1) -- (tr3);
    \draw[red, thick, dashed] (bl1) -- (tr1) -- (br3);
    \draw[blue, thick, dotted] (ml1) -- (mr1) -- (mr3);
    \draw[red, thick, dashed] (tl1) -- (tr3);
    \draw[red, thick, dashed] (tl3) -- (tr1);

    \def\labelOffset{0.07}
    \node at ({-\labelOffset}, {\cubeL/4}) {$s_1$};
    \node at ({\cubeL/4-2*\labelOffset/3},{3*\cubeL/4-2*\labelOffset/3}) {$s_2$};
    \node at ({\cubeL/2-\labelOffset}, {\cubeL/5}) {$s_3$};
            
\end{tikzpicture}
        \caption{The Coxeter arrangement of the type $B_3$ Coxeter group.
        The edges labelled by $s_1$ are solid black, those labelled by $s_2$ are dashed red and those labelled by $s_3$ are dotted blue. The labels represent right multiplication.}
        \label{fig:B3-Coxeter-Complex}
    \end{figure}
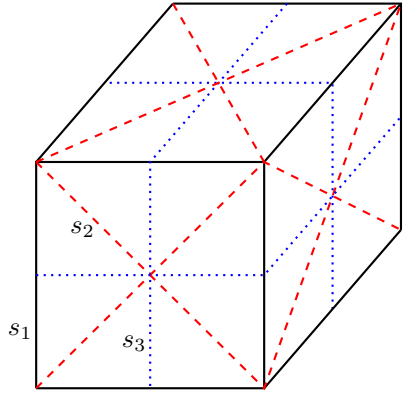
    
    Recall from \Cref{sssec:coxeter-groups-root-systems-and-weyl-groups} that this is the underlying Coxeter group of two different Weyl groups, types $B_3$ and $C_3$, with different root systems.
    The Weyl groups $B_3$ and~$C_3$ have one minuscule parabolic subgroup each,
    which are not equal as subsets of the Coxeter group.
    For $B_3$ the minimal parabolic subgroup is obtained when $J = \{s_1, s_2\} = S \setminus \{s_3\}$ and for $C_3$ the minimal parabolic subgroup is obtained when $J' = \{s_2, s_3\} = S \setminus \{s_1\}$.
    We next draw the orbit polytopes for these two subgroups.
    As we've discussed, these are the convex hull of the orbit of some fixed $\omega_J$ and $\omega_{J'}$ respectively.
    On the left, we have that the convex hull for the orbit of $\omega_J$ gives us the $3$-cube and on the right we have the $3$-cross polytope for the orbit of $\omega_{J'}$.
    Each vertex in the polytope is associated to a coset of $\parabolic{J}$ and $\parabolic{J'}$ respectively.
    \begin{figure}[ht]
        \centering
        \begin{tikzpicture}
            [
                scale=3,
                vertex/.style={inner sep=1.4pt,circle,draw=black,fill=black,thick},
            ]
            \def\cubeL{1}
            \def\cubeX{0.6}
            \def\cubeY{0.7}
            \begin{scope}
                \coordinate (bl1) at (0,0);
                \coordinate (bm1) at ({\cubeL/2}, 0);
                \coordinate (br1) at ({\cubeL},0);
                \coordinate (ml1) at (0, {\cubeL/2});
                \coordinate (mm1) at ({\cubeL/2}, {\cubeL/2});
                \coordinate (mr1) at ({\cubeL}, {\cubeL/2});
                \coordinate (tl1) at (0,{\cubeL});
                \coordinate (tm1) at ({\cubeL/2}, {\cubeL});
                \coordinate (tr1) at ({\cubeL},{\cubeL});
    
                \coordinate (bl2) at ({\cubeX/2}, {\cubeY/2});
                \coordinate (bm2) at ({\cubeX/2 + \cubeL/2}, {\cubeY/2});
                \coordinate (br2) at ({\cubeX/2 + \cubeL},{\cubeY/2});
                \coordinate (ml2) at ({\cubeX/2}, {\cubeY/2 + \cubeL/2});
                \coordinate (mm2) at ({\cubeX/2 + \cubeL/2}, {\cubeY/2 + \cubeL/2});
                \coordinate (mr2) at ({\cubeX/2 + \cubeL}, {\cubeY/2 + \cubeL/2});
                \coordinate (tl2) at ({\cubeX/2}, {\cubeY/2 + \cubeL});
                \coordinate (tm2) at ({\cubeX/2 + \cubeL/2}, {\cubeY/2 + \cubeL});
                \coordinate (tr2) at ({\cubeX/2 + \cubeL},{\cubeY/2 + \cubeL});
    
                \coordinate (bl3) at ({\cubeX},{\cubeY});
                \coordinate (bm3) at ({\cubeX + \cubeL/2}, {\cubeY});
                \coordinate (br3) at ({\cubeX + \cubeL},{\cubeY});
                \coordinate (ml3) at ({\cubeX}, {\cubeY + \cubeL/2});
                \coordinate (mm3) at ({\cubeX + \cubeL/2}, {\cubeY + \cubeL/2});
                \coordinate (mr3) at ({\cubeX + \cubeL}, {\cubeY + \cubeL/2});
                \coordinate (tl3) at ({\cubeX},{\cubeY + \cubeL});
                \coordinate (tm3) at ({\cubeX + \cubeL/2}, {\cubeY + \cubeL});
                \coordinate (tr3) at ({\cubeX + \cubeL},{\cubeY + \cubeL});

                \draw (bl1) -- (br1) -- (tr1) -- (tl1) -- (bl1);

                \draw[blue, dotted] (br2) -- (tr2) -- (tl2);
                \draw (br3) -- (tr3) -- (tl3);

                \draw (tl1) -- (tl3);
                \draw (br1) -- (br3);
                \draw (tr1) -- (tr3);

                \draw[blue, dotted] (tm3) -- (tm1) -- (bm1);
                \draw[red, dashed] (tl1) -- (br1) -- (tr3);
                \draw[red, dashed] (bl1) -- (tr1) -- (br3);
                \draw[blue, dotted] (ml1) -- (mr1) -- (mr3);
                \draw[red, dashed] (tl1) -- (tr3);
                \draw[red, dashed] (tl3) -- (tr1);

                \node[vertex] at (tl1) {};
                \node[vertex] at (tr1) {};
                \node[vertex] at (bl1) {};
                \node[vertex] at (br1) {};
                \node[vertex] at (tl3) {};
                \node[vertex] at (tr3) {};
                \node[vertex] at (bl3) {};
                \node[vertex] at (br3) {};
                
                \draw[ultra thick] (bl1) -- (br1) -- (tr1) -- (tl1) -- (bl1);
                \draw[ultra thick] (tl1) -- (tl3) -- (tr3) -- (br3) -- (br1);
                \draw[ultra thick] (tr1) -- (tr3);
                \draw[ultra thick, dashed] (bl1) -- (bl3) -- (tl3);
                \draw [ultra thick, dashed] (bl3) -- (br3);
                
            \end{scope}

            \begin{scope}[shift={({2 * \cubeL + \cubeX}, 0)}]
                \coordinate (bl1) at (0,0);
                \coordinate (bm1) at ({\cubeL/2}, 0);
                \coordinate (br1) at ({\cubeL},0);
                \coordinate (ml1) at (0, {\cubeL/2});
                \coordinate (mm1) at ({\cubeL/2}, {\cubeL/2});
                \coordinate (mr1) at ({\cubeL}, {\cubeL/2});
                \coordinate (tl1) at (0,{\cubeL});
                \coordinate (tm1) at ({\cubeL/2}, {\cubeL});
                \coordinate (tr1) at ({\cubeL},{\cubeL});
    
                \coordinate (bl2) at ({\cubeX/2}, {\cubeY/2});
                \coordinate (bm2) at ({\cubeX/2 + \cubeL/2}, {\cubeY/2});
                \coordinate (br2) at ({\cubeX/2 + \cubeL},{\cubeY/2});
                \coordinate (ml2) at ({\cubeX/2}, {\cubeY/2 + \cubeL/2});
                \coordinate (mm2) at ({\cubeX/2 + \cubeL/2}, {\cubeY/2 + \cubeL/2});
                \coordinate (mr2) at ({\cubeX/2 + \cubeL}, {\cubeY/2 + \cubeL/2});
                \coordinate (tl2) at ({\cubeX/2}, {\cubeY/2 + \cubeL});
                \coordinate (tm2) at ({\cubeX/2 + \cubeL/2}, {\cubeY/2 + \cubeL});
                \coordinate (tr2) at ({\cubeX/2 + \cubeL},{\cubeY/2 + \cubeL});
    
                \coordinate (bl3) at ({\cubeX},{\cubeY});
                \coordinate (bm3) at ({\cubeX + \cubeL/2}, {\cubeY});
                \coordinate (br3) at ({\cubeX + \cubeL},{\cubeY});
                \coordinate (ml3) at ({\cubeX}, {\cubeY + \cubeL/2});
                \coordinate (mm3) at ({\cubeX + \cubeL/2}, {\cubeY + \cubeL/2});
                \coordinate (mr3) at ({\cubeX + \cubeL}, {\cubeY + \cubeL/2});
                \coordinate (tl3) at ({\cubeX},{\cubeY + \cubeL});
                \coordinate (tm3) at ({\cubeX + \cubeL/2}, {\cubeY + \cubeL});
                \coordinate (tr3) at ({\cubeX + \cubeL},{\cubeY + \cubeL});

                \draw (bl1) -- (br1) -- (tr1) -- (tl1) -- (bl1);

                \draw[blue, dotted] (br2) -- (tr2) -- (tl2);
                \draw (br3) -- (tr3) -- (tl3);

                \draw (tl1) -- (tl3);
                \draw (br1) -- (br3);
                \draw (tr1) -- (tr3);

                \draw[blue, dotted] (tm3) -- (tm1) -- (bm1);
                \draw[red, dashed] (tl1) -- (br1) -- (tr3);
                \draw[red, dashed] (bl1) -- (tr1) -- (br3);
                \draw[blue, dotted] (ml1) -- (mr1) -- (mr3);
                \draw[red, dashed] (tl1) -- (tr3);
                \draw[red, dashed] (tl3) -- (tr1);

                \node[vertex] at (mm1) {};
                \node[vertex] at (tm2) {};
                \node[vertex] at (mr2) {};
                \node[vertex] at (mm3) {};
                \node[vertex] at (ml2) {};
                \node[vertex] at (bm2) {};
                \draw [ultra thick] (mm1) -- (tm2) -- (mr2) -- (mm1);
                \draw [ultra thick] (mm1) -- (tm2) -- (ml2) -- (mm1);
                \draw [ultra thick] (mm1) -- (bm2) -- (mr2) -- (mm1);
                \draw [ultra thick] (ml2) -- (mm1) -- (bm2);
                
                \draw[ultra thick, dashed] (ml2) -- (bm2);
                \draw[ultra thick, dashed] (bm2) -- (mm3) -- (ml2);
                \draw[ultra thick] (tm2) -- (mm3) -- (mr2);
            \end{scope}

        \end{tikzpicture}
        \caption{The ambient polytopes of $W^{\{s_1,s_2\}}$ and $W^{\{s_2,s_3\}}$ from Example~\ref{ex:matroidPolys}.}
        \label{fig:B3-Coset-Polytopes}
    \end{figure}
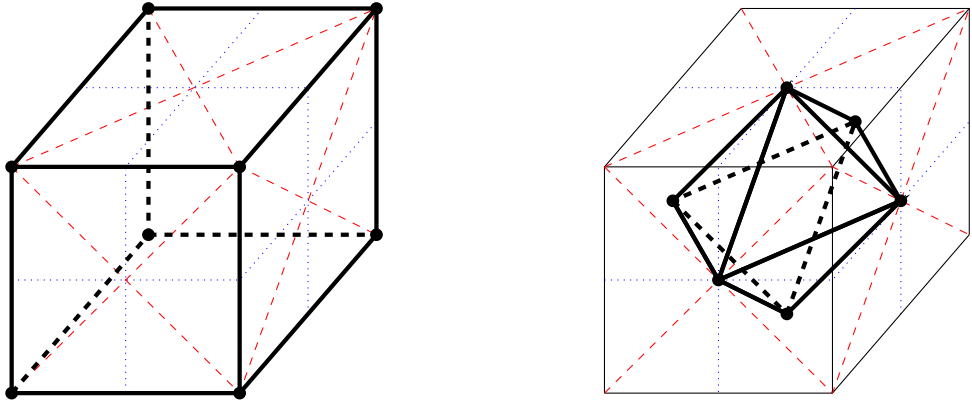
\end{example}

\subsubsection{Strong exchange property}
The \defn{strong exchange property for Coxeter matroids} is the following property 
that a Coxeter matroid $\matroid \subseteq \cosets{J}$ may have:
\begin{gather} \label{eq:SEP}
    \text{for any distinct cosets }A, B \text{ in } \matroid \subseteq \cosets{J} \\ 
    \text{ there is a reflecting hyperplane } H 
    \text{ separating } A \text{ and } B
    \text{ such that } s_H A, s_H B \in \matroid\, .\nonumber
\end{gather}
We say a Coxeter matroid is \defn{strong} if it satisfies the strong exchange property.
In fact the strong exchange property is a sufficient condition for an arbitrary subset $\matroid \subseteq \cosets{J}$ to form a Coxeter matroid.

For ordinary matroids, all matroids are strong, because recasting the strong exchange property in terms of bases yields the so-called symmetric basis exchange property, which is equivalent to the usual basis exchange property~\cite{Brualdi1969}.
But not every Coxeter matroid satisfies the strong exchange property in general.
Borovik, Gelfand and White ask the following within Open Problem 6.16.1 of \cite{BorovikGelfandWhite2003} as a potential characterisation of strong Coxeter matroids.
\begin{question}\label{q:6.16.1}
    Let $W$ be a Coxeter group and let $P$ be a parabolic subgroup of $W$.
    Is it true that every matroid for $W$ and $P$ satisfies the strong exchange property if and only if $P$ is a minuscule weight parabolic subgroup?
\end{question}

The next example answers this question in the negative,
showing that there exists a matroid $M$ for $W$ and $P$ such that $P$ is a minuscule weight parabolic subgroup and $M$ does not satisfy the strong exchange property.
\begin{example}\label{ex:3D-counterexample}
Let $W$ be the type $B_n$ Coxeter group generated by $\{s_1, s_2, s_3\}$ and let $J = \{s_1, s_2\}$.
Then $\parabolic{J}$ is a minuscule weight parabolic.
Fixing $\omega_J = (-\tshalf,-\tshalf,-\tshalf)$, the ambient polytope $\ambientPoly$ is the 3-cube $[-\tshalf,\tshalf]^3$ with vertices associated to $\cosets{J}$.
Suppose $M\subseteq \cosets{J}$ has polytope
\[
    \conv\{
        (-\tshalf,-\tshalf,-\tshalf),\,
        (-\tshalf,-\tshalf,\tshalf),\,
        (-\tshalf,\tshalf,-\tshalf),\,
        (\tshalf,-\tshalf,-\tshalf),\,
        (\tshalf,\tshalf,\tshalf)
    \}\, ,
\]
as displayed in Figure \ref{fig:3D-counterexample}.
As every edge is parallel to a root, this subset of cosets gives us a Coxeter matroid by definition.
Then taking $A$ and $B$ to be the cosets whose vertices are shown in the figure, one can verify that no such hyperplane $H$ exists such that both $s_H A$ and $s_H B$ are in~$M$.
\end{example}
\begin{figure}[ht]
    \centering
    \includegraphics[scale=0.3]{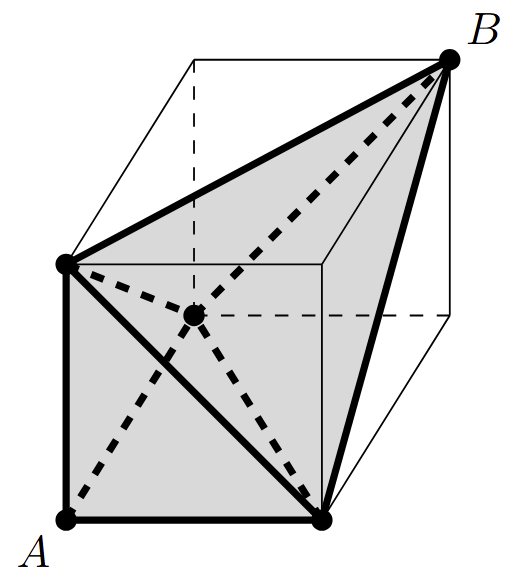}
    \caption{An example of a Coxeter matroid that is not a strong Coxeter matroid from Example~\ref{ex:3D-counterexample}.}
    \label{fig:3D-counterexample}
\end{figure}

\begin{lemma} \label{lem:face_reflection}
Let $M \subseteq W^J$ be a strong Coxeter matroid, $\ambientPoly$ its ambient polytope and $P(M)$ the corresponding Coxeter matroid polytope.
Let $A,B \in M$ and let $F \subseteq P(M)$ be a face containing both $A\cdot\omega_J$ and $B\cdot\omega_J$.
If $H$ is a reflecting hyperplane separating $A$ and $B$ such that $s_H A, s_H B \in M$, then $(s_H A)\cdot\omega_J$ and $(s_H B)\cdot\omega_J$ are both also contained in $F$.
\end{lemma}
\begin{proof}
Consider the line segment $[A\cdot\omega_J,B\cdot\omega_J]$.
If $H$ separates $A\cdot\omega_J$ and $B\cdot\omega_J$, then the midpoint $\gamma$ of the line segment must lie on $H$.
In particular, it remains fixed under the reflection through $H$ and so $\gamma$ is also the midpoint of $[(s_H A)\cdot\omega_J,(s_H B)\cdot\omega_J]$.
Suppose $(s_H A)\cdot\omega_J \notin F$. Then there exists some linear function $\ell$ such that $\ell(x) \leq 0$ for all $x \in P(M)$ and the inequality is tight on~$F$, but $\ell((s_H A)\cdot\omega_J) < 0$.
As $\ell(\gamma) = 0$, it follows by convexity that $\ell((s_H B)\cdot\omega_J) > 0$ and hence $(s_H B)\cdot\omega_J$ is not contained in $P(M)$.
This is a contradiction, whence $(s_H A)\cdot\omega_J, (s_H B)\cdot\omega_J \in F$.
\end{proof}

\subsection{Ordinary matroids}
\label{ssec:ordinary-matroids}
In this section we describe matroids, first as set systems then as type $A$ Coxeter matroids. 
Let $\scP_{n, k}$ denote the set of $k$-element subsets of~$[n]$.
A set system $\matroid \subseteq \scP_{n, k}$ satisfies the \defn{basis exchange property} if
\[
    \text{for all } A, B \in \matroid \text{ and } a \in A \setminus B\text{, there exists } b \in B \setminus A\text{ such that } A \symDiff\{a, b\} \in \matroid\, .
\]
We call such a subset $\matroid$ satisfying the basis exchange property an \defn{(ordinary) matroid}.
There are many cryptomorphic definitions, but the one we will focus on is the strong exchange property for ordinary matroids, as given in \cite[Section 1.5.1]{BorovikGelfandWhite2003}.

A set system $\matroid \subseteq \scP_{n, k}$ satisfies the \defn{strong exchange property for (ordinary) matroids} if
\[
    \text{for all distinct }A, B \in \matroid \text{ there exist } a \in A \setminus B,\,b\in B\setminus A \text{ such that } A \symDiff\left\{ a, b \right\}, B\symDiff\left\{ a, b \right\} \in \matroid \, .
\]
It is shown in \cite[Theorem 1.5.2]{BorovikGelfandWhite2003} that a subset $\matroid \subseteq \scP_{n, k}$ satisfies the strong exchange property if and only if it is a matroid.

\begin{remark}
Brualdi gives another cryptomorphic description of matroids via the symmetric basis exchange property \cite{Brualdi1969}.
Brualdi's definition is also sometimes known as the strong exchange property in the literature, though we emphasise it is not a specialisation of \eqref{eq:SEP} to type $A$.
To avoid confusion, we will only ever use ``strong exchange property'' when referring to \eqref{eq:SEP} or its specialisation to a specific type.
\end{remark}

We now describe how matroids arise as Coxeter matroids of type~$A_{n-1}$.
Fix $\vectorSpace \cong \RR^n$ with orthonormal basis $\{e_1, \dots, e_n\}$.
The root system of $A_{n-1}$ has a choice of simple roots $\simpleRoots(A_{n-1}) = \{\alpha_i \coloneqq e_{i+1} - e_i\}$.
The corresponding Weyl group $W = W(A_{n-1}) \cong \fS_{n}$ is isomorphic to the symmetric group where the simple roots correspond to adjacent transpositions.

Fix some simple root $\alpha_k$ and the corresponding group of simple reflections $J = S \setminus \{s_k\}$.
The classification of minuscule representations, \Cref{eq:minuscule}, states that $P_J$ is minuscule for all $1 \leq k \leq n-1$.
To compute the ambient polytope $\ambientPoly$, we choose $\omega_J = \sum_{i=1}^k e_i$: note that this satisfies \eqref{eq:omega_J}.
Acting by $W$ permutes the coordinates of $\omega_J$, hence the ambient polytope of $\cosets{J}$ is the hypersimplex:
\begin{equation} \label{eq:hypersimplex}
    \cH_{n,k} := \conv\{W(A_{n-1})\cdot \omega_{S \setminus \{s_k\}}\} = \conv\set{\sum_{i \in I} e_i}{I \in \scP_{n, k}} \, .
\end{equation}
As the vertices are indexed by elements of $\scP_{n, k}$, this allows us to translate between matroids as set systems in terms of bases, subpolytopes of the hypersimplex, and as subsets of the cosets $W^J$.

The strong exchange property for Coxeter matroids in type $A$ translates exactly to the strong exchange property for (ordinary) matroids.
Explicitly, any reflecting hyperplane $H$ is orthogonal to some root $e_a - e_b$, so any distinct cosets $A, B \in M$ separated by hyperplane $H$ must have $a \in A \setminus B$ and $b \in B \setminus A$ or vice versa.
It is straightforward to verify $s_H(A) = A \symDiff \{a,b\}$ and $s_H(B) = B \symDiff \{a,b\}$.
This shows that matroids on $n$ elements are precisely the strong Coxeter matroids of type $A_{n-1}$.
It was shown in \cite{GelfandGoreskyMacphersonSerganova} that $M \subseteq \scP_{n,k}$ is a matroid if and only if each edge of $P(M)$ is parallel to some root direction $e_a - e_b$, i.e.\ $M$ is a Coxeter matroid of type $A_{n-1}$.
A corollary of these two results is that every type $A$ Coxeter matroid is strong.

\subsection{$\Delta$-matroids}\label{ssec:delta-matroids}
In this section we define $\Delta$-matroids and even $\Delta$-matroids, first as set systems then as type $B_n$ and $D_n$ Coxeter matroids respectively. 

\begin{definition}[{\cite[Definition 6]{Bouchet1987}}]
    A \defn{$\Delta$-matroid} is a set system $\matroid$ on $[n]$ such that the following \defn{symmetric exchange property for $\Delta$-matroids} is satisfied:
    \begin{equation}\label{axiom:symm-ex}
        \text{For all } A,B \in \matroid, a \in A \symDiff B, \text{ there exists } b \in A \symDiff B \text{ such that } A \symDiff \{a, b\} \in \matroid \, .
    \end{equation}
    Note that $a$ and $b$ might be equal.
    
    We say that a set system $\matroid$ is \defn{even} if all $A \in \matroid$ have the same parity.
    We say that $\matroid$ is an \defn{even $\Delta$-matroid} if it is a $\Delta$-matroid and an even set system.
    Despite the naming convention, we emphasise that the sets may be of odd or even cardinality.
\end{definition}

As with the other types of matroids, we consider the following \defn{strong exchange property for $\Delta$-matroids}:
\begin{equation}\label{axiom:strong-symm-ex}
    \text{For all } A,B \in \matroid, \text{ there exist }a,\, b  \in A \symDiff B,\, \text{ such that } A \symDiff \{a,b\} \, , \, B \symDiff \{a,b\} \in M \, .
\end{equation}
We emphasise that \eqref{axiom:strong-symm-ex} again allows $a, b$ to be equal.
This is a specialisation of the strong exchange property \eqref{eq:SEP} for Coxeter matroids to type $B_n$. 
In addition, we consider the following \defn{strong exchange property for even $\Delta$-matroids}:
\begin{equation}\label{axiom:even-strong-symm-ex}
    \text{For all } A,B \in \matroid, \text{ there exist distinct }a,\, b  \in A \symDiff B,\, \text{ such that } A \symDiff \{a,b\} \, , \, B \symDiff \{a,b\} \in M \, .
\end{equation}
This is a specialisation of the strong exchange property \eqref{eq:SEP} for Coxeter matroids to type $D_n$. 

We now describe how all (resp.\ even) $\Delta$-matroids arise as Coxeter matroids of type $B_n$ (resp.\ $D_n$).
The root systems of $B_n$ and~$D_n$ have the following choices of simple roots:
\begin{align*}
\simpleRoots(B_n) &= \set{\alpha_i \coloneqq e_{i} - e_{i+1}}{1 \leq i \leq n-1} \cup \{\alpha_n \coloneqq e_n\} \, , \\
\simpleRoots(D_n) &= \set{\alpha_i \coloneqq e_{i} - e_{i+1}}{1 \leq i \leq n-1} \cup \{\alpha_n \coloneqq e_{n-1} + e_n\}\, .
\end{align*}
For $W = W(B_n)$, the only minuscule parabolic subgroup is $P_J$ where $J = S \setminus \{s_n\}$.
By concretely choosing $\omega_J=(-\frac12,\ldots,-\frac12)$, we see that the ambient polytope is the cube $[-\frac12,\frac12]^n$.
As with ordinary matroids, we can identify set systems with subpolytopes of this cube by labelling its vertices by sets.
For a set $I\subseteq[n]$, define $\cubeVertex{I}\in \vectorSpace=\RR^n$ to be the indicator vector of~$I$ translated by $(-\frac12,\ldots,-\frac12)$. Then
\begin{equation}\label{eq:cubeVertex}
Q_n \coloneqq \set{\cubeVertex{I}}{I\subseteq[n]} = W(B_n) \cdot \omega_{S \setminus \{s_n\}} \, ,
\quad (\cubeVertex{I})_i = \begin{cases}
    -\frac12 & i\not\in I\\
    \frac12 & i\in I\, .
    \end{cases}
\end{equation}
The ambient polytope is the cube obtained as the convex hull of these points.
With this dictionary established, it follows that the strong exchange property for $\Delta$-matroids is a specialisation of the strong exchange property for Coxeter matroids to type $B_n$ for the parabolic subgroup that excludes $s_n$.

For $W = W(D_n)$, there are two minuscule parabolic subgroups\footnote{The third minuscule parabolic subgroup $\maximalParabolic{1}$ has the cross polytope as its ambient polytope and will be explored in \Cref{ssec:Dn-cross-poly}.} that give rise to $\Delta$-matroids: $\maximalParabolic{n-1}$ and $\maximalParabolic{n}$.
The ambient polytope of these are the two demicubes inside the cube whose vertices have a prescribed parity.
To see this, observe that acting on $\cubeVertex{I}$ by $W(D_n)$ does not change the parity of $I$.
For $J = S \setminus \{s_n\}$, we can take the same $\omega_J=\cubeVertex{\emptyset}$ as before, hence
\begin{equation}\label{eq:demicubeVertex+}
    \cD_{n,+} \coloneqq \set{\cubeVertex{I}}{|I| \equiv 0 \mod 2} = W(D_n) \cdot \omega_{S \setminus \{s_n\}} \, .
\end{equation}
The ambient polytope is the demicube obtained as the convex hull of these points.

For $J = S \setminus \{s_{n-1}\}$, we can take $\omega_J = (-\frac12,\ldots,-\frac12,\frac12) = \cubeVertex{n}$, and hence 
\begin{equation}\label{eq:demicubeVertex-}
    \cD_{n,-} \coloneqq \set{\cubeVertex{I}}{|I| \equiv 1 \mod 2} = W(D_n) \cdot \omega_{S \setminus \{s_{n-1}\}} \, .
\end{equation}
The ambient polytope is the demicube obtained as the convex hull of these points.
Note that this gives a partition of the cube vertices $Q_n = \cD_{n,+} \cup \cD_{n,-}$ into the demicube vertices.
This dictionary allows us to go between even $\Delta$-matroids and subpolytopes of these demicubes.
In particular, it follows that the strong exchange property for even $\Delta$-matroids is a specialisation of the strong exchange property for Coxeter matroids to type $D_n$.

In type $D_n$, we can upgrade the strong exchange property for even $\Delta$-matroids to \defn{Wenzel's exchange property}\footnote{Wenzel's exchange property is also commonly referred to as the ``strong exchange property''. Introduced by Wenzel~\cite{Wenzel1993}, it has been used throughout the literature, including in Chapter 4 of \cite{BorovikGelfandWhite2003}. We have renamed it to avoid confusion with the strong exchange property for Coxeter matroids.}:
\begin{equation}\label{axiom:wenzel-exchange}
    \text{For all } A,B \in \matroid, a \in A \symDiff B, \text{ there exists } b \in (A \symDiff B) \setminus a \text{ such that } A \symDiff \{a,b\} \, , \, B \symDiff \{a,b\} \in M \, .
\end{equation}
 
 Wenzel's main result is that a $\Delta$-matroid satisfies property \eqref{axiom:wenzel-exchange} if and only if it is even (\cite[Theorem 1]{Wenzel1993}). This result can then be used to show that every even $\Delta$-matroid satisfies the strong exchange property for even $\Delta$-matroids.
\begin{theorem}
    \label{t:evenIsStrong}
    Let $M \subseteq [n]$. The following are equivalent:
    \begin{enumerate}
        \item $M$ is an even $\Delta$-matroid; \label{i:delta+1}
        \item $M$ satisfies the strong exchange property for even $\Delta$-matroids; \label{i:delta+2}
        \item $M$ satisfies Wenzel's exchange property. \label{i:delta+3}
    \end{enumerate}
\end{theorem}
\begin{proof}
    By \cite[Theorem 1]{Wenzel1993} \Cref{i:delta+1} and \Cref{i:delta+3} are equivalent.
    It's clear that \Cref{i:delta+3} implies \Cref{i:delta+2}.
    Additionally \Cref{i:delta+2} implies \Cref{i:delta+1} as any strong Coxeter matroid is a Coxeter matroid.
\end{proof}

In contrast, not all $\Delta$-matroids satisfy the strong exchange property for $\Delta$-matroids.
We say a $\Delta$-matroid is \defn{strong} if it satisfies the strong exchange property for $\Delta$-matroids.
Recasting \Cref{ex:3D-counterexample} in the language of $\Delta$-matroids, we obtain the following nonexample.
\begin{example} \label{ex:delta+not+strong}
    Let $M$ be the following set system on $[3]$:
    \[
        M = \{\emptyset, 1, 2, 3, 123\} \, .
    \]
    These sets correspond to vertices of the polytope from \Cref{ex:3D-counterexample} under the identification \eqref{eq:cubeVertex}.
    It follows that $M$ is a $\Delta$-matroid but does not satisfy the strong exchange property.
    Taking $A = \emptyset$ and $B = 123$, we see there is no choice of $a$ and $b$ such that both $A \symDiff \{a,b\}$ and $B \symDiff \{a,b\}$ are contained in $M$, as at least one must have cardinality two.
\end{example}

\begin{remark}
    While Wenzel's exchange property was originally defined for all $\Delta$-matroids, \Cref{t:evenIsStrong} implies that the property is never satisfied for non-even $\Delta$-matroids.
    In particular, it is a strictly stronger property than the strong exchange property for $\Delta$-matroids.
\end{remark}

\section{Characterising strong Coxeter matroids} \label{s::charCoxMatroids}

In this section we give combinatorial characterisations of strong minuscule Coxeter matroids $\matroid$, type by type.
At least for the classical types, similar results are already in the literature,
but we want characterisations of a special form.
We only use conditions on the polytope $\polyOf{\matroid}$ that consider whether pairs of antipodal vertices on a given face $\face$ of the ambient polytope $\ambientPoly$ can be present,
in particular by requiring that among a chosen set of antipodes in~$\face$, not exactly one pair of antipodal vertices can appear in $\polyOf{\matroid}$.
The reason for preferring these particular conditions is that they can be cast as tropical equations, defined below, that are suitably homogeneous.
Then in \Cref{s:eqsfromGrass} we will see the same tropical equations by tropicalising equations for $\bbG/\bbP$.

\subsection{Preliminaries}

\subsubsection{Tropical equations}
We define tropical equations here in a bare-bones way. For context and a detailed treatment of tropical algebra, see \cite{MaclaganSturmfels}.

Let $\BB=(\{0,1\},\oplus,\odot)$ be the Boolean semifield,
where $0$ is the zero, $1$ the multiplicative identity, and $1\oplus 1=1$.
(Since we are not concerned here with any richer semifield of scalars,
we eschew the common tropical notation in which instead $\BB=\{\infty,0\}$.)
Let $X$ be a finite set of indeterminates.
Then $\BB[X]$ is the monoid semiring with coefficients in~$\BB$ of the free commutative monoid $\mathbb N^X=\Hom_{\mathrm{Set}}(X,\mathbb N)$ on the generators $X$.
Explicitly, for $a\in\mathbb N^X$, let $X^{\odot a}$ denote the formal monomial 
\raisebox{0pt}[0pt][0pt]{$\bigodot_{x\in X}\underbrace{x\odot\cdots\odot x}_{a(x)\mbox{\scriptsize\ times}}$}. 
Then every element of~$\BB[X]$ is of the form
\[\bigoplus_{a\in A}X^{\odot a}\]
for some finite subset $A\subset\mathbb N^X$.
Addition in $\BB[X]$ corresponds to union of the sets~$A$, and multiplication to Minkowski sum.
The \defn{tropicalisation} map $\trop:\CC[X]\to\BB[X]$
is defined by $\trop(\sum_{a\in A} z_a X^a) = \bigoplus_{a\in A} X^{\odot a}$ 
where $z_a\in\CC$ is nonzero for all $a\in A$.

When we call an element $f\in\BB[X]$ a \defn{tropical equation}, 
we are thinking of it as a condition that may or may not be satisfied by a tuple $p\in\BB^X$.
The tropical equation $f=\bigoplus_{a\in A}X^{\odot a}$ is \defn{satisfied} by~$p$ 
if the number of $a\in A$ such that $p^{\odot a}=1$ -- i.e.\ such that $p(x)=1$ for all $x$ with $a(x)>0$ -- 
is not exactly~1.
If $F\in\CC[X]$ is a complex polynomial such that $F(P)=0$ at the point $P\in\CC^X$
and we let $f=\trop(F)$ and $p\in\BB^X$ be the support of~$P$, i.e.\ $p(x)=1$ if and only if $P(x)\ne0$,
then the tropical equation $f$ is satisfied by~$p$.

We warn the reader that tropicalisation is not a homomorphism of additive semigroups, on account of cancellations in~$\CC[X]$.
When we present tropical exchange equations later in this section,
these will come with lower bounds on the sizes of sets involved.
Comparing to the equations over $\CC[X]$ in \Cref{s:eqsfromGrass}, we see that these bounds arise from ``sporadic'' cancellations of terms when the sets are small, making the tropical equation not agree with the tropicalisation.

Unless otherwise stated, the tropical equations arising in the remainder of this section will have $X=\set{x_S}{S\subseteq[n]}$.
Set systems $M$ on~$[n]$ are in bijection with tuples $\nu_M\in\BB^X$
by the rule that $\nu_M(x_S)=1$ if and only if $S\in M$.
In this way we may make sense of saying that a set system $M$ satisfies a tropical equation $f\in\BB[X]$:
we mean that $\nu_M$ satisfies~$f$.

\subsubsection{Warm up: Matroids (type A)}

As discussed in the Introduction, the tropical equations that characterise matroids have been known implicitly since Rota~\cite{RotaBowdoin}.
It is straightforward to verify that they are equivalent to the symmetric basis exchange property given by Brualdi~\cite{Brualdi1969}.

\paragraph{Strong exchange equations (Type A)}

The \defn{type A strong exchange equations} on $[n]$ are the collection of equations $\cF^{(A)}_{n,k} \coloneqq \set{f_{I,J}^{(A)} }{ I \in \scP_{n,k-1},\, J \in \scP_{n,k+1} , \, |J \setminus I| \geq 3 }$ where
\begin{equation} \label{eq:SEE-A}
f^{(A)}_{I,J} \coloneqq \bigoplus_{i \in J \setminus I} x_{I \cup i} \odot x_{J \setminus i} \in \BB[X] \, . 
\end{equation}

\begin{theorem}\label{t:Astrexch}
    A set system $M \subseteq \scP_{n, k}$ is a matroid if and only if $\nu_M$ satisfies the type $A$ strong exchange equations $\cF^{(A)}_{n,k}$.
\end{theorem}

For each type A minuscule parabolic, the Coxeter strong exchange equations will have a geometric interpretation.
Explicitly, we will be able to determine whether $M$ is a strong Coxeter matroid by checking whether certain `antipodes' are present in the Coxeter matroid polytope $P(M)$.
In type $A$, we are interested in antipodes of the hypersimplex $\cH_{n,k}$ as defined in \eqref{eq:hypersimplex} and its sub-hypersimplices.
Fixing some $I\in \scP_{n,k-1}$ and $J \in \scP_{n,k+1}$, we define the hypersimplex
\[
\cH_{I,J} = \conv\set{e_K}{I \cap J \subseteq K \subseteq I \cup J\, , \, |K| = k} \subseteq \cH_{n,k} \, , \quad e_K := \sum_{i \in K}e_i \, .
\]
Note that $\cH_{I,J}$ is a face of $\cH_{n,k}$.
For each $i \in J \setminus I$, we have that $e_{I \cup i}$ and $e_{J \setminus i}$ are both vertices of $\cH_{I,J}$, and moreover are antipodes of each other in $\cH_{I,J}$: by abuse of notation we will call the pair $(e_{I \cup i}, e_{J \setminus i})$ an \defn{antipode} of $\cH_{I,J}$.

This viewpoint combined with \Cref{t:Astrexch} leads to a characterisation of a matroid $M$ in terms of its polytope $P(M)$.
Namely, the equation $f_{I,J}^{(A)}$ is satisfied by $\nu_M$ if and only if there is not exactly one antipode of the form $(e_{I \cup i}, e_{J \setminus i})$ for $i \in J \setminus I$ present in the polytope $P(M)$.

\subsection{Cubes}
This section deals with $\Delta$-matroids (type $B$) and even $\Delta$-matroids (type $D$), which can be viewed as subpolytopes of cubes.
The corresponding formulations in tropical equations are \Cref{t:Bstrexch} and \Cref{p:strexcheqstrop}.

\subsubsection{$\Delta$-matroids (Type $B$)} \label{ssec:Bstrexch}

\paragraph{Strong exchange equations (Type $B$)}

The \defn{type $B$ strong exchange equations} on $[n]$ are the equations in the collection
$\cF^{(B)}_n \coloneqq \set{f_{I,J}^{(B)}}{I, J \subseteq [n],\, |I\symDiff J| \ge 3}$ where
\begin{align} \label{eq:SEE-B-1}
    f^{(B)}_{I,J} &:= \bigoplus_{i \in I \symDiff J} x_{I \symDiff i} \odot x_{J \symDiff i} & \text{ if } |I \symDiff J| \equiv 0 \mod 2\, \rlap{\text{ and}} \\
    f^{(B)}_{I,J} &:=  x_I \odot x_J \oplus \bigoplus_{i \in I \symDiff J} x_{I \symDiff i} \odot x_{J \symDiff i} & \text{ if } |I \symDiff J| \equiv 1 \mod 2\, \rlap{.} \label{eq:SEE-B-2}
\end{align}

\begin{remark}
While not necessary for any proofs, we give some geometric intuition for these equations in terms of antipodes of cubes.
Recall from \eqref{eq:cubeVertex} the notation $\cubeVertex{I}\in\RR^n$ for the vertex of the cube $Q_n$ indexed by the subset $I\subseteq[n]$.
Given two sets $I,J \subseteq [n]$, we can consider the subcube of $Q_n$
\[
    \cube{I}{J} := \set{\cubeVertex{K}}{I \cap J \subseteq K \subseteq I \cup J} \subseteq Q_n \, .
\]
This notation has the property that $\cubeVertex{I}$ is an antipode of $\cubeVertex{J}$ within $\cube{I}{J}$: by an abuse of notation we will call the pair $(\cubeVertex{I},\cubeVertex{J})$ an \defn{antipode}.
Moreover, for each $S \subseteq I \symDiff J$ we also have that $(\cubeVertex{I\symDiff S},\cubeVertex{J\symDiff S})$ is an antipode of $\cube{I}{J}$.
In particular, the antipodes that neighbour $(\cubeVertex{I},\cubeVertex{J})$ along an antipodal pair of edges are of the form $(\cubeVertex{I\symDiff i},\cubeVertex{J\symDiff i})$ for $i \in I \symDiff J$.

With this viewpoint, the equation \eqref{eq:SEE-B-1} fixes some antipode $(\cubeVertex{I},\cubeVertex{J})$ and is satisfied by $\nu_M$ if and only if there is not exactly one neighbouring antipode present in the polytope $P(M)$.
The equation \eqref{eq:SEE-B-2} similarly fixes some antipode $(\cubeVertex{I},\cubeVertex{J})$ and is satisfied by $\nu_M$ if and only if there is not exactly one of $(\cubeVertex{I},\cubeVertex{J})$ or its neighbouring antipodes present in $M$.
\end{remark}

\begin{example}
    To better understand these equations, we work through them for $n = 3$ and $n = 4$.
    For ease of notation we will let $i_1\ldots i_k$ denote $\{i_1, \dots, i_k\}$, e.g.\ $12$ denotes $\{1, 2\}$.

    Let $n = 3$, where the list of subsets of~$[3]$ is $\emptyset$, $1$, $2$, $3$, $12$, $13$, $23$, $123$.
    The requirement that $\order{I\symDiff J}\ge3$ 
    means that $\cF^{(B)}_3$ contains no equation in which $\order{I \symDiff J} \equiv 0 \mod 2$.
    There are potentially four equations with $\order{I \symDiff J}=3 \equiv 1 \mod 2$.  Two of these are
    \begin{align}
        f^{(B)}_{\emptyset, 123} &= x_{\emptyset} \odot x_{123} \oplus \bigoplus_{i \in 123} x_{\emptyset \symDiff i} \odot x_{123 \symDiff i} = x_{\emptyset} \odot x_{123} \oplus x_{1} \odot x_{23} \oplus x_{2} \odot x_{13} \oplus x_3 \odot x_{12} \label{eq:ex B_3}\\
        f^{(B)}_{1, 23} &= x_{1} \odot x_{23} \oplus \bigoplus_{i \in 123} x_{1 \symDiff i} \odot x_{23 \symDiff i} = x_{1} \odot x_{23} \oplus x_{\emptyset} \odot x_{123} \oplus x_{13} \odot x_{2} \oplus x_{12} \odot x_{3} = f^{(B)}_{\emptyset, 123} \notag
    \end{align}
    We see that they are equal elements of $\BB[X]$, giving us the sum of all antipodes of the $3$-cube.
    The other two equations are equal to this same sum as well.
    So in fact $\order*{\cF^{(B)}_3}=1$.

    Before moving on, we look at examples of equations $f^{(B)}_{I,J}$ \textbf{excluded} from $\cF^{(B)}_3$ because $\order{I\symDiff J}<3$:
    \begin{align*}
        f^{(B)}_{\emptyset, 1} &= x_\emptyset \odot x_1 \oplus \bigoplus_{i \in 1} x_{\emptyset \symDiff i} \odot x_{1 \symDiff i} = x_\emptyset \odot x_1 \oplus x_1 \odot x_\emptyset  \\
        f^{(B)}_{\emptyset, 12} &= \bigoplus_{i \in 12} x_{\emptyset \symDiff i} \odot x_{12 \symDiff i}  
= x_1 \odot x_2 \oplus x_2 \odot x_1 \notag 
    \end{align*}
    As written, it appears these equations would be trivially satisfied: they are just two copies of the same monomial, so that at any~$\nu_M$ both summands or neither would evaluate to~$1\in\BB$, never exactly one.
    However, the definition of satisfaction in~$\BB[X]$ dictates that we first simplify each sum, to a single monomial.  So in fact they will fail to be satisfied whenever $M$ contains $\emptyset$ and~$1$, resp.\ $1$ and~$2$.
    We therefore must exclude these equations: there is no pair of sets whose mere presence in~$M$ certifies that $M$ is not a strong $\Delta$-matroid.
    As we will see in \Cref{s:eqsfromGrass}, the quadratic embedding equations in~$\CC[X]$ that would be expected to tropicalise to the above $f^{(B)}_{I,J}$ in fact cancel to zero.

    When $n = 4$, the set $\cF^{(B)}_4$ contains the $S_4$-orbit of the equation in \eqref{eq:ex B_3}, of size~$4$, and three additional $S_4$-orbits of sizes $1$, $1$, and~$4$, containing respectively:
    \begin{align*}
        f^{(B)}_{\emptyset, 1234} &= \bigoplus_{i \in 1234} x_{\emptyset \symDiff i} \odot x_{1234 \symDiff i} = x_{1} \odot x_{234} \oplus x_{2} \odot x_{134} \oplus x_{3} \odot x_{124} \oplus x_4 \odot x_{123} \\
        &= f^{(B)}_{12, 34}\\
        f^{(B)}_{1, 234} &= \bigoplus_{i \in 1234} x_{1 \symDiff i} \odot x_{234 \symDiff i} = x_{\emptyset} \odot x_{1234} \oplus x_{12} \odot x_{34} \oplus x_{13} \odot x_{24} \oplus x_{14} \odot x_{23}\\
        f^{(B)}_{1, 1234} &= x_{1} \odot x_{1234} \oplus \bigoplus_{i \in 234} x_{1 \symDiff i} \odot x_{1234 \symDiff i} = x_{1} \odot x_{1234} \oplus x_{12} \odot x_{134} \oplus x_{13} \odot x_{124} \oplus x_{14} \odot x_{123}
    \end{align*}
    Altogether we have $\order*{\cF^{(B)}_4}=10$. 
\end{example}

\medskip 

We aim to show that the type $B$ strong exchange equations characterise strong $\Delta$-matroids.
To do this, we need the following two technical lemmas.

\begin{lemma} \label{lem:odd+antipode+exchange}
    Let $M \subseteq [n]$ be a strong $\Delta$-matroid.
    If $I,J \in M$ are such that $|I \symDiff J| \equiv 1 \mod 2$, then there exists $i \in I \symDiff J$ such that $I \symDiff \{i\}, J \symDiff \{i\} \in M$.
\end{lemma}

\begin{proof}
    We prove this by induction on $|I \symDiff J| = 2k + 1$.
    For the base cases of $k=0, 1$, any strong exchange suffices.

    For the inductive step where $k \geq 2$, we first show the weaker statement that there exists some $i \in I \symDiff J$ such that $J \symDiff i \in M$.
    Applying the strong exchange property to $I$ and $J$ for any $a \in I \symDiff J$, there exists $b \in I \symDiff J$ such that $I \symDiff \{a,b\}$ and $J \symDiff \{a,b\}$ are in $M$.
    If $a = b$, we have found such an $i$.
    Otherwise, applying the inductive hypothesis to $I \symDiff \{a,b\}$ and $J$ implies that there exists $i \in I \symDiff J \setminus \{a,b\}$ such that $J \symDiff i \in M$.

    We now prove the full statement.
    Applying the strong exchange property to $I, J$ and this $i \in I \symDiff J$ with $J \symDiff i \in M$, we find there exists $\ell \in I \symDiff J$ such that $I \symDiff \{i,\ell\}$ and $J \symDiff \{i, \ell\} \in M$.
    If $i = \ell$ we are done, so assume that $i \neq \ell$.
    Applying the inductive hypothesis to $I \symDiff \{i,\ell\}$ and $J$, there exists some $j \in I \symDiff J \setminus \{i,\ell\}$ such that $I \symDiff \{i,j,\ell\}$ and $J \symDiff j \in M$.
    We now apply the strong exchange property to $I$ and $I \symDiff \{i,j,\ell\}$ to obtain that either $I \symDiff i, I \symDiff j$ or $I \symDiff \{i,j\} \in M$.
    In the first two cases, we have already seen that $J \symDiff i$ and $J \symDiff j \in M$, so we are done.
    It remains to show the final case where $I \symDiff \{i,j\} \in M$.

    By applying the inductive hypothesis to $I \symDiff \{i,j\}$ and $J$, there exists some $k \in I \symDiff J \setminus\{i,j\}$ such that $I \symDiff \{i,j, k\}$ and $J \symDiff k \in M$.
    We now apply the strong exchange property to $I$ and $I \symDiff \{i,j,k\}$ to obtain that either $I \symDiff i, I \symDiff j$ or $I \symDiff k \in M$.
    The first two cases have already been addressed, and for the final case we also have $J \symDiff k \in M$.
    This completes the proof.
\end{proof}

\begin{lemma} \label{lem:even+antipode+exchange}
    Let $M \subseteq [n]$ be a strong $\Delta$-matroid.
    If $I,J \in M$ are such that $|I \symDiff J| \equiv 0 \mod 2$ and $i \in I \symDiff J$ is such that $I \symDiff \{i\}, J \symDiff \{i\} \in M$, then there exists some $j \in I \symDiff J \setminus \{i\}$ such that $I \symDiff \{j\}, J \symDiff \{j\} \in M$.
\end{lemma}

\begin{proof}
    We prove this via induction on $|I \symDiff J| = 2k$.
    For the base case of $k=1$, we have $I \symDiff J = \{i,j\}$, and so $I \symDiff i = J \symDiff j$ and vice versa.
    Hence the claim trivially holds.

    For the inductive step, applying \Cref{lem:odd+antipode+exchange} to $I$ and $J \symDiff i$ implies there exists some $j \in I \symDiff J \setminus i$ such that $I \symDiff j$ and $J \symDiff \{i,j\}$ are in $M$.
    Applying the strong exchange property for $i \in (I \symDiff i) \symDiff (J \symDiff \{i,j\})$, we obtain that either $J \symDiff j \in M$ or there exists some $k \in I \symDiff J \setminus \{i,j\}$ such that $I \symDiff k$ and $J \symDiff \{j,k\}$ are both in $M$.
    In the former case we are done, so let us assume the latter case.
    Applying the strong exchange property to $J \symDiff i$ and $J \symDiff \{j,k\}$, we have either $J \symDiff j$, $J \symDiff k$ or $J \symDiff\{i,j,k\}$ are in $M$.
    If either of the first two hold then we are done, so assume the final case.

    Consider $I$ and $J \symDiff\{j,k\}$, and observe that we have $I \symDiff i$ and $J \symDiff\{i,j,k\}$ both contained in $M$.
    As $|I \symDiff (J \symDiff\{j,k\})| = 2k -2$, by induction there exists some $\ell \in (I \symDiff J) \setminus \{i,j,k\}$ such that $I \symDiff \ell$ and $J \symDiff \{j,k,\ell\}$ are both in $M$.
    Finally, applying the strong exchange property to $J \symDiff \{j,k,\ell\}$ and $J$ implies one of $J \symDiff j$, $J \symDiff k$ or $J \symDiff \ell$ are in $M$.
    In all cases, we are done.
\end{proof}

\begin{theorem}\label{t:Bstrexch}
    Let $M$ be a set system on $[n]$.
    Then $M$ is a strong $\Delta$-matroid if and only if $\nu_M$ satisfies the type $B$ strong exchange equations $\cF^{(B)}_n$.
\end{theorem}

\begin{proof}
    Suppose $\nu_M$ satisfies $\cF^{(B)}_n$.
    For any $I, J \in M$ and $i \in I \symDiff J$, we show there exists $j \in I \symDiff J$ such that $I\symDiff \{i,j\},\, J \symDiff \{i,j\} \in M$.
    Consider the equation $f^{(B)}_{I',J'}$ where $I' = I \symDiff i$ and $J' = J \symDiff i$, i.e.
    \[
        f^{(B)}_{I',J'} \coloneqq \left(\bigoplus_{k \in I' \symDiff J'} x_{I' \symDiff k} \odot x_{J' \symDiff k} \right) \oplus \underbrace{x_{I'} \odot x_{J'}}_{|I' \symDiff J'| \equiv 1 \mod 2} \, ,
    \]
    where the last term only appears when $|I' \symDiff J'| \equiv 1 \mod 2$.
    As $I \symDiff J = I' \symDiff J'$, we have $i \in I' \symDiff J'$ and so $x_{I \symDiff i \symDiff i} \odot x_{J \symDiff i \symDiff i} = x_I \odot x_J$ is a term of $f^{(B)}_{I',J'}$.
    As $\nu_M(x_I \odot x_J) = 1$ and $\nu_M$ satisfies $f^{(B)}_{I',J'}$\,, there exists at least one other term of $f^{(B)}_{I',J'}$ on which $\nu_M$ evaluates to $1$.
    If there exists $j \in (I' \symDiff J') \setminus i$ such that $\nu_M(x_{I' \symDiff j} \odot x_{J' \symDiff j}) = 1$, then it follows $I \symDiff \{i,j\} ( = I' \symDiff j) $ and $J \symDiff \{i.j\} ( = J' \symDiff j)$ are in $M$ and so $M$ satisfies the strong exchange property.
    If $\nu_M(x_{I'} \odot x_{J'}) = 1$, then it follows that $M$ satisfies the strong exchange property with $i = j$. 

    Conversely, suppose $M$ is a strong $\Delta$-matroid.
    Fix some $I, J \subseteq [n]$. We show that $\nu_M$ satisfies $f^{(B)}_{I,J}$.
    Suppose there exists some $i \in I \symDiff J$ such that $I' \coloneqq I \symDiff i$ and $J' \coloneqq J \symDiff i$ are contained in $M$.
    This implies that $\nu_M(x_{I \symDiff i} \odot x_{J \symDiff i}) = 1$.
    By the strong exchange property, there exists some $j \in I' \symDiff J' = I \symDiff J$ such that $I'' \coloneqq I' \symDiff \{i,j\}$ and $J'' \coloneqq J' \symDiff \{i,j\}$ are both elements of $M$.
    If $j \neq i$, then $I'' = I \symDiff j$ and $J'' = J \symDiff j$ and hence $\nu_M(x_{I \symDiff j} \odot x_{J \symDiff j}) = 1$, proving $\nu_M$ satisfies $f^{(B)}_{I,J}$.
    If $j = i$, then $I'' = I$ and $J'' = J$ and hence $\nu_M(x_{I} \odot x_{J}) = 1$.
    If $I \symDiff J \equiv 1 \mod 2$, this proves that $\nu_M$ satisfies $f^{(B)}_{I,J}$.
    If $I \symDiff J \equiv 0 \mod 2$, \Cref{lem:even+antipode+exchange} implies there exists some distinct $k \in (I \symDiff J) - j$ such that $I \symDiff k$ and $J \symDiff k$ are elements of $M$.
    This implies there is some other term of $f^{(B)}_{I,J}$ with $\nu_M(x_{I \symDiff k} \odot x_{J \symDiff k}) = 1$, hence $\nu_M$ satisfies $f^{(B)}_{I,J}$.

    Finally suppose there exists no $i \in I \symDiff J$ such that $I' \coloneqq I \symDiff i$ and $J' \coloneqq J \symDiff i$ are contained in $M$.
    If $|I\symDiff J| \equiv 0 \mod 2$, these account for all of the terms of $f^{(B)}_{I,J}$ and so $\nu_M$ trivially satisfies it.
    If $|I\symDiff J| \equiv 1 \mod 2$, \Cref{lem:odd+antipode+exchange} implies that we can't have both $I$ and $J$ in $M$.
    This accounts for the final additional term of $f^{(B)}_{I,J}$, hence $\nu_M$ trivially satisfies $f^{(B)}_{I,J}$.
\end{proof}

\subsubsection{Even $\Delta$-matroids (Type $D$)}

We next consider even $\Delta$-matroids.
As these are a special class of $\Delta$-matroids, our alternative characterisation follows fairly straightforwardly from the analysis in \Cref{ssec:Bstrexch}.

\paragraph{Strong exchange equations (Type $D$)}

The \defn{type $D$ strong exchange equations} on $[n]$ are the collection of equations $\cF^{(D)}_n \coloneqq \set{f_{I,J}^{(D)}}{I, J \subseteq [n] , \, |I \symDiff J| \ge 4,\, |I \symDiff J| \equiv 0 \bmod 2}$ where
\begin{equation} \label{eq:SEE-D}
f^{(D)}_{I,J} \coloneqq \bigoplus_{i \in I \symDiff J} x_{I \symDiff i} \odot x_{J \symDiff i} \in \BB[X] \, . 
\end{equation}
Note that these are the subset of the type $B$ strong exchange equations where $I,J$ have the same parity.

The following result characterises even $\Delta$-matroids as the even set systems satisfying the type $D$ strong exchange equations.

\begin{proposition}\label{p:strexcheqstrop}
    Let $M$ be an even set system on $[n]$.
    Then $M$ is an even $\Delta$-matroid if and only if $\nu_M$ satisfies the type $D$ strong exchange equations $\cF^{(D)}_n$.
\end{proposition}

\begin{proof}[Proof of {\Cref{p:strexcheqstrop}}]
    Let $M$ be an even set system.
    By definition, each $I \in M$ has the same parity, and so for any $I, J \in M$ we have $\order{I \symDiff J} \equiv 0 \mod 2$.
    In particular, $f_{I,J}^{(B)}$ are trivially satisfied by $\nu_M$ for $\order{I \symDiff J} \equiv 1 \mod 2$, and hence $\nu_M$ satisfies $\cF^{(B)}_n$ if and only if it satisfies $\cF^{(D)}_n$.
    By \Cref{t:Bstrexch} and using the fact that $M$ is even, $M$ is a strong even $\Delta$-matroid if and only if $\nu_M$ satisfies $\cF^{(D)}_n$.
    As every even $\Delta$-matroid is strong by \Cref{t:evenIsStrong}, we have our desired result.
\end{proof}

\begin{remark}
    The type $D$ strong exchange equations are a subset of the tropical \defn{Wick relations} \cite[Definition 4.2]{Rincon}, where one defines $f^{(D)}_{I,J}$ for all $I, J \subseteq [n]$ and not just those of the same parity.
    Adding in these extra equations allows one to drop the assumption from \Cref{p:strexcheqstrop} that $M$ is already an even set system.
\end{remark}

\begin{remark}
    A corollary of \Cref{t:Bstrexch} and \Cref{p:strexcheqstrop} is that if $M \subseteq [n]$ is a strong $\Delta$-matroid, then its restriction to the sets of odd and even parity are both even $\Delta$-matroids.
    Note that the converse is not true: \Cref{ex:delta+not+strong} gives an example of a non-strong $\Delta$-matroid whose odd and even restrictions are both even $\Delta$-matroids.

    To see a $\Delta$-matroid whose odd and even restrictions are not both even $\Delta$-matroids, consider the set system $M$ containing all subsets of $[4]$ except those in~$\scP_{4, 2}$.
    It is routine to check that this is a $\Delta$-matroid, but not strong.
    Furthermore, its even restriction is $\{\emptyset, 1234\}$, and hence not an even $\Delta$-matroid.
\end{remark}
\begin{remark}\label{rem:strongB_nisStrongD_n+1}
 One can construct a map from the vertices of the $n$-cube $Q_n$ to the even (or odd) demicube $\cD_{n+1}^+$ ($\cD_{n+1}^-$): the map to the even demicube is
 \[ 
 I \mapsto \begin{cases}
    I &  |I| \text{ even,} \\
    I \cup \{n+1\} & |I| \text{ odd}
 \end{cases}
 \]
 with the odd version defined similarly.
 \cite[Theorem 3.2]{Murota2021} shows that a set system is a type $B_n$ strong Coxeter matroid if and only if its image under this map is a type $D_{n+1}$ (strong) Coxeter matroid.
 This can also be deduced from the strong exchange equations: under this map, the type $B$ strong exchange equations \eqref{eq:SEE-B-1} and \eqref{eq:SEE-B-2} map to the type $D$ strong exchange equations \eqref{eq:SEE-D}.
\end{remark}

\subsection{Cross polytopes}
In this section, we consider the remaining minuscule parabolic subgroups in types $C$ and $D$.
In both cases, the resulting orbit polytopes will be \defn{cross polytopes}.
Given an orthogonal basis $\{e_1, \dots, e_n\}$ for a Euclidean vector space $\vectorSpace$, the $n$-cross polytope is the convex hull of the basis vectors and their negatives:
\[
\crossPoly_n = \conv(\pm e_i \st i \in [n]) \, .
\]
The edges of $\crossPoly_n$ are all parallel to $\pm e_i \pm e_j$ for some $i \neq j$.
Moreover, every pair of vertices in $\crossPoly_n$ are adjacent aside from pairs of the form $\{e_i, -e_i\}$: we call these the \defn{antipodes} of $\crossPoly_n$.

\subsubsection{$D_n$ cross polytope} \label{ssec:Dn-cross-poly}

Fix the root system $D_n$ and its corresponding Weyl group $W = W(D_n)$.
The classification of minuscule weights gives us a single outstanding case, where $J = S \setminus \{s_1\}$.
The standard parabolic subgroup generated by $J$ is $W(D_{n-1})$, hence Coxeter matroids of this type are subsets of $W^J = W(D_n)/W(D_{n-1})$.

Fix the choice of simple roots $\alpha_i = e_i - e_{i+1}$ for $1 \leq i \leq n-1$ and $\alpha_n = e_{n-1} + e_n$.
Under this choice of roots, we can choose $\omega_J = -e_1$ to satisfy \eqref{eq:omega_J}.
Then the ambient polytope $\ambientPoly$ is precisely the $n$-cross polytope $\crossPoly_n = \conv(\pm e_i \st i \in [n])$.

To construct tropical equations characterising Coxeter matroids in this type, we set the variables to be $X = \{x_i\}_{i=1}^n \cup \{x_{-i}\}_{i=1}^n$.
Subsets $M \subseteq W^J$ are is bijection with tuples $\nu_M \in \BB^X$ by the rule that $\nu_M(x_{\pm i}) = 1$ if and only if $A \in M$ for $A\cdot \omega_J = \pm e_{i}$.
We derive a single tropical equation from $\crossPoly_n$ by taking the tropical sum of its antipodes:
\begin{equation} \label{eq:SEE-D*}
    f^{(D)}_{\crossPoly_n} \coloneqq \bigoplus_{i=1}^n x_{i} \odot x_{-i} \in \BB[X] \, . 
\end{equation}

\begin{lemma}\label{l:strexchDncross}
Let $M \subseteq W^J$ where $W = W(D_n)$ and $J = S \setminus \{s_1\}$.
Then $M$ is a strong Coxeter matroid if and only if $\nu_M$ satisfies the tropical equation $f^{(D)}_{\crossPoly_n}$.
\end{lemma}

\begin{proof}
Let $M \subseteq W^J$ be a strong Coxeter matroid and assume $P(M)$ contains the antipode $(e_i, -e_i)$.
By the strong exchange property, there exists some hyperplane $H$ in $W(D_n)$ separating $e_i$ and $-e_i$ such that $s_H(e_i)$ and $s_H(-e_i)$ are also in $P(M)$.
Considering the roots of $D_n$, this hyperplane is orthogonal to $e_i \pm e_j$ for some $j \neq i$, hence $s_H(e_i) = e_j$ and  $s_H(-e_i) = -e_j$ forms another antipode of $\crossPoly_n$.
It immediately follows that $\nu_M$ satisfies the tropical equation $f^{(D)}_{\crossPoly_n}$.

Conversely, suppose $\nu_M$ satisfies $f^{(D)}_{\crossPoly_n}$ and let $A, B \in M$.
If $(A\cdot \omega_J, B\cdot \omega_J)$ forms an antipode $(e_i, -e_i)$, the tropical equation $f^{(D)}_{\crossPoly_n}$ implies there exists another antipode $(e_j, -e_j)$ in $P(M)$.
This can be obtained from $(e_i, -e_i)$ via a reflection through the hyperplane orthogonal to $e_i - e_j$.
If $(A\cdot \omega_J, B\cdot \omega_J)$ do not form an antipode, they are adjacent vertices in $\crossPoly_n$ and reflecting them through the hyperplane orthogonal to $A\cdot \omega_J - B\cdot \omega_J$ exchanges them.
As such, $M$ satisfies the strong exchange property.
\end{proof}

\begin{corollary}
    Let $M \subseteq W^J$ be a Coxeter matroid where $W = W(D_n)$ and $J = S \setminus \{s_1\}$.
    Then $M$ is a strong Coxeter matroid.
\end{corollary}
\begin{proof}
    Let $M \subseteq W^J$ and assume that $\nu_M$ does not satisfy $f^{(D)}_{\crossPoly_n}$.
    Hence $P(M)$ contains some antipode, say $(e_1,-e_1)$, and contains at most one of the vertices $e_i$ and $-e_i$ for each $2 \leq i \leq n$.
    Define the vector $a \in \RR^n$ where $a_1 = 0$ and
    \[
        a_i = \begin{cases}
            -1 & e_i \in P(M) \\
            +1 & e_i \notin P(M)
        \end{cases}
        \quad\text{ for all } 2 \leq i \leq n \, .
    \]
    By construction, the linear functional $\ell(x) = a \cdot x$ takes the value $0$ on the antipode $(e_1, -e_1)$ and $-1$ on all other vertices of $P(M)$.
    As such, $(e_1, -e_1)$ is an edge of $P(M)$ that is not parallel to a root direction, hence $M$ is not a Coxeter matroid.
\end{proof}

\subsubsection{$C_n$ cross polytope}

Fix the root system $C_n$ and its corresponding Weyl group $W = W(C_n)$.
The first fundamental weight is minuscule.
In this case, we have $J = S\setminus\{ s_1\}$ generates the parabolic Weyl group $W(C_{n-1})$, hence Coxeter matroids of this type are subsets of $W^J = W(C_n)/W(C_{n-1})$.

Fix the choice of simple roots $\alpha = e_i - e_{i+1}$ for $1 \leq i \leq n-1$ and $\alpha_n = 2e_n$.
Under this choice of roots, we can choose $\omega_J = -e_1$ again to satisfy \eqref{eq:omega_J}.
Then the ambient polytope $\ambientPoly$ is also the $n$-cross polytope $\crossPoly_n = \conv(\pm e_i \st i \in [n])$.

Despite having the same ambient polytope, it is worth noting that the $W(C_n)$ setting is slightly different from the $W(D_n)$ cross polytope.
This is due to the fact that the Weyl group $W(C_n)$ and its set of roots $\rootSystem(C_n)$ is strictly bigger than $W(D_n)$ and its roots $\rootSystem(D_n)$.
In particular, the vertices $e_i$ and $-e_i$ can be exchanged by a reflection through the hyperplane orthogonal to $2e_i \in \rootSystem(C_n)$, hence antipodes of $\crossPoly_n$ are no obstruction unlike for $D_n$.
This is displayed in Figure~\ref{fig:crossPoly_reflection}.
In fact, we find every subset of $W(C_n)^J$ is a strong Coxeter matroid and hence have no tropical equations to impose.

\begin{lemma}\label{l:strexchCncross}
    Let $M \subseteq W^J$ where $W = W(C_n)$ and $J = S \setminus \{s_1\}$.
    Then $M$ is a strong Coxeter matroid.
\end{lemma}

\begin{proof}
    Let $M \subseteq W^J$ and take any pair of cosets $A,B \in M$.
    If $(A\cdot \omega_J, B\cdot \omega_J)$ forms an antipode $(e_i,-e_i)$ in $\crossPoly_n$, we can reflect through the hyperplane orthogonal to $2e_i$ to exchange them.
    The case where $(A\cdot \omega_J, B\cdot \omega_J)$ do not form an antipode is identical to \Cref{l:strexchDncross}.
\end{proof}

\begin{figure}[hbt]
    \centering
    \includegraphics[width=0.5\linewidth]{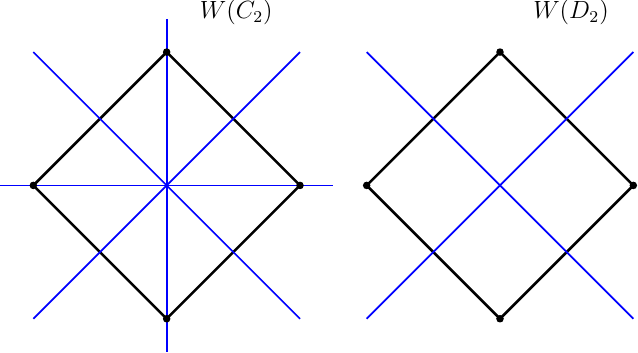}
    \caption{The cross polytope $\crossPoly_2$ and the hyperplanes associated to $W(C_2)$ and $W(D_2)$.
    In $W(C_2)$, we can exchange vertices of an antipode via a reflection, while no such reflection exists in $W(D_2)$.}
    \label{fig:crossPoly_reflection} 
\end{figure}

\begin{remark}
A trivial corollary of \Cref{l:strexchCncross} is that for $W = W(C_n)$ and $J = S \setminus \{s_1\}$, every Coxeter matroid $M \subseteq W^J$ is also a strong Coxeter matroid.
\end{remark}

\subsection{The exceptional Weyl groups}
\label{ssec:the-exceptional-Weyl-groups}
We will consider the $E_6$ and $E_7$ cases in this subsection.
We will give an explicit description of these root systems as subsystems of $E_8$.
Fix $\vectorSpace \cong \RR^8$ with orthonormal basis $\{e_1, \dots, e_8\}$.
We write the root system of $E_8$ as
\begin{align*}
\rootSystem(E_8) 
&= \set{\pm e_i \pm e_j }{ i \neq j} \cup \set{ \tshalf\left(\pm e_1 \pm \cdots \pm e_8\right) }{ \text{even number of minuses }} \, .
\end{align*}
Let $\rho = \half\left(\sum_{i=1}^8 e_i\right)$.
We can write $\rootSystem(E_7)$ as the subsystem of roots orthogonal with $\rho$, i.e.
\begin{align}
    \rootSystem(E_7) &= \set{ \alpha \in \rootSystem(E_8) }{ \scaledInnerProd{\alpha}{\rho} = 0} \nonumber \\ 
    &= \set{\alpha_{ij} \coloneqq e_i - e_j }{ 1 \leq i,j \leq 8,  i \neq j} \cup \set{ \beta_{I} \coloneqq \half\left(\sum_{i\in I}e_i - \sum_{j\in [8]\setminus I}e_j\right) }{ I \in \binom{[8]}{4} } \, . \label{eq:E7}
\end{align}
We denote these two families of roots by $\alpha_{ij}$ for $i\neq j$ and $\beta_I$ for $I \in \binom{[n]}{8}$.
Using this notation, we obtain a combinatorial description of $\rootSystem(E_6)$ as the subsystem of $\rootSystem(E_7)$ of roots orthogonal with $e_7 + e_8$:
\begin{align}
\rootSystem(E_6) &= \set{ \alpha \in \rootSystem(E_7) }{ \scaledInnerProd{\alpha}{e_7+e_8} = 0} \nonumber \\
&= \set{\alpha_{ij} \in \rootSystem(E_7) }{ |\{i,j\} \cap \{7,8\}| \neq 1 } \cup \set{ \beta_I \in \rootSystem(E_7) }{ |I \cap \{7,8\}| = 1} \, . \label{eq:E6}
\end{align}

The previous construction is slightly different than the standard construction of Bourbaki in~\cite{Bourbaki2007b} where $E_7$ is the subsystem of $\rootSystem(E_8)$ orthogonal to $e_7 + e_8$ and where $E_6$ is the subsystem of $\rootSystem(E_7)$ orthogonal to $e_6 + e_8$.
The construction we use is due to Kac \cite{kac+lecture+notes}, which we have chosen as the indexing works better with our polytopes.

\subsubsection{$E_6$ minuscule case}

Fix the root system $E_6$ and its corresponding Weyl group $W = W(E_6)$.
The classification of minuscule weights, \Cref{eq:minuscule}, gives us two cases: $J= S \setminus \{s_1\}$ and $J= S \setminus \{s_6\}$.
From the $E_6$ Dynkin diagram, we see that $\maximalParabolic{1} \cong \maximalParabolic{6} \cong D_5$ and moreover that $\maximalCosets{1} \cong \maximalCosets{6} \cong W(E_6)/W(D_5)$ as sets with a $W$-action, hence we do not need to distinguish between the cases. 

The ambient polytope $\ambientPoly$ associated to $W^J$ is the polytope $2_{21}$.
Following Coxeter~\cite{Coxeter1940}, this is the convex hull of the following 27 points in $\vectorSpace$:
\begin{align}\label{eq:explicit+2_21}
a_i = 2e_i + 2e_7 \, , \; b_i = 2e_i + 2e_8 \quad\text{ for } 1 \leq i \leq 6 \, , \quad
c_{ij} = \left(\sum_{k=1}^8 e_k\right) - 2e_i - 2e_j \quad \text{ for } 1 \leq i < j \leq 6 \, .
\end{align}
Alternatively, it is straightforward to verify one could fix $w_J$ equal to any of these points and take the convex hull of its $W(E_6)$ orbit.
It is a 6-dimensional polytope due to the relations $x_1 + \cdots + x_6 = x_7 + x_8 = 2$.

We will give a number of facts about $2_{21}$ without proof: these either can be found in the articles of Coxeter~\cite{Coxeter1940,Coxeter1988} or are straightforward to manually check using this explicit representation.
Recall from \Cref{rem:combinatorial+equivalence} that any other choice of $\ambientPoly$ is combinatorially equivalent and hence has the same structure.

The 1-skeleton of $2_{21}$ is the Sch\"afli graph, in which each vertex is adjacent to sixteen other vertices, and at distance two from the remaining ten vertices.
Its facets are $72$ 5-simplices and $27$ cross polytopes $\crossPoly_5$, the latter of which are in one-to-one correspondence with the vertices of $2_{21}$.
Explicitly, each cross polytope is the convex hull of vertices non-adjacent to some fixed vertex of $2_{21}$, hence for each $A \in W^J$ we have a cross polytope
\[
\crossPoly(A) \coloneqq \conv(B\cdot \omega_J \st A\cdot \omega_J, B\cdot \omega_J \text{ distinct and non-adjacent}) \subseteq \vectorSpace \, .
\]
Furthermore, each $\crossPoly_5$ is the orbit polytope of a vertex in $2_{21}$ 
under a conjugate of the standard $W(D_5)$ subgroup of~$W(E_6)$, hence we can speak of antipodes within them following the definition in \Cref{ssec:Dn-cross-poly}.
In the explicit coordinates \eqref{eq:explicit+2_21}, these cross polytopes are:
\begin{align*}
    \crossPoly(a_i) &= \conv\left(\set{b_j }{ j \neq i} \cup \set{c_{ij} }{ j \neq i}\right)  \\
    \crossPoly(b_i) &= \conv\left(\set{a_j }{ j \neq i} \cup \set{c_{ij} }{ j \neq i}\right)  \\
    \crossPoly(c_{ij}) &= \conv\left(\{a_i,a_j,b_i,b_j\} \cup \set{c_{k\ell} }{ k,\ell \neq i,j}\right)\, .
\end{align*}
where the antipodes of $\crossPoly(a_i)$ are the sets given by $\set{b_j, c_{ij}}{j \neq i}$, the antipodes of $\crossPoly(b_i)$ are the sets given by $\set{a_j, c_{ij}}{j \neq i}$, and the antipodes of $\crossPoly(c_{ij})$ are the sets given by $\{a_i, b_j\}, \{a_j, b_i\}$, and $\{c_{k\ell},c_{k'\ell'} \}$ for all $\{k,k',\ell,\ell'\} = [6] \setminus \{ij\}$.

\begin{figure}
\centering
\includegraphics[scale=1]{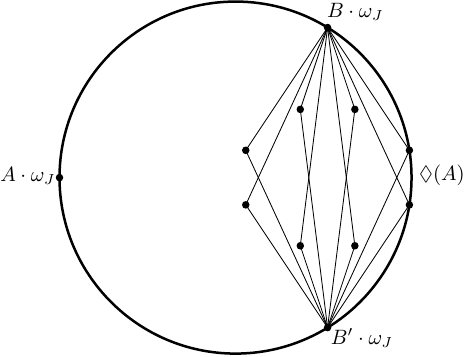}
\caption{A schematic of $2_{21}$. Given some fixed vertex $A\cdot\omega_J$, the cross polytope $\crossPoly(A)$ is the collection of vertices non-adjacent to it. One such antipode of $\crossPoly(A)$ is the pair of non-adjacent points $B\cdot\omega_J,B'\cdot\omega_J$. The remaining 16 vertices not pictured are adjacent to $A\cdot\omega_J$.}
\label{fig:2_21+polytope}
\end{figure}

To construct tropical equations characterising Coxeter matroids of $W(E_6)^J$, we set the variables $X = \set{x_A}{A \in W(E_6)^J}$ to be indexed by cosets in $W(E_6)^J$, or equivalently vertices of $2_{21}$.
As such, subsets $M \subseteq W^J$ are naturally in bijection with tuples $\nu_M \in \BB^X$ by the rule that $\nu_M(x_{A}) = 1$ if and only if $A \in M$.
We obtain a system of 27 tropical equations by tropically summing variables indexed by the antipodes of the 27 $D_5$ cross polytopes.

\paragraph{Strong exchange equations (Type $E_6$)}

The \defn{type $E_6$ strong exchange equations} are the equations in the collection
$\cF^{(E)}_6 \coloneqq \set{f_A }{ A \in W^J}$ where
\begin{align} \label{eq:SEE-E6}
f_A = \bigoplus_{\cA} x_B \odot x_{B'} \, , \quad \cA = \set{(B,B')}{(B\cdot\omega_J,B' \cdot \omega_J) \text{ antipode in } \crossPoly(A)} \, .
\end{align}
In the explicit coordinates \eqref{eq:explicit+2_21}, this is the collection of 27 equations:
\begin{align}\label{eq:strongexEqE6}
f_{a_i} &= \bigoplus_{j\neq i} x_{b_j} \odot x_{c_{ij}} \\
f_{b_i} &= \bigoplus_{j\neq i} x_{a_j} \odot x_{c_{ij}} \\
f_{c_{ij}} &= x_{a_i} \odot x_{b_j} \oplus x_{a_j}\odot x_{b_i} \oplus x_{c_{k\ell}} \odot x_{c_{k'\ell'}} \oplus x_{c_{k\ell'}} \odot x_{c_{k'\ell}} \oplus x_{c_{kk'}} \odot x_{c_{\ell\ell'}} 
\end{align}

Our main result of this subsection is that the type $E_6$ strong exchange equations $\cF^{(E)}_6$ characterise strong Coxeter matroids of $W(E_6)^J$.

\begin{proposition} \label{prop:E_6+matroid}
Let $M \subseteq W^J$ where $W = W(E_6)$ and $J = \maximalRoots{1}$ or $\maximalRoots{6}$.
Then $M$ is a strong Coxeter matroid if and only if $\nu_M$ satisfies the type $E_6$ strong exchange equations $\cF^{(E)}_6$.
    \end{proposition} 
    \begin{proof}
    Suppose $\nu_M$ satisfies $\cF_6^{(E)}$.
    We show that for any $A,B \in M$, there exists a hyperplane $H$ separating $A, B$ such that such that $s_H A, s_H B$ are also in $M$.
    If $A\cdot\omega_J,B\cdot\omega_J$ are adjacent in $2_{21}$, then we can take $H$ to be the hyperplane orthogonal to the root direction $A\cdot\omega_J - B\cdot\omega_J$. Hence we assume they are not adjacent.
    Then there exists some cross polytope $\crossPoly$ in which $\{A\cdot\omega_J, B\cdot\omega_J\}$ is an antipode.
    Moreover as $\nu_M$ satisfies $\cF_6^{(E)}$, there exists another antipode $\{A'\cdot\omega_J, B'\cdot\omega_J\}$ of $\crossPoly$ such that $A',B' \in M$.
    As $\crossPoly$ is the orbit polytope of $W(D_5)$, we can restrict to $D_5$ and use the same analysis as in \Cref{l:strexchDncross} to show there exists a hyperplane $H$ separating $A,B$ such that $s_H A = A'$ and $s_H B = B'$.

Conversely, suppose $M$ is a strong $E_6$-matroid and fix some $D_5$ cross polytope $\crossPoly$.
As $\crossPoly$ is a facet of $2_{21}$, the restriction $F \coloneqq P(M) \cap \crossPoly$ is a face of $P(M)$.
Suppose both $P(M)$ and $\crossPoly$ contain the antipode $\{A\cdot\omega_J,B\cdot\omega_J\}$. Then $\{A\cdot\omega_J,B\cdot\omega_J\}$ is contained in the face $F$.
By the strong exchange property, there exists some separating hyperplane $H$ such that $\{(s_H A)\cdot \omega_J, (s_H B)\cdot \omega_J \}$ are also points in $P(M)$.
Applying \Cref{lem:face_reflection}, we see that $\{(s_H A)\cdot \omega_J, (s_H B)\cdot \omega_J \}$ must also be contained in~$F$, and is another antipode of $\crossPoly$.
Moreover, they must be distinct as there are no roots parallel to antipodes of $\crossPoly$.
\end{proof}

\subsubsection{$E_7$ minuscule case}

Fix the root system $E_7$ and its corresponding Weyl group $W = W(E_7)$.
The classification of minuscule weights, \Cref{eq:minuscule}, gives us only one case, that where $J= S \setminus \{s_7\}$.
As such, Coxeter matroids of this type are subsets $M \subseteq W^J = W(E_7)^{S \setminus \{s_7\}}$.

The ambient polytope $\ambientPoly$ associated to $W^J$ is the polytope $3_{21}$.
Following Coxeter~\cite{Coxeter1988}, one such realisation of $3_{21}$ is as the convex hull of the following 56 points in $\vectorSpace$:
\begin{align} \label{eq:explicit+3_21}
a_{ij} = 4 e_i + 4 e_j - \sum_{k=1}^8 e_k \quad \text{ and } -a_{ij} \quad \text{ for } 1 \leq i < j \leq 8 \, .    
\end{align}
Alternatively, one can fix $\omega_J$ equal to any of these points and take the convex hull of its $W(E_7)$ orbits.
It is a $7$-dimensional polytope as every point has coordinate sum equal to zero.

As with $2_{21}$, we will give a number of facts about $3_{21}$ without proof: these can be found in~\cite{Coxeter1988} or manually checked using this explicit representation.

The 1-skeleton of $3_{21}$ is the Gosset graph, whose properties are listed in~\cite[Example 2]{Haemers1996}.
In particular, every vertex has $27$ neighbours, a further $27$ vertices at graph distance two away and one final vertex at graph distance three away.
We call this final vertex the \defn{antipode} of $A\cdot \omega_J$ in $3_{21}$, and denote it $A\cdot(-\omega_J)$.
In the explicit coordinates from~\eqref{eq:explicit+3_21}, the distance between vertices $v,w$ is given by $d(v,w)= (v,w)/8$, hence the antipode of $a_{ij}$ is $-a_{ij}$.

The facets of $3_{21}$ are 576 6-simplices and $126$ 6-cross polytopes $\crossPoly_6$.
Choosing two vertices $A\cdot \omega_J, B \cdot \omega_J$ that are distance two apart determines a unique $\crossPoly_6$ that contains both.
The remaining ten vertices of this $\crossPoly_6$ are those adjacent to both $A\cdot \omega_J$ and $B \cdot \omega_J$: a schematic of this structure is given in \Cref{fig:3_21+polytope}.
As with the $5$-cross polytopes of $2_{21}$, each $\crossPoly_6$ is the orbit polytope of a vertex in $3_{21}$ under a conjugate of the standard $W(D_6)$ subgroup of $W(E_7)$, hence we can also speak of antipodes within them as in \Cref{ssec:Dn-cross-poly}.

\begin{figure}
    \centering
    \includegraphics[width=0.5\linewidth]{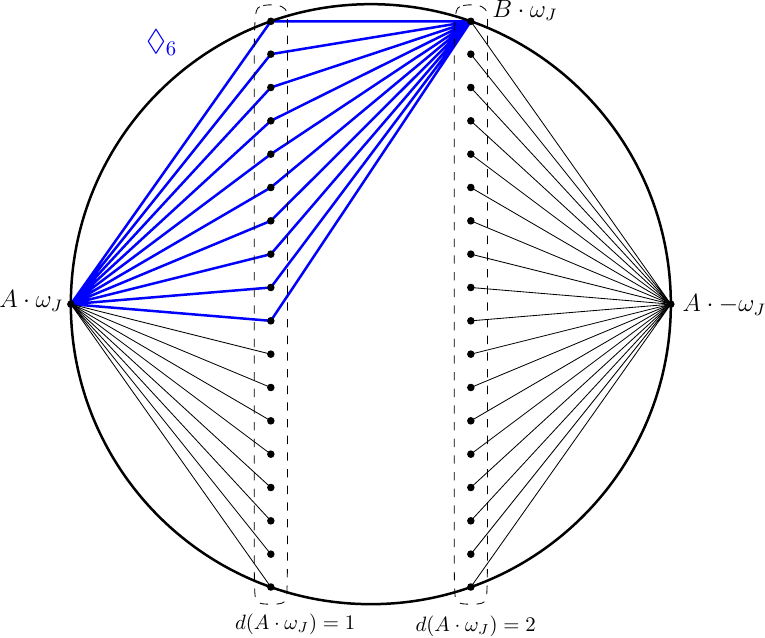}
    \caption{A schematic of $3_{21}$. Given some fixed vertex $A\cdot \omega_J$, there are 27 vertices distance one away, 27 vertices distance two away and a unique vertex distance three away, the antipode $A \cdot -\omega_J$.
    Fixing some vertex $B\cdot \omega_J$ distance two away determines a $6$-cross polytope $\lozenge_6$ whose remaining vertices are those ten adjacent to both $A\cdot \omega_J$ and $B\cdot \omega_J$.}
    \label{fig:3_21+polytope}
\end{figure}

To construct tropical equations characterising Coxeter matroids of $W(E_7)^J$, we set the variables $X = \set{x_A}{A \in W(E_7)^J}$ to be indexed by cosets in $W(E_7)^J$, or equivalently vertices of $3_{21}$.
As such, subsets $M \subseteq W^J$ are naturally in bijection with tuples $\nu_M \in \BB^X$ by the rule that $\nu_M(x_{A}) = 1$ if and only if $A \in M$.
We obtain a system of 126 tropical equations by tropically summing variables indexed by the antipodes of the 126 $D_6$ cross polytopes.
We also obtain a single extra equation by tropically summing variables indexed by the antipodes of $3_{21}$.

\paragraph{Strong exchange equations (Type $E_7$)}

The \defn{type $E_7$ strong exchange equations} are the equations in the collection
$\cF^{(E)}_7 \coloneqq \set{f_\crossPoly }{ \crossPoly \text{ facet of }3_{21} } \cup \{f_{3_{21}}\}$ where
\begin{align} \label{eq:SEE-E7}
    f_\crossPoly &= \bigoplus_{\cA(\crossPoly)} x_B \odot x_{B'} \, , & \cA(\crossPoly) &= \set{(B,B')}{(B\cdot\omega_J,B' \cdot \omega_J) \text{ antipode in } \crossPoly} \, , \\
    f_{3_{21}} &= \bigoplus_{\cA(3_{21})} x_B \odot x_{B'} \, , & \cA(3_{21}) &= \set{(B,B')}{(B\cdot\omega_J,B' \cdot \omega_J) \text{ antipode in } 3_{21}} \, .
\end{align}

    \begin{proposition} \label{prop:E_7+matroid}
Let $M \subseteq W^J$ where $W = W(E_7)$ and $J = \maximalRoots{7}$.
Then $M$ is a strong Coxeter matroid if and only if $\nu_M$ satisfies the type $E_7$ strong exchange equations $\cF^{(E)}_7$.
    \end{proposition} 

    \begin{proof}
    Suppose $\nu_M$ satisfies $\cF_7^{(E)}$.
    We show that for any $A,B \in M$, there exists a hyperplane $H$ separating $A, B$ such that such that $s_H A, s_H B$ are also in $M$.
    If $A\cdot\omega_J,B\cdot\omega_J$ are adjacent in $3_{21}$, then we can take $H$ to be the hyperplane orthogonal to the root direction $A\cdot\omega_J - B\cdot\omega_J$. Hence we assume they are not adjacent.
    If they are distance two away, then there exists some cross polytope $\crossPoly$ such that $\{A\cdot\omega_J, B\cdot\omega_J\}$ is an antipode.
    As $\nu_M$ satisfies $\cF_7^{(E)}$, there exists another antipode $\{A'\cdot\omega_J, B'\cdot\omega_J\}$ of $\crossPoly$ such that $A',B' \in M$.
    As $\crossPoly$ is the orbit polytope of $W(D_6)$, we can restrict to $D_6$ and use the same analysis as in \Cref{l:strexchDncross} to show there exists a hyperplane $H$ separating $A,B$ such that $s_H A = A'$ and $s_H B = B'$.
    If they are distance three away, then $\{A\cdot\omega_J, B\cdot\omega_J\}$ forms an antipode in $3_{21}$.
    As $\nu_M$ satisfies $\cF_7^{(E)}$, there exists another antipode $\{A'\cdot\omega_J, B'\cdot\omega_J\}$ of $3_{21}$ such that $A',B' \in M$.
    At most one of $\{A'\cdot\omega_J, B'\cdot\omega_J\}$ can be distance two away from $A\cdot\omega_J$, similarly for $B\cdot\omega_J$, hence we can assume without loss of generality that $A\cdot\omega_J, A'\cdot\omega_J$ and are $B\cdot\omega_J, B'\cdot\omega_J$ are adjacent.
    Moreover the edges between them are parallel to the same root, hence reflection through the orthogonal hyperplane suffices.

Conversely, suppose $M$ is a strong Coxeter matroid and fix some $D_6$ cross polytope $\crossPoly$: we first show $\nu_M$ satisfies $f_\crossPoly$.
As $\crossPoly$ is a facet of $3_{21}$, the restriction $F \coloneqq P(M) \cap \crossPoly$ is a face of $P(M)$.
Suppose both $P(M)$ and $\crossPoly$ contain the antipode $\{A\cdot\omega_J,B\cdot\omega_J\}$, then $\{A\cdot\omega_J,B\cdot\omega_J\}$ is contained in the face $F$.
By the strong exchange property, there exists some separating hyperplane $H$ such that $\{(s_H A)\cdot \omega_J, (s_H B)\cdot \omega_J \}$ are also points in $P(M)$.
Applying \Cref{lem:face_reflection}, we see that $\{(s_H A)\cdot \omega_J, (s_H B)\cdot \omega_J \}$ must also be contained in $F$, hence are antipodes of $\crossPoly$.
Moreover, they must be distinct as there are no roots parallel to antipodes of $\crossPoly$.
Finally, consider the antipodes $A \cdot \omega_J$ and $A\cdot(-\omega_J)$ in $3_{21}$.
There is some separating hyperplane $H$ through which we can reflect to get the antipode $s_HA \cdot \omega_J$ and $s_HA\cdot(-\omega_J)$.
Moreover, this antipode is distinct as $A \cdot \omega_J - A\cdot(-\omega_J)$ is not parallel to a root, hence $\nu_M$ satisfies $f_{3_{21}}$.
\end{proof}

\section{Equations from homogeneous spaces}\label{s:eqsfromGrass}

In this section, we show that the strong exchange equations defined in \Cref{s::charCoxMatroids} arise as tropicalisations of equations of the Lichtenstein embedding,
the inclusion of the homogeneous space $\bbG/\bbP_\alpha$ into the projective space $\proj(V_\lambda)$ (\Cref{sssec:GrassmannPrelims}).
For each minuscule type, we will derive a set of quadratic embedding equations whose tropicalisations match the previous section
and which span the degree~2 graded component of the Lichtenstein embedding equations. This section is a case-by-case proof of \Cref{thm:A},
excluding the well-established case of type~$A_n$ which we state without proof.
Table~\ref{table:where} lists 
the specialisation of \Cref{thm:A} for each $\bbG,\bbP$ with $\bbG$ complex and semisimple and $\bbP$ minuscule.

\begin{remark}\label{rem:coefficients}
We will see that the coefficients of the quadrics in our results can all be taken to be integers,
so that in fact the image of the embedding $\bbG/\bbP\subset\proj(V_\lambda)$ is defined over~$\ZZ$.
For classical groups $\bbG$ and $E_6$, the coefficients of the quadrics are all $\pm1$,
and therefore for these types we can change base from $\ZZ$ to~$\FF_p$ for any prime $p$ and still have a variety defined by equations whose tropicalisations agree with \Cref{thm:A}.
For $E_7$, coefficients $\pm3$ appear and thus for $E_7$ we can only change base from $\mathbb{Z}$ to $\mathbb{F}_p$ for a prime $p \neq 3$.
\end{remark}

\subsection{Preliminaries}

\subsubsection{Minuscule varieties}\label{sssec:GrassmannPrelims}
We begin with some preliminaries on the Lie theoretic approach to Grassmannians and their generalisations, minuscule varieties,
in particular how we obtain equations cutting out these minuscule varieties.
A general reference for this section is~\cite{Fulton+Harris:2004}.

Let $W$ be a Weyl group with root system $\rootSystem$ and a fixed choice of simple roots $\simpleRoots$.
Fix $\alpha \in \simpleRoots$ and let $\maximalParabolic{\alpha}$ be the parabolic subgroup of $W$ associated to $\simpleRoots \setminus \{\alpha\}$.
Recall from \Cref{sssec:coxeter-matroids-from-polytopes} that a Coxeter matroid is a subset of $\cosets{\simpleRoots\setminus \alpha} = W/\maximalParabolic{\alpha}$ whose associated polytope $P(M)$ has edges parallel to roots.
Let $\bbG$ be the simply connected complex Lie group with root system $\rootSystem$ and Lie algebra $\fg = \Lie(\bbG)$.
Let $\{\lambda_i\}$ be the set of fundamental weights in bijection with $\simpleRoots$.
We fix $\lambda$ to be the fundamental weight that pairs with the omitted simple root $\alpha$ and let $V_\lambda$ be the irreducible representation of $\bbG$ with highest weight $\lambda$ and highest weight vector $v_\lambda$. The homogeneous space we study in this section is $\bbG/\bbP_{\alpha}$ where $\bbP_\alpha$ is the stabiliser of $v_\lambda \in \proj(V_{\lambda})$.
The subgroup $\bbP_\alpha$ is a maximal parabolic of $\bbG$ and $\bbG /\bbP_\alpha$ embeds into the projective space $\proj(V_\lambda)$ as the $\bbG$-orbit of $v_\lambda$.

The equations that cut out $\bbG /\bbP_\alpha$ in $\proj(V_\lambda)$ are all quadrics, hence we can view them as linear forms on $S^2(V_\lambda)$.
Moreover, they can be calculated by the action of the Casimir operator $\Omega$.

\begin{theorem}[\cite{Lichtenstein:1982}]\label{t:eqfromcas}
Let $S^2(V_\lambda)$ be the symmetric square of $V_\lambda$, and $\rho$ half the sum of the positive roots.
The system of quadrics
 \begin{equation} \label{eq:casimir+eqns}
 \Omega (u \otimes v ) - \InnerProd{ 2\lambda}{\, 2\lambda+ 2\rho}(u \otimes v) = 0
 \end{equation}
 cuts out $\bbG/\bbP$ inside $\proj(V_\lambda)$.
\end{theorem}

Rather than working directly with the Casimir operator, consider the decomposition of $S^2(V_\lambda) = \bigoplus V_\mu$ into irreducible $\bbG$-modules.
Fixing a basis for each $V_\mu$, their union is a basis $\{b_i\}$ of $S^2(V_\lambda)$ with associated dual basis $\{b_i^\vee\}$ of $S^2(V_\lambda^\vee)$.
The following lemma states that the quadratic equations \eqref{eq:casimir+eqns}
have a simple description in the dual basis, which is a basis for the space of quadrics on~$V_\lambda$.

\begin{lemma}\label{lem:grassmann+embedding+equations}
    Let $S^2(V_\lambda) = \bigoplus V_\mu$ be a decomposition into irreducible $\bbG$-modules and $\{b_1, \dots, b_k\}$ a basis for $S^2(V_\lambda)$ such that each $b_i \in V_\mu$ for some highest weight $\mu$.
    Then the system of quadrics cutting out $\bbG/\bbP$ inside $\proj(V_\lambda)$ is the subspace
  
    \begin{equation} \label{eq:grassmann+embedding+equations}
        \mathrm{Span}\{ b_i^\vee \, \st \, b_i \notin V_{2\lambda} \} \subset S^2(V_\lambda^\vee) \, ,
    \end{equation}
    where $\{b_1^\vee, \dots, b_k^\vee\}$ is the dual basis of $\{b_1,\ldots,b_k\}$ for $S^2(V_\lambda^\vee)$.
\end{lemma}

\begin{proof}
Recall the system of quadrics \eqref{eq:casimir+eqns} from \Cref{t:eqfromcas}.
The Casimir acts on an irreducible module by a scalar, namely, if $u \otimes v \in V_\mu$ then $\Omega(u \otimes v) = \InnerProd{ \mu}{\, \mu + 2 \rho} \left( u \otimes v\right)$~\cite[page 429]{Fulton+Harris:2004}.
Given our choice of basis, the equations~\eqref{eq:casimir+eqns} become 
\begin{equation}
\label{eq:intermediate+casimir}
\Omega(b_i) - \InnerProd{ 2\lambda}{ 2\lambda + 2\rho } b_i = (\InnerProd{ \mu}{\, \mu + 2 \rho} - \InnerProd{ 2\lambda}{\, 2\lambda + 2\rho }) b_i =  0 \, .    
\end{equation}
Let $\mu$ be the highest weight of some irreducible summand of $S^2(V_\lambda)$.
We show that $\InnerProd{ \mu}{\mu + 2 \rho} = \InnerProd{ 2\lambda}{2\lambda + 2\rho}$ if and only if $\mu = 2 \lambda$.

Recall that any weight of $S^2(V_\lambda)$ is a sum of two weights of $V_\lambda$ and that every weight of $V_\lambda$ is of the form $\lambda -\sum_{\alpha \in \rootSystem^+} c_\alpha \alpha$ for some $c_\alpha\geq 0$.
As such, we can write $\mu = 2\lambda - \mu'$ where $\mu'$ is a (possibly empty) sum of positive roots.
Then 
\begin{align*} 
    \InnerProd{\mu}{\mu + 2 \rho} &= \InnerProd{2 \lambda - \mu'}{2\lambda-\mu' + 2\rho} \\
    &= \InnerProd{2\lambda}{2\lambda + 2\rho} -\InnerProd{\mu'}{2\lambda + 2\rho}  -\InnerProd{2\lambda}{\mu'} + \InnerProd{\mu'}{\mu'} \\
    &= \InnerProd{2 \lambda}{2\lambda + 2\rho} - \InnerProd{\mu'}{2\rho + (2\lambda - \mu') + 2\lambda} .
\end{align*}
Since $2\lambda$, $\mu$ and $2\rho$ are highest weights of representations, they are all dominant and thus pair with a sum of positive roots non-negatively, i.e.\ $\InnerProd{\mu'}{2\rho}$, $\InnerProd{\mu'}{2\lambda - \mu'}$, $\InnerProd{\mu'}{2\lambda}$ are all non-negative.
Furthermore, $2\rho$ is not stabilised by any Weyl group element and thus is strictly dominant.
Hence $\InnerProd{\mu'}{2\rho} > 0$ for any positive $\mu' \neq 0$.
Therefore $\InnerProd{\mu'}{2\rho + 2\lambda - \mu' + 2\lambda} \geq 0$ with equality if and only if $\mu' = 0$.
It follows that $\InnerProd{\mu}{\mu + 2\rho} = \InnerProd{2 \lambda}{2 \lambda + 2\rho}$ if and only if $\mu' = 0$, i.e.~$\mu = 2\lambda$.

With this observation, we have that equation \eqref{eq:intermediate+casimir} simplifies to
\begin{equation*} 
b_i = 0  \quad \forall \, b_i \notin V_{2\lambda} \, .
\end{equation*}
As an element of $S^2(V_\lambda^\vee)$, this equation corresponds to the dual basis vector $b_i^\vee$.
\end{proof}

We call the equations \eqref{eq:grassmann+embedding+equations} from \Cref{lem:grassmann+embedding+equations} the \defn{Lichtenstein embedding equations}. 
While these are elegantly phrased in this basis, they do not yield explicit polynomial equations in terms of variables on $V_\lambda$.
To obtain explicit equations, we restrict our analysis to minuscule representations for the following reason.

Recall we can decompose $V_\lambda$ as a direct sum of weight spaces $V_\lambda = \bigoplus_{\mu \in \Lambda} V_\lambda(\mu)$.
If $V_\lambda$ is minuscule, each weight space is one-dimensional.
Fixing some vector $v_\mu \in V_\lambda(\mu)$ for each $\mu \in \Lambda$ gives us a canonical basis $\set{v_\mu}{\mu \in \Lambda}$ for $V_\lambda$ up to scaling.
Likewise, we get a canonical dual basis $\set{x_\mu}{\mu \in \Lambda}$ for $V_\lambda^\vee$.
In this sense, the equations \eqref{eq:casimir+eqns} can be viewed not just as elements of $S^2(V_\lambda)^\vee \cong S^2(V_\lambda^\vee)$, but as quadrics in variables $x_\mu$ indexed by the weights of $V_\lambda$.
This allows us to write the Lichtenstein embedding equations from \Cref{lem:grassmann+embedding+equations} as explicit polynomial equations using only the weight space decomposition.

We note that if $V_\lambda$ is not minuscule there is no clear canonical basis for $V_\lambda$ and hence we would have to choose a particular basis.
Tropicalisation is very sensitive to the choice of basis and it is not clear what the correct choice should be.

\subsubsection{Grassmann cone preserving maps}
The collections of strong exchange equations for the various types in \Cref{s::charCoxMatroids}
share the feature that many of the equations concern just vertices on a particular face $\face$ of $\ambientPoly$.
Furthermore, $\face$ is itself a $W'$ orbit polytope where $W'$ is the stabiliser of $\face$ in~$W$,
and moreover these equations supported on~$\face$ are also strong exchange equations for~$W'$.
This property has an algebro-geometric meaning,
which we explain here for context, 
although we will do our computations without relying on it.

With the notations of the start of this section, let $\rootSystem'\subset\rootSystem$ be the root system of~$W'$. 
Let $\simpleRoots'$ be the choice of simple roots for~$\rootSystem'$
that makes a root of~$\rootSystem'$ positive with respect to $\simpleRoots'$ if and only if it is positive with respect to $\simpleRoots$ in~$\rootSystem$.
Let $\bbG'$ be a complex reductive group with root system~$\rootSystem'$.
By restriction $V_\lambda$ is a $\bbG'$-module.
Since complex reductive groups are semisimple, any $\bbG'$-submodule $V'$ of $V_\lambda$ splits,
so there exists a projection $p:V_\lambda\to V'$ of $\bbG'$-modules,
and its projective counterpart the rational map $\proj(p):\proj(V_\lambda)\dashrightarrow\proj(V')$.

It may happen that $\proj(p)(\bbG/\bbP)$ is a single $\bbG'$-orbit in $\proj(V')$, 
indeed the orbit of a highest weight vector $v_{\lambda'}$.
By composing $p$ with an automorphism of~$V'$ we may assume that $p(v_\lambda) = v_{\lambda'}$.
Let $\bbP'$ be the stabiliser of $v_{\lambda'}\in\proj(V')$.
Then $\bbP'$ is a parabolic subgroup of $\bbG'$
and we obtain an inclusion
\[
    \bbG'/\bbP'\cong\proj(p)(\bbG/\bbP)\subseteq\proj(V')\, .
\]

When a linear projection $p:V_\lambda\to V'$ is applied to a projective variety $X\subseteq\proj(V_\lambda)$, 
the ideal of homogeneous polynomials vanishing on~$p(X)$ is obtained by intersecting the ideal of~$X$ with a subring.
Therefore, in the situation of this section,
the equations in $S^2(V_{\lambda'}^\vee)$ defining $\bbG'/\bbP'$, under the inclusion $p^*:S^2(V_{\lambda'}^\vee)\hookrightarrow S^2(V_{\lambda}^\vee)$,  are equations that vanish on~$\bbG/\bbP$.
These equations in the image of the pullback $p^*$ will contain only variables on the face $\face$
and it is their tropicalisations which will be the strong exchange equations supported on~$\face$.

\begin{remark}
    In \cite{KPRS}, Kasman, Pedings, Reiszl and Shiota
    observe that wedge powers of linear maps are examples of projections $V_\lambda\to V'$ that send Grassmannians to Grassmannians,
    which they call \defn{Grassmann cone preserving maps}.
    Their main result is that all Grassmannians are set-theoretically defined
    by equations pulled back from the single Grassmannian of 2-planes in 4-space.
    Seynnaeve and Tairi \cite{SeynnaeveTairi}
    show a counterpart for isotropic Grassmannians in types $B$, $C$, and~$D$,
    where the equations are pulled back from those for isotropic 3-planes in 7-space or 4-planes in 8-space.
    
    For both of these results, it is crucial that Grassmann cone preserving maps be considered
    which do not send weight spaces of $V_\lambda$ to weight spaces of~$V'$ (or zero).
    The maps $p$ that arise from faces $\face$ of $\ambientPoly$
    do preserve weight spaces in this sense,t
    and pullbacks along only Grassmann cone preserving maps which preserve weight spaces do not set-theoretically cut out the correct varieties.
\end{remark}

\subsubsection{Matroids (type A)}

As a warm-up, we detail the setup in type $A_{n-1}$.
Its simply connected complex Lie group is $\bbG=\mathrm{SL}_n$ with corresponding Lie algebra $\mathfrak{sl}_n$.
The fundamental weight $\lambda_k$ is minuscule for all $1 \leq k \leq n-1$, and its corresponding fundamental $\mathfrak{sl}_n$-representation is $V_{\lambda_k} = \bigwedge^k \CC^n$ whose weights are $\set{\sum_{i \in I}e_i}{I \in \scP_{n,k}}$.
Hence, if $\PP = \PP_{\alpha_k}$ is the maximal parabolic corresponding to the $k$-th fundamental weight, we have that $\bbG/\PP \subset \proj(\bigwedge^k \CC^n)$.

Recall that we can write the embedding equations cutting out $\bbG/\PP \subset \proj(\bigwedge^k \CC^n)$ in terms of variables indexed by the weights of $\bigwedge^k \CC^n$.
Hence, we write $X = \set{x_I}{I \in \scP_{n,k}}$ which we consider both as a set of variables and a dual basis for $\bigwedge^k \CC^n$.
To find the Lichtenstein embedding equations for type $A$ we could find the irreducible decomposition $S^2(\bigwedge^k \CC^n) = \bigoplus V_{\mu}$ and then write the equations from \Cref{lem:grassmann+embedding+equations} terms of the dual basis vectors $X$.
However, in type $A$ these equations are already well known: they are the Grassmann-Pl\"ucker relations.

\begin{definition}
    The \defn{quadratic Grassmann--Pl\"ucker relations} are the collection of equations 
    \begin{align*} 
        \cE^{(A)}_{n,k} &\coloneqq \set{F_{I,J}^{(A)} }{ I \in \scP_{n,k-1} \, , \, J \in \scP_{n,k+1} \, , \, |J \setminus I| \geq 3 } \text{ where } \\
        F^{A}_{I,J} &\coloneqq \sum_{i \in J \setminus I} (-1)^{\lngS{i,I} + \lngS{i,J}} x_{I \cup i}x_{J \setminus i} \in \CC[X] \, .
    \end{align*}
\end{definition}

The following theorem is standard. An exposition can be found in the textbook \cite[Section 4.4]{MaclaganSturmfels}.

\begin{theorem}\label{t:Anmain}
Let $\bbG=\mathrm{SL}_n$ and $\bbP$ be the $k$-th maximal parabolic subgroup,
so that $\bbG/\bbP$ is the Grassmannian $\mathrm{Gr}(k,n)$.
The quadratic Grassmann--Pl\"ucker relations $\cE^{(A)}_{n,k}$ are a basis for the equations of the embedding $\bbG/\bbP\subset\proj(\bigwedge^k \CC^n)$,
and tropicalise to the type $A$ strong exchange equations $\cF^{(A)}_{n,k}$.
\end{theorem}

\subsection{Cubes}

\subsubsection{Type $B$ embedding equations}
\label{sssec:type-B-embedding-equations}

We begin with type $B_n$, where the only minuscule fundamental weight is $\lambda_n$ corresponding to the simple root $\alpha_n$.
This means we have simply connected complex Lie group $\bbG = \mathrm{Spin}(2n+1)$ with corresponding Lie algebra $\mathfrak{so}_{2n+1}$.
The fundamental representation of $\mathfrak{so}_{2n+1}$ corresponding to $\lambda_n$ is the \defn{spin representation} $\Sp$.
It is $2^n$-dimensional with weights $(\pm \half,\, \ldots,\, \pm \half)$.
Hence, if $\bbP = \bbP_{\alpha_n}$ is the maximal parabolic corresponding to $\alpha_n$ we have that $\bbG/\bbP \subset \proj(\Sp)$.

Recall that we define the embedding equations cutting out $\bbG/\bbP \subset \proj(\Sp)$ in terms of variables indexed by the weights of $\Sp$.
The weights of $\Sp$ are exactly $Q_n = \set{\cubeVertex{I}}{I \subseteq [n]}$ as defined in \eqref{eq:cubeVertex} and hence are in 1-1 correspondence with subsets of~$[n]$ or, equivalently, the vertices of the $n$-cube.
Hence, we write $X=\set{x_I}{I\subseteq[n]}$ where $x_I$ is the variable corresponding to weight $\cubeVertex{I}$.

\begin{definition} \label{d:quadraticembeddingeqs} 
The \defn{type $B$ quadratic embedding equations} are the collection of equations
\begin{align}
    \cE^{(B)}_{n} &:= \set{F_{I,J}^{(B)} }{ I, J \subseteq [n], |I \symDiff J| \geq 3 }  \quad \text{ where } \nonumber\\
    F^{(B)}_{I,J} &:= \sum_{i \in I \symDiff J} (-1)^{\lngS{i,I}+\lngS{i,J}}x_{I \symDiff i} x_{J \symDiff i} & \text{ if } |I \symDiff J| \equiv 0 \mod 2\phantom. \label{eq:GP-B-even} \\
    F^{(B)}_{I,J} &:= - x_Ix_J + \sum_{i \in I \symDiff J}(-1)^{\lngS{i,I}+\lngS{i,J}}x_{I \symDiff i} x_{J \symDiff i} & \text{ if } |I \symDiff J| \equiv 1 \mod 2. \label{eq:GP-B-odd}
\end{align}
\end{definition}

We motivate calling them the quadratic embedding equations via the following theorem.

\begin{theorem}\label{t:BnSpinmain}
Let $\bbG = \mathrm{Spin}(2n+1)$ and $\bbP$ be a the maximal parabolic corresponding to the last fundamental weight. 
The type $B$ quadratic embedding equations $\cE^{(B)}_n$ are a spanning set for the equations of the embedding $\bbG/\bbP\subset\proj(\Sp)$ and tropicalise to the type $B$ strong exchange equations $\cF^{(B)}_n$.
\end{theorem}

The claim in the theorem about tropicalisation is clear by inspection when one compares (\ref{eq:GP-B-even},~\ref{eq:GP-B-odd}) to (\ref{eq:SEE-B-1},~\ref{eq:SEE-B-2}) above.
The proof of the rest is the subject of \Cref{App:B,App:B.2}.
The equations are not new, but we have provided a proof as we could not find a reference that produces these explicit quadratic polynomials, especially with their signs overt. See further \Cref{rem:B+D+Gr+iso}.

\begin{example}\label{ex:B5}
    We emphasise that the type $B$ quadratic embedding equations do not form a basis for the Lichtenstein embedding equations, as the following example demonstrates.
    Consider the $n=5$ case: this corresponds to the minuscule variety $\mathrm{Spin}(11)/ \bbP_{\alpha_5} \subseteq \proj(\Sp)$ where $\dim(\Sp) = 2^5$.
    The analysis in \Cref{p:Ssquaredecomp} demonstrates that $S^2(\Sp)$ decomposes into $\twedge^2 \CC^{11} \oplus \twedge^{6} \CC^{11} \oplus \twedge^{10} \CC^{11}$, where $\twedge^{6} \CC^{11}$ is the component with highest weight $2\lambda = (1, \dots, 1)$.
    It follows from \Cref{lem:grassmann+embedding+equations} that the space of Lichtenstein embedding equations is of dimension $\binom{11}{2} + \binom{11}{10} = 66$.
    Moreover, the space of Lichtenstein embedding equations of weight zero is six-dimensional, corresponding to the direct sum of the Cartan subalgebra (dimension $5$) and the $1$-dimensional span of the weight zero vector in $(\CC^{11})^\vee \cong \twedge^{10} \CC^{11}$.
    This can be equivalently computed via the analysis in \Cref{App:BspinEqs} by counting subsets $M \subseteq [5]$ such that 
    \[
    2|M| \in \cM_{\emptyset,\emptyset} = \set{4k + 6 - \epsilon}{k \in \ZZ \, , \, \varepsilon \in \{0,1\}} \setminus \{3,4,5,6,7\} \, .
    \]
    There are $6$ such subsets, namely, $M$ is either a singleton or equal to $[5]$.

    On the other hand, the quadratic embedding equations of weight zero are precisely those $F_{I,J}^{(B)}$ with $I \symDiff J = [5]$.
    The number of these is equal to the number of antipodes, which is $2^4 = 16$.
    This demonstrates that, although the quadratic embedding equations span the Lichtenstein embedding equations, they are far from a basis. 

    Note that in Theorem \ref{t:BspinEqs} we give equations that do form a basis for the Lichtenstein embedding equations.
\end{example}

\subsubsection{Type $D$ embedding equations}

We now look at type $D_n$, where the minuscule fundamental weights are $\lambda_1, \lambda_{n-1}$ and $\lambda_n$.
We have simply connected complex Lie group $\bbG = \mathrm{SO}(2n)$ with corresponding Lie algebra $\mathfrak{so}_{2n}$.
The fundamental representation of $\mathfrak{so}_{2n}$ corresponding to $\lambda_1$ is the standard representation $V_{\lambda_1} = \CC^{2n}$: we will treat this case in \Cref{s:Dncrosspolytope} as its analysis is quite different.
The fundamental representations corresponding to $\lambda_{n-1}$ and $\lambda_n$ are the \defn{half spin representations} $V_{\lambda_{n-1}} = \Sp_-$ and $V_{\lambda_n} = \Sp_+$.
These have highest weights $(\half,\ldots,\half,-\half)$ and $(\half,\ldots,\half)$ respectively, and are both of dimension $2^{n-1}$.
Hence, we have $\bbG/\bbP_{\alpha_{n-1}} \subset \proj(\Sp_-)$ and $\bbG/\bbP_{\alpha_{n}} \subset \proj(\Sp_+)$.

The weights of $\Sp_+$ are exactly $\cD_{n,+} = \set{\cubeVertex{I}}{|I| \equiv 0 \mod 2}$ as given in \eqref{eq:demicubeVertex+} and the weights of $\Sp_-$ are $\cD_{n,-} = \set{\cubeVertex{I}}{|I| \equiv 1 \mod 2}$  as in \eqref{eq:demicubeVertex-}.
We write $X_+ = \set{x_I}{|I| \equiv 0 \mod 2}$ and $X_- = \set{x_I}{|I| \equiv 1 \mod 2}$ for the variables corresponding to the weights of $\cD_{n,+}$ and $\cD_{n,-}$ respectively.
As $\cD_{n,+}$ and $\cD_{n,-}$ partition $Q_n$, the variables $X_+$ and $X_-$ partition the variables $X = \set{x_I}{I \subseteq [n]}$.
Just as the Type $B_n$ quadratic embedding equations were supported on antipodes in subcubes, we will define the Type $D_n$ quadratic embedding equations to be supported on antipodes in subdemicubes. 

The \defn{type $D$ quadratic embedding equations} are the collection of equations $\cE_n^{(D)} = \cE_{n,+}^{(D)} \cup \cE_{n,-}^{(D)}$ where
\begin{align*}
\cE_{n,+}^{(D)} &:= \set{F_{I,J}^{(D)}}{I,J \subseteq [n]\, , \, \order{I \symDiff J} \geq 4 \, , \, \order{I} \equiv \order{J} \equiv 1 \mod 2} \subseteq \CC[X_+] \\
\cE_{n,-}^{(D)} &:= \set{F_{I,J}^{(D)}}{I,J \subseteq [n]\, , \, \order{I \symDiff J} \geq 4 \, , \, \order{I} \equiv \order{J} \equiv 0 \mod 2} \subseteq \CC[X_-] \\
F_{I,J}^{(D)} &:= \sum_{i \in I \symDiff J} (-1)^{\lngS{i,I}+ \lngS{i,J}} x_{I \symDiff i} x_{J \symDiff i} \in \CC[X] \, .
\end{align*}

Observe that we can obtain the type $D$ quadratic embedding equations as a subset of the type $B$ quadratic embedding equations that are entirely supported on the demicubes $\cD_{n,+}$ and $\cD_{n,-}$, i.e.
\[
\cE_{n,+}^{(D)} = \cE_{n}^{(B)} \cap \CC[X_+] \, , \quad \cE_{n,-}^{(D)} = \cE_{n}^{(B)} \cap \CC[X_-] \, .
\] 
Proposition \ref{p:intersectgrassBtoD}, proved in the appendix, states that the analogous statement holds for the Type $D_n$ Lichtenstein embedding equations.

\begin{restatable}{proposition}{thDneqs}
 \label{p:intersectgrassBtoD}
    The subspace of type $D_n$ Lichtenstein embedding equations for $\bbG/\PP_{\alpha_n}$ (respectively $\bbG/\PP_{\alpha_{n-1}}$) is the subspace of type $B_n$ Lichtenstein embedding equations that are entirely supported on the even demicube $\cD_{n,+}$ (respectively the odd demicube $\cD_{n,-}$).
\end{restatable}

Our main theorem of this subsection follows from Theorem \ref{t:BnSpinmain}, the analogous statement for type $B_n$.

\begin{theorem}\label{t:Dnspinmain}
Let $\bbG = \mathrm{SO}(2n)$.
\begin{itemize}
    \item The equations $\cE_{n,+}^{(D)}$ are a spanning set for the equations of the embedding $\bbG/\bbP_{\alpha_n}\subset\proj(\Sp_+)$,
    \item The equations $\cE_{n,-}^{(D)}$ are a spanning set for the equations of the embedding $\bbG/\bbP_{\alpha_{n-1}}\subset\proj(\Sp_-)$.
\end{itemize}
Moreover, the type $D$ quadratic embedding equations $\cE^{(D)}_n$ tropicalise to the type $D$ strong exchange equations $\cF^{(D)}_n$.
\end{theorem}

\begin{proof}
    Observe from \eqref{eq:GP-B-even} and \eqref{eq:GP-B-odd} that
$\cE_{n}^{(B)} = \cE_{n,+}^{(D)} \cup \cE_{n,-}^{(D)} \cup \cE_{\pm}^{(B)}$ where $\cE_{\pm}^{(B)}$ are the equations \eqref{eq:GP-B-odd}.
Furthermore, each monomial in every equation in $\cE_{\pm}^{(B)}$ is the product of a variable in $X_+$ and a variable in $X_-$, hence $\operatorname{Span}\cE_{\pm}^{(B)} \cap \CC[X_+] = \operatorname{Span}\cE_{\pm}^{(B)} \cap \CC[X_-] = \{0\}$.
It follows that
\begin{equation}\label{eq:BtoD+intersect}
    \mathrm{Span} (\cE_{n}^{(B)} ) \cap \CC[X_+] =\mathrm{Span} ( \cE_{n,+}^{(D)}) \, , \quad \mathrm{Span} (\cE_{n}^{(B)} ) \cap \CC[X_-] =\mathrm{Span} ( \cE_{n,-}^{(D)}) \, .
\end{equation}
By \Cref{t:BnSpinmain}, the subspace $\mathrm{Span} (\cE_{n}^{(B)})$ equals the subspace of type $B_n$ Lichtenstein embedding equations.
\Cref{p:intersectgrassBtoD} states that intersecting with $\CC[X_+]$ (or equivalently those equations supported on $\cD_{n,+}$) gives the type $D_n$ Lichtenstein embedding equations for $\bbG/\bbP_\alpha$.
It follows from \eqref{eq:BtoD+intersect} that $\cE_{n,+}^{(D)}$ are a spanning set for these Lichtenstein embedding equations.
The proof for $\cE_{n,-}^{(D)}$ follows analogously.
\end{proof}

\begin{remark}\label{rem:B+D+Gr+iso}
    The Type $B_n$ minuscule variety, the orthogonal Grassmannian $\mathrm{OGr}(n,2n+1)$, is diffeomorphic to either of the type $D_{n+1}$ minuscule varieties, the orthogonal Grassmannians $\mathrm{OGr}^+(n+1,2n+2)\cong\mathrm{OGr}^-(n+1,2n+2)$. 
    In fact either of the half spin representations of $\mathrm{Spin}(2n+2)$, restricted to $\mathrm{Spin}(2n+1)$, is isomorphic to the single spin representation. Thus one can include the spin representation of $\mathrm{Spin}(2n+1)$ in either of the half spin representations of $\mathrm{Spin}(2n+2)$ (\Cref{lem:isospinmod}).  The isomorphism sends the type~$B_n$ variable $x_I$ to the type~$D_{n+1}$ variables $x_I$ or $\sqrt{-1} x_{I\cup\{n+1\}}$ depending on which variable has a subscript of the correct parity.
    Hence, this induces an isomorphism between the Lichtenstein embedding equations of $\mathrm{OGr}(n,2n+1)$ and the Lichtenstein embedding equations of $\mathrm{OGr}^\pm(n+1,2n+2)$, and likewise between the two sets of the quadratic embedding equations.
    Note that the map from \Cref{rem:strongB_nisStrongD_n+1} is the combinatorial analogue of this isomorphism.
   
    
    
    


    
    In fact the equations $\cE_{n,+}^{(D)}$ are the Wick relations \cite{conciniprocesi82}. We failed to find a reference that proves the Wick relations cut out the orthogonal Grassmannian $\mathrm{OGr}^+(n+1,2n+2)$ although this does seem to be folklore, mentioned in \cite{sturmfelsvelasco2010} and \cite{Rincon}.
    Manivel gives these equations up to signs in \cite[6]{manivel2009}, in which he states 
    \begin{quote}The [signs] can be written down explicitly but seemingly not in a pleasant way. These relations are probably known, but we do not know any suitable reference.\end{quote}
    This reflects our feeling precisely.
    Procesi states only that these equations can be deduced from an equation \cite[434]{procesi2007} using two identities \cite[(7.2.3)]{procesi2007} and \cite[(7.2.1)]{procesi2007}. Seshadri \cite{seshadri1978geometry} gives quadratic equations up to signs with different support.
    In \cite[Corollary 16]{chirivi2013pfaffians}, Chirivì and Maffei give an alternative presentation using shuffle relations, but their equations have strictly larger support.
    It seems that in the literature the equations $\cE_{n}^{(D)}$ are implicitly mentioned but not explicitly calculated. For completeness we have therefore included a full derivation.

\end{remark}

\subsection{Cross polytopes} 

\subsubsection{$D_n$ cross polytope}\label{s:Dncrosspolytope}

In this section we consider $\bbG/\bbP$ to be such that $\bbG = \mathrm{SO}(2n)$ and $\bbP_{\alpha_1}$ is the maximal parabolic associated to the first fundamental weight.
The standard representation $V_{\lambda_1} = \CC^{2n}$ of $\mathfrak{so}_{2n}$ is minuscule with weights $\set{\pm e_i}{1 \leq i \leq n}$, corresponding to the $n$-cross polytope $\lozenge_n$.
Let $\{v_1,\ldots ,v_n,v_{-1} ,\ldots, v_{-n}\}$ be a basis of weight vectors in $V_{\lambda_1}$ such that $v_i$ has weight $e_i$ and $v_{-i}$ weight $-e_i$.
The form $B: V_{\lambda_1}\times V_{\lambda_1} \to \CC$ defining $\SO(2n) = \SO(2n,B)$ is such that $B(v_i,v_{-j}) = \delta_{ij} = B(v_{-j},v_{i})$ and $B(v_i,v_j) = 0 = B(v_{-i},v_{-j})$.
It is well known, as demonstrated using Newell-Littlewood coefficients \cite{Littlewood1958,Newell1951}, that as a $\mathfrak{so}(2n)$ representation
\[
    S^2(V_{\lambda_1}) = V_{2\lambda_1} \oplus V_0 \, .
\]
As $V_0$ is one-dimensional, applying \Cref{lem:grassmann+embedding+equations} we see there is a single Lichtenstein embedding equation given by the dual basis vector to $V_0$.
We can write this in terms of the dual basis for $V_{\lambda_1}$ as follows.
The form $B$ on $V_{\lambda_1}$ defining $\mathfrak{so}(2n)$ is symmetric and by definition $\mathfrak{so}(2n)$ invariant. Thus $B \in S^2(V_{\lambda_1}^\vee)$ spans the submodule dual to $V_0 \subset S^2(V_{\lambda_1})$.
Write $X = \{x_i\} \cup \{x_{-i}\}$ for the dual basis in $V_{\lambda_1}^\vee$.
Because $V^\vee \cong V$ as $\mathfrak{so}(2n)$ modules, we find that 
\[ 
    F_{\lozenge_n}^{(D)} := \sum_{i=1}^n x_ix_{-i} \in \CC[X]
\] 
is equal to $B$ and hence is our single quadratic embedding equation.
This is summarised in the following lemma.
\begin{lemma}\label{l:Dncrossmain}
Let $\bbG = \mathrm{SO}(2n)$ and $\bbP_{\alpha_1}$ be the maximal parabolic associated to the first fundamental weight.
The quadratic embedding equation $F_{\lozenge_n}^{(D)}$ cuts out the embedding $\bbG/\bbP_{\alpha_1} \subset \proj(\CC^{2n})$, and tropicalises to the strong exchange equation $f_{\lozenge_n}^{(D)}$.
\end{lemma}

\subsubsection{$C_n$ cross polytope}
In this section, we consider $\bbG/\PP$ to be such that $\bbG = \mathrm{Sp}(2n)$ and $\PP_{\alpha_1}$ is the maximal parabolic associated to the first fundamental weight.
As in the previous case, the standard representation $V_{\lambda_1} = \CC^{2n}$ of $\mathfrak{sp}_{2n}$ is minuscule with weights $\set{\pm e_i}{1 \leq i \leq n}$, corresponding to the $n$-cross polytope $\lozenge_n$.
In this case, using character calculations (as described in \cite{Littlewood1958}) the representation $S^2(V_{\lambda_1})$ is isomorphic to $V_{2\lambda_1}$, thus 
there are no Lichtenstein embedding equations in this setting.
This precisely matches the equivalent statement \Cref{l:strexchCncross} which states that the strong exchange condition holds for any subset of the $C_n$ cross polytope.

\begin{lemma}\label{l:Cncrossmain}
    Let $\bbG = \mathrm{Sp}(2n)$ and $\bbP_{\alpha_1}$ be the maximal parabolic associated to the first fundamental weight.
 Then $\bbG/\bbP_{\alpha_1} = \proj(\CC^{2n})$.
\end{lemma}

\subsection{The exceptional cases}

We consider the $E_6$ and $E_7$ cases in this section.
Explicit descriptions of the root systems for $E_6$ and $E_7$ were given in \Cref{ssec:the-exceptional-Weyl-groups}. Many of the results and facts in this section have been verified by computer check: please see our repository \cite{githubrepo} 
for explicit implementation.

\subsubsection{$E_6$ minuscule case(s)}
For type $E_6$, recall that there are two minuscule fundamental weights, the last fundamental weight $\lambda = \lambda_6$ and the first fundamental weight $\lambda' = \lambda_1$.
As these give rise to diffeomorphic varieties, it suffices to consider only $\lambda$.
The fundamental representation $V_{\lambda}$ is $27$-dimensional and its weights are, up to translation and dilation, the vertices of the $2_{21}$ polytope.
A computer calculation shows the symmetric square decomposes as $S^2(V_\lambda^\vee) \cong V_{2\lambda}^\vee \oplus V_{\lambda'}^\vee$.
As $V_{\lambda'}^\vee$ is $27$-dimensional, this gives us $27$ quadrics as a basis for the Lichtenstein embedding equations.
Moreover, as $V_{\lambda'}^\vee$ is minuscule, we can choose a basis such that there is one equation per weight space of $V_{\lambda'}^\vee$.

To explicitly write these equations, set $X = \set{x_A}{A \in W(E_6)^{S \setminus s_6}}$ and recall that $S = \{s_1, \dots, s_6\}$ is the set of simple reflections for $W(E_6)$.
Recalling that each coset $A$ has an associated minimal coset representative $w_A$, we can equivalently think of $X$ as variables, or as a basis for $V_{\lambda}^\vee$, where $x_A := x_{w_A(\lambda)}$ is the vector associated to weight $w_A(\lambda)$.
Consider the equation
\begin{align*}
    F_6 := 
     x_\lambda & \cdot x_{s_6 s_5 s_4 s_3 s_2 s_4 s_5 s_6 (\lambda)}\\
     &-\left( x_{s_6(\lambda)} \cdot x_{s_5 s_4 s_3 s_2 s_4 s_5 s_6 (\lambda)} \right)\\
     &-\left( x_{s_5 s_6 (\lambda)} \cdot  x_{s_4 s_3 s_2 s_4 s_5 s_6 (\lambda)} \right)\\
     &-\left( x_{s_4 s_5 s_6 (\lambda)} \cdot x_{s_3 s_2 s_4 s_5 s_6 (\lambda)} \right)\\
     &+\left(x_{s_2 s_4 s_5 s_6 (\lambda)} \cdot x_{s_3 s_4 s_5 s_6 (\lambda)} \right)\, .
\end{align*}
This is one of the Lichtenstein embedding equations or, equivalently, an element of $V_{\lambda'}^\vee \subset S^2(V_{\lambda}^\vee)$. Its support corresponds to the five antipodes in a $D_5$ cross-polytope in $2_{21}$: in particular, it tropicalises to one of the type $E_6$ strong exchange equations \eqref{eq:E6}.

To construct the the type $E_6$ quadratic embedding equations, we first observe that the $\fe_6$-submodule generated by $F_6$ is exactly $V_{\lambda'}^\vee$.
As such, we can compute a basis for $V_{\lambda'}^\vee \subset S^2(V_{\lambda}^\vee)$ via computer: the code in \Cref{app:code} produces a basis that we call the \emph{type $E_6$ quadratic embedding equations} $\cE^{(E)}_{6}$.
These equations have the additional property that each one is supported on one of the 27 copies of the $D_5$ cross polytope in $2_{21}$, hence it tropicalises to the type $E_6$ strong exchange equations \eqref{eq:E6}.

\begin{proposition}\label{p:mainthmE6}
Let $\bbG$ be the simply connected complex Lie group of type $E_6$ and $\bbP$ the minuscule parabolic associated to the last fundamental weight $\lambda$.
The type $E_6$ quadratic embedding equations cut out the embedding $\bbG/\PP \subseteq \proj(V_{\lambda})$.
Moreover, the type $E_6$ quadratic embedding equations $\cE^{(E)}_{6}$ tropicalise to the type $E_6$ strong exchange equations $\cF_6^{(E)}$.

\end{proposition}

\begin{proof}    
    We proved this by computer (see \Cref{ss:grassmanianBasis}), constructing
    our special equation $F_6$ in $S^2(V_\lambda^\vee)$ and then calculating the $\fe_6$-subrepresentation generated by this equation.
    We find this to be $27$-dimensional, hence it is $V_{\lambda'}^\vee$.
    Computing a basis of weight vectors for the representation gives our type $E_6$ quadratic embedding equations $\cE^{(E)}_{6}$. As the $\fe_6$ action on $F_6$ gives an $W(E_6)$ action on the weight on $F_6$, each of these equations must be supported on a distinct $D_5$ cross polytope in $2_{21}$, hence $\cE_6^{(E)}$ tropicalises to $\cF_6^{(E)}$. 
\end{proof}

\begin{remark}
Each of the the type $E_6$ quadratic embedding equations are in the $W(E_6)$ orbit of $F_6$, hence one may ask whether it is the full orbit.
However, the full $W(E_6)$ orbit of $F_6$ is of size $270$: each of the 27 $D_5$ cross polytopes has $\binom{5}{2}$ equations in the orbit supported on it.
The only difference between these equations are signs: two coefficients are chosen of the $5$ to have a different sign than the other $3$.
\end{remark}

\subsubsection{$E_7$ minuscule case}

For type $E_7$, recall that the only minuscule fundamental weight is the last, $\lambda = \lambda_7$.
Its fundamental representation $V_{\lambda}$ is $56$-dimensional and its weights are, up to translation and dilation, the vertices of the $3_{21}$ polytope.
A computer calculation shows the symmetric square decomposes as $S^2(V_\lambda^\vee) \cong V_{2\lambda}^\vee \oplus V_{\omega}^\vee$ where 
$V_{\omega}^\vee$ is the first fundamental representation: the adjoint representation of $\mathfrak{e}_7$.
As $V_{\omega}^\vee$ is $133$-dimensional, this gives us $133$ quadrics as a basis for the Lichtenstein embedding equations.
However, $V_{\omega}^\vee$ is not minuscule: it decomposes into $126$ one-dimensional weight spaces indexed by the roots of $E_7$ and a $7$-dimensional zero~weight space corresponding to the Cartan subalgebra of $\fe_7$.
As such, we will have $126$ equations whose weight corresponds to a root of $E_7$ and must choose $7$ equations of weight zero.

To explicitly write these equations, set $X = \set{x_A}{A \in W(E_7)^{S \setminus s_7}}$ and recall that $S = \{s_1, \dots, s_7\}$ are the set of simple reflections for $W(E_7)$.
Recalling that each coset $A$ has an associated minimal coset representative $w_A$, we can equivalently think of $X$ as variables, or as a basis for $V_{\lambda}^\vee$, where $x_A := x_{w_A(\lambda)}$ is the vector associated to weight $w_A(\lambda)$.
Consider the equation
\begin{align*}
    F_7 \coloneqq x_{\lambda} &\cdot x_{s_7 s_6 s_5 s_4 s_2 s_3 s_4 s_5 s_6 s_7(\lambda)}\\
    &- \left(x_{s_7(\lambda)} \cdot x_{s_6 s_5 s_4 s_2 s_3 s_4 s_5 s_6 s_7(\lambda)} \right) \\
    &- \left(x_{s_6 s_7(\lambda)} \cdot x_{s_5 s_4 s_2 s_3 s_4 s_5 s_6 s_7(\lambda)} \right) \\
    &- \left(x_{s_5 s_6 s_7(\lambda)} \cdot x_{s_4 s_2 s_3 s_4 s_5 s_6 s_7(\lambda)} \right) \\
    &- \left(x_{s_4 s_5 s_6 s_7(\lambda)} \cdot x_{s_2 s_3 s_4 s_5 s_6 s_7(\lambda)} \right) \\
    &+ \left(x_{s_2 s_4 s_5 s_6 s_7(\lambda)} \cdot x_{s_3 s_4 s_5 s_6 s_7(\lambda)} \right)\, .
\end{align*}
This is one of the Lichtenstein embedding equations or, equivalently, an element of $V_{\omega}^\vee \subset S^2(V_{\lambda}^\vee)$. 
Its support corresponds to the six antipodes in a $D_6$ cross-polytope in $3_{21}$: in particular, it tropicalises to one of the type $E_7$ strong exchange equations.
As in $E_6$, the $\fe_7$-submodule generated by $F_7$ is exactly $V_{\omega}^\vee$, and hence we can compute a basis using the code in \Cref{app:code}.
This gives $126$ equations that are supported on one of the $126$ copies of the $D_6$ cross polytope in $3_{21}$, along with seven other weight zero equations.
The weight zero equations given by the computer do not tropicalise to any of the type $E_7$ strong exchange equations, hence we must choose these more rigorously.

To do so, we take the standard subalgebra $\mathfrak{e}_6$ corresponding to the maximal standard parabolic subgroup $\maximalParabolic{7}$.
There exists a unique vector (up to scalars) in the weight zero space of $V_\omega^\vee$ that this subalgebra acts by zero on.
This gives an equation supported on the the $28$ antipodes of the $3_{21}$ polytope, hence it tropicalises to the strong exchange equations $f_{3_{21}}$ from \eqref{eq:SEE-E7}. 
We then take six other subalgebras isomorphic to $\mathfrak{e}_6$ corresponding to six non-trivial conjugates of $\maximalParabolic{7}$,
yielding six more equations.
These equations are linearly independent, each have coefficients $\pm 1$ and $\pm 3$ and are all supported on the $28$ antipodes of the $3_{21}$ polytope, hence tropicalise to $f_{3_{21}}$.
The explicit choice we take is the following seven equations associated to seven parabolic subgroups, listed with their list of coefficients:

{\footnotesize
\begin{align}\label{eq:weightzerobasis}
    \begin{split}
    \maximalParabolic{7} & \to [ -3,  1,  1,  1,  1, -1, -1, -1,  1,  1,  1,  1,  1,  1,  1,  1,  1,  1,  1, -1,  1,  1,  1, -1,  1,  1,  1, -1],\\
    s_7(\maximalParabolic{7}) & \to [ -1,  3,  1,  1,  1, -1, -1, -1,  1,  1,  1,  1,  1,  1,  1, -1,  1, -1, -1, -1, -1, -1,  1,  1, -1, -1, -1,  1],\\
    s_6s_7(\maximalParabolic{7})& \to [ -1,  1,  3,  1,  1, -1, -1, -1,  1,  1,  1,  1, -1,  1, -1,  1, -1,  1, -1,  1,  1, -1, -1, -1, -1,  1, -1,  1],\\
    s_5s_6s_7(\maximalParabolic{7})& \to [ -1,  1,  1,  3,  1, -1, -1, -1,  1,  1, -1, -1,  1, -1,  1,  1,  1,  1, -1,  1,  1, -1, -1,  1, -1, -1,  1, -1],\\
    s_4s_5s_6s_7(\maximalParabolic{7})& \to [ -1,  1,  1,  1,  3, -1, -1, -1, -1, -1,  1,  1,  1, -1,  1,  1, -1,  1, -1, -1, -1, -1, -1, -1,  1, -1,  1,  1],\\
    s_3s_4s_5s_6s_7(\maximalParabolic{7}) & \to [ -1,  1,  1,  1,  1,  1, -3, -1,  1, -1,  1, -1,  1,  1, -1,  1,  1, -1, -1, -1,  1,  1, -1, -1, -1, -1,  1,  1],\\
    s_2s_3s_4s_5s_6s_7(\maximalParabolic{7})& \to [ -1,  1,  1,  1, -1, -1, -1,  1,  3,  1,  1, -1,  1,  1, -1,  1,  1, -1, -1,  1,  1,  1,  1,  1, -1,  1, -1, -1]
    \end{split}
\end{align}}
The construction of this linear combination can again be found on our repository \cite{githubrepo} 
in the file ``Zero.gap''.
Note that this choice is not unique: there are $28$ subalgebras of $\mathfrak{e}_7$ isomorphic to $\mathfrak{e}_6$.
A different choice of seven linearly independent equations would lead to a different basis,
or we could even forgo linear independence and take all $28$ of the equations in the spirit of \Cref{ex:B5}.
However, these would be tropically equivalent to our choice.

With this all together, we define the \emph{type $E_7$ quadratic embedding equations} $\cE_7^{(E)}$ to be the 126 equations supported on each copy of $D_6$ in $3_{21}$ as calculated in \Cref{app:code}, along with the seven weight zero equations \eqref{eq:weightzerobasis} supported on the antipodes of $3_{21}$.

\begin{proposition} \label{p:mainthmE7}
Let $\bbG$ be the simply connected complex Lie group of type $E_7$ and $\bbP$ the minuscule parabolic associated to the last fundamental weight $\lambda$.
The type $E_7$ quadratic embedding equations cut out the embedding $\bbG/\PP \subseteq \proj(V_{\lambda})$.
Moreover, the type $E_7$ quadratic embedding equations $\cE^{(E)}_{7}$ tropicalise to the type $E_7$ strong exchange equations $\cF_7^{(E)}$.
\end{proposition}

\begin{proof}
We proved this by computer (see \Cref{app:code}), constructing our special equation $F_7$ in $S^2(V_{\lambda}^\vee)$ and then calculating the $\fe_7$-subrepresentations generated by this equation.
We find this to be $133$-dimensional, hence it is $V_{\omega}^\vee$.
Computing a basis of weight vectors for the representation gives $126$ equations supported on some $D_6$ polytope of $3_{21}$ that are part of $\cE_7^{(E)}$.
As the $\fe_7$ action on $F_7$ gives a $W(E_7)$ action on the weight of $F_7$, each of these equations must be supported on a distinct $D_6$ cross polytope in $3_{21}$, hence these equations tropicalise to the subset of $\cF_7^{(E)}$ supported on cross polytopes.
The remaining seven equations of $\cE_7^{(E)}$ form a basis for the weight zero space of $V_{\omega}^\vee$ by computer check and all tropicalise to the final strong exchange equation $f_{3_{21}}$.
\end{proof}

\begin{remark}\label{rem:do+not+minimise+support}
The seven equations we found in the zero weight space of $V_\omega^\vee$ were not of minimal support.
For each $D_6$ cross polytope $\lozenge$ in $3_{21}$, there is a corresponding cross polytope $\lozenge'$ whose vertices are the antipodes of the vertices of $\lozenge$ in $3_{21}$.
For each of these $63$ $D_6$ cross polytope pairs $(\lozenge, \lozenge')$, there exists a weight zero equation supported on the antipodes on this pair.
These $63$ equations generate the zero weight space of $V_\omega^\vee$.
However, the tropicalisation of these equations do not characterise strong Coxeter matroids: we have a (counter-)example of a strong Coxeter matroid that does not satisfy these tropical equations.
This counter-example can be found in our repository \cite{githubrepo} 
in the file ``Coxeter\_matroids.ipynb''.

In type $B$, we could have defined the quadratic embedding equations to be a set of equations of maximal support in each weight.
This would have been sufficient combinatorially, as tropicalising these also characterises strong $\Delta$-matroids: this will be the theme of a subsequent paper~\cite{peerless_antipodes}.
However, we can also take linear combinations of equations in each weight to minimise the supports, and this gives us the quadratic embedding equations that we presented in \Cref{d:quadraticembeddingeqs}.
In type $E_7$, only the former approach is correct tropically, as trying to minimise the supports gives rise to tropical equations that are not satisfied by strong Coxeter matroids.
This also demonstrates that strong Coxeter matroids may not satisfy the tropicalisation of \emph{all} quadratic equations from the Lichtenstein embedding.

\end{remark}

\printbibliography

\appendix

\section{Explicit equations}\label{app:equations}


This appendix is dedicated to calculating the Lichtenstein embedding equations for the minuscule varieties corresponding to type $B_n$ and $D_n$ with minuscule representation the spin module or the two half spin modules. These varieties are the isotropic Grassmannians in odd and even dimensional spaces respectively. These are not new results, but we failed to find explicit equations for~$B_n$ in the literature, so we record the working here. We approach the computation through representation theory, using results of Lichtenstein \cite{Lichtenstein:1982}. A general reference for this section is \cite{Fulton+Harris:2004}. 
See Section \ref{sssec:GrassmannPrelims} for preliminaries on Lichtenstein's results (Theorem \ref{t:eqfromcas}), and recall Lemma \ref{lem:grassmann+embedding+equations} which shows that if we construct a basis $b_i$ of $S^2(V_\lambda)$ that splits into bases for each irreducible summand then the ideal cutting out $\bbG/\bbP$ is generated by $b_i^\vee$ where $b_i \notin V_{2\lambda}$.

To understand the defining equations of the projective embedding of $\bbG/\bbP_\alpha$ into $\proj(V_\lambda)$, we find a basis of $S^2(V_\lambda)$ consisting of eigenvectors for the Casimir operator $\Omega$. Equivalently, we decompose $S^2(V_\lambda^\vee)$ into irreducible $\fg$-modules and then construct a basis of these irreducible modules. Each basis element of $V_\mu$ for $\mu \neq 2\lambda$ gives a defining equation. 
In the following two sections, we explicitly calculate a basis for the spin module and its symmetric square for Type $B_n$, corresponding to $\bbG = \mathrm{Spin}(2n+1)$ and $V_\lambda$ the spin representation (\Cref{App:B}). We then use these results and knowledge of the restriction from $\mathrm{Spin}(2n+1)$ (type $B_n$) to $\mathrm{Spin}(2n)$ (type $D_n$) to construct a suitable basis for the symmetric square of the two half spin representations of $\mathrm{Spin}(2n)$ (\Cref{App:D}).

\subsection{Equations from Lichtenstein: Type $B_n$}\label{App:B}
In this section, we shall derive explicit equations cutting out $\bbG/\bbP_\alpha$ in type $B_n$ corresponding to the only minuscule representation, the spin representation.
This means that $\bbG = \mathrm{Spin}(2n+1)$ with corresponding Lie algebra $\mathfrak{so}_{2n+1}$.
As representations of $\mathrm{Spin}(2n+1)$ and $\mathfrak{so}_{2n+1}$ are in one-to-one correspondence by differentiation, we shall decompose $S^2(V_{\lambda}^\vee)$ into irreducible $\so_{2n+1}$-modules and construct a basis of these irreducible modules.
The fundamental representations of $\mathfrak{so}_{2n+1}$ are wedge powers of the standard representation $\bigwedge^k(\CC^{2n+1})$ for $1 \leq k < n$ along with the spin representation $\Sp$.
This is the representation $V_{\lambda_{n}}$ with highest weight $\lambda_n = (\half,\, \half,\,\ldots,\, \half)$.
The spin representation $\Sp$ is of dimension $2^n$ and is the only minuscule representation, its single Weyl group of weights being $(\pm \half,\, \ldots,\, \pm \half)$. 

The outline of this subsection is as follows.
Recall from \Cref{lem:grassmann+embedding+equations} that the Lichtenstein embedding equations are a subspace of $S^2(\Sp^\vee)$ consisting of the irreducible $\so_{2n+1}$-modules whose highest weight is not equal to $2\lambda_n$.
To find these equations, we must first find an explicit decomposition of $S^2(\Sp^\vee)$ into irreducible $\so_{2n+1}$-modules.
We do this by considering $S^2(\Sp^\vee) \subset \Sp^\vee \otimes \Sp^\vee$, and showing that the latter module decomposes as
\begin{equation} \label{eq:spvee+iso}
\Sp^\vee \otimes \Sp^\vee \cong \Sp \otimes \Sp^\vee \cong \End(\Sp)   \cong \Cl^{\mathrm{even}} V \cong \twedge^{\mathrm{even}} V = \bigoplus_{k=0}^{n} \twedge^{2k} V\, .
\end{equation}
These four isomorphisms are given in \Cref{l:Spisotodual}, \Cref{l:end_iso_S_tensor_Svee}, \Cref{c:evenclifftoend} and \Cref{lem:quant+iso} respectively.
We then show which submodules in this decomposition correspond to $S^2(\Sp^\vee)$: this is the content of \Cref{p:Ssquaredecomp}.
Note that we will describe all of these isomorphisms very explicitly, as \Cref{App:B.2} will be focused on carefully pulling a basis of $S^2(\Sp^\vee)$ through each of these isomorphisms.

We begin by giving an explicit model for $\Sp$.
In this section we denote by~$V$ the vector representation $V_{\lambda_1} = \CC^{2n+1}$ of $\mathfrak{so}_{2n+1}$ with non-degenerate symmetric bilinear form $B$. Thus $\mathfrak{so}_{2n+1} = \mathfrak{so}_{2n+1}(V,B)$.
Following~\cite[Section 1]{Meinrenken:2013}, there exists a pair of maximal isotropic subspaces $F$, $F^*$  with respect to~$B$ and a unit length vector $w \in V$ orthogonal to $F$ and $F^*$ such that $V = F \oplus F^* \oplus \mathrm{Span}\{{w}\}$.
We have $B(w,w) =1$, and moreover we fix bases $\{v_1, \dots, v_n\}$ for $F$ and $\{x_1, \dots, x_n\}$ for $F^*$ such that $B(v_i, x_j) = \delta_{ij}$.

Let $F^\vee = \Hom(V,\mathbb{C})$ be the dual space of $V$ with basis $v_i^\vee \in F^\vee$ where $v_i^\vee(v_j) = \delta_{ij}$.
Note that we use $F^*$ to denote an isotropic complement to $F$ and we use $F^\vee$ to denote the dual space of $F$.
Though these are different spaces, we identify $F^*$ with $F^\vee$ by associating $x_i$ to $v_i^\vee$.
Explicitly, the form $B$ defines an isomorphism
\begin{equation} \label{eq:B+form+iso}
B^\sharp: V \rightarrow V^\vee \, , \quad B^\sharp(v)(u) = B(v,u) \, ,
\end{equation}
which also restricts to an isomorphism between $F^*$ and $F^\vee$.

We model $\Sp$ by $\bigwedge^\bullet F$ with inherited basis
\[
    \set{v_I \coloneqq v_{i_1} \wedge \cdots \wedge v_{i_k} }{ I = \{i_1, \dots, i_k\} \subseteq [n]} \, .
\]
Here we implicitly assume $I$ is ordered $i_1 < \cdots < i_k$ for this basis.
If instead we assume $I$ is ordered in some other way, say $I' = \sigma(I)$ for some permutation $\sigma$ of~$I$ as introduced in \Cref{sssec:Symmetric-group}, note that
\[
    v_{I'} \coloneqq v_{\sigma(I)_1} \wedge \cdots \wedge v_{\sigma(I)_k} = \sgn\left(\sigma\right) \cdot v_I = (-1)^{\lng{I'}} \cdot v_I\, .
\]
Similarly, we write the natural basis for $\bigwedge^\bullet F^*$ as
\[
    \set{x_I \coloneqq x_{i_1} \wedge \cdots \wedge x_{i_k} }{ I = \{i_1, \dots, i_k\} \subseteq [n]} \, .
\]

\begin{remark}\label{r:Bdetsharp}
We can extend the identification between $F^*$ and $F^\vee$ to an identification between $\bigwedge^\bullet F^*$ and $\Sp^\vee =\left(\bigwedge^\bullet F\right)^\vee$.
Explicitly, the form $B$ extends to a bilinear form on $\bigwedge^\bullet V$ called the \emph{determinant pairing}, defined as
\begin{align*}
    B_\det: \twedge^\bullet V \times \twedge^\bullet V \to \mathbb{C}\, , \quad
     B_\det (a_1 \wedge \ldots \wedge a_k,\, b_1 \wedge \ldots \wedge b_l) = \begin{cases} 
     \det(B(a_i,b_j)) & k = l\\
     0 & k \neq l
     \end{cases} \, .
\end{align*}
    The determinant pairing $B_\det$ induces an isomorphism between $\bigwedge^\bullet F^*$ and $\Sp^\vee =\left(\bigwedge^\bullet F\right)^\vee$ by 
\begin{align*}
    B_\det^\sharp : \twedge^\bullet F^* \to \left(\twedge^\bullet F\right)^\vee \, , \quad
B_\det^\sharp(x_I)(v_J) = B_\det(x_I,v_J).
\end{align*}
Note that $B_\det (x_I,v_J) = \delta_{I,J}$, here assuming that $I$ and $J$ are increasingly ordered. Thus $B_\det^\sharp(x_I) = (v_I)^\vee \in \Sp^\vee$.
\end{remark}

We now begin constructing the isomorphisms in \eqref{eq:spvee+iso} from right to left.
We will recall Clifford algebras and will recap the facts we need without proof: for further details see~\cite[Section 20.1]{Fulton+Harris:2004} or \cite{Meinrenken:2013}.
Let $\Cl(V)$ denote the Clifford algebra with a $\so_{2n+1}$-invariant non degenerate form $B: V \times V \to \CC$.
This is the tensor algebra $T^\bullet(V)$ modulo the two sided ideal $I(B)$ generated by all elements of the form $v \otimes u + u \otimes v - 2B(v,u)\cdot 1$.

\begin{lemma}[{\cite[Proposition 2.11]{Meinrenken:2013}}] \label{lem:quant+iso}
    Let $u_1, \dots, u_{2n+1}$ be any basis for $V$.
    The Clifford algebra $\Cl(V)$ is isomorphic to $\twedge^\bullet V$ as an $\mathfrak{so}_{2n+1}$-module via the \emph{quantisation} map
    \begin{equation}\label{eq:quant}
    q: \twedge^\bullet V \to \Cl(V) \, , \quad
    q( u_1 \wedge \ldots \wedge u_k) = \frac{1}{k!}\sum_{\sigma \in S_k} (-1)^{\lng{\sigma}} u_{\sigma(1)}  \ldots  u_{\sigma(k)}\, . 
\end{equation}
\end{lemma}
We note some useful properties of the quantisation map.
In our basis $\{x_1,\ldots,x_n,v_1,\ldots,v_n,w\}$, each $x_i$ is orthogonal to every element except $v_i$, hence $x_i$ anticommutes in the Clifford algebra with every basis element except $v_i$. Therefore the quantisation of $x_i \wedge v_j$ is equal to 
\begin{equation} \label{eq:quantisation+identity}
    q(x_i \wedge v_j) = \begin{cases} 
        x_iv_j & i \neq j\, , \\
        \half(x_iv_i -v_i x_i)& i = j\, .
    \end{cases}
\end{equation}
Using this rule and the fact that $q(x_I \wedge v_J) = x_Iv_J$ when $I \cap J = \varnothing$, this explicit description of the image of $q$ can be extended to any monomial in $\twedge^\bullet V$.

To extend \eqref{eq:quant} to an isomorphism of \eqref{eq:spvee+iso}, we note that the tensor algebra $T(V)$ has a $\ZZ$-grading by degree and $\ZZ_2$-grading by even and odd degrees.
The ideal defining $\Cl(V)$ is also $\mathbb{Z}_2$-graded, thus $\Cl(V)$ is $\mathbb{Z}_2$-graded.
We write $\Cl^{\mathrm{even}}(V)$ for the subalgebra consisting of elements of degree $\overline{0}$.
Similarly $\twedge^\bullet V$ is $\mathbb{Z}$-graded and also $\mathbb{Z}_2$-graded into odd and even degrees, so we let $\twedge^{\mathrm{even}} V$ be the subalgebra spanned by even degree elements.
As $q$ respects the $\ZZ_2$-grading, this restricts to an isomorphism of $\so_{2n+1}$-modules between $\Cl^{\mathrm{even}}(V)$ and $\twedge^{\mathrm{even}}V$.

To describe the next isomorphism, we need some additional notation.
Given some $x \in F^\vee$, the \defn{contraction} $\iota_x$ is an endomorphism of $\Sp = \twedge^\bullet F$ extended linearly from its definition on basis vectors as
\[ 
    \iota_x(v_1 \wedge v_2 \wedge \cdots \wedge v_k) = \sum_{i=1}^k (-1)^{i-1} x(v_i) v_1 \wedge \cdots \wedge \widehat{v_i} \wedge \cdots \wedge v_k
\]
where $\widehat{v_i}$ denotes omitting $v_i$ from the product.
For $v \in F$ we define the contraction $\iota_v: \Sp^\vee \to \Sp^\vee$ similarly.
Given some $v \in F$, we define the operator of left multiplication $\epsilon_v$ on $\Sp$ by
\[
    \epsilon_v(v_1 \wedge \cdots \wedge v_k) = v \wedge  v_1 \wedge \cdots \wedge v_k\, ,
\]
and extend linearly.
Similarly, for $x \in F^\vee$ we define $\epsilon_x: \Sp^\vee \to \Sp^\vee$ to be left multiplication by $x$. 

For ordered subsets $I$ and~$J$ of~$[n]$, we extend this notation to $x_I$ and $v_J$ to denote iterated contraction or left-multiplication, i.e.,
\[ \iota_{x_{I}} = \prod_{i \in I} \iota_{x_{i}}, \quad \epsilon_{v_J} = \prod_{j \in J} \epsilon_{v_j},\]
with the product taken with the order of $I$ and~$J$. 
We note the following useful identities  of these operations.

\begin{lemma} \label{lem:howepsandiotawork}
Let $A$ and $B$ be ordered sets such that $A \cap B = \varnothing$.
Then 
\[
    \epsilon_{v_A}(v_B) = v_A \wedge v_B = v_{A\cat B} = (-1)^{\lng{A\cat B}}v_{\ord{A \cat B}}
\]
where $\ord{A \cat B}$ denotes $A\cat B$ in increasing order.

Let $C$ and $D$ be ordered sets such that $C$ is included in $D$ as a set. Then 
\[
    \iota_{x_C}(v_D) = (-1)^{\lng{\rev{C} \cat (\remS{D}{C}})} v_{\remS{D}{C}}\, .
\]
\end{lemma}

\begin{proof}
Let $A=\{a_1,\ldots,a_k\}$ and $B=\{b_1,\ldots,b_l\}$ be ordered subsets of $[n]$. Then $v_A= v_{a_1} \wedge \ldots \wedge v_{a_k}$ and $v_{B} = v_{b_1} \wedge \ldots \wedge v_{b_\ell}$, thus 
\[
    \epsilon_{v_A}(v_B) = v_A \wedge v_B = v_{a_1} \wedge \ldots \wedge v_{a_k} \wedge v_{b_1} \wedge \ldots \wedge v_{b_\ell} = v_{A\cat B}\, .
\]
The first statement follows from the antisymmetric nature of the exterior product.

Let $C=\{c_1,\ldots,c_k\} \subset D=\{d_1,\ldots,d_\ell\}$ be increasing subsets of $[n]$.
Recall, $v_{\rev{C}} = v_{c_k} \wedge v_{c_{k-1}} \wedge \ldots \wedge v_{c_1}$.
Then
\[
    \iota_{x_C}(v_{\rev{C}}) = \iota_{x_{c_1}}  \iota_{x_{c_2}} \ldots  \iota_{x_{c_k}} (v_{c_k} \wedge \ldots \wedge x_{c_1}) = 1\, .
\]
Therefore, $\iota_{x_C}(v_{\rev{C}\cat (\remS{D}{C})}) = v_{\remS{D}{C}}$.
The second statement follows from the antisymmetric nature of the exterior product. 
\end{proof}

\begin{proposition}[{\cite[Lemma 20.16]{Fulton+Harris:2004}}] \label{p:cliffisend}
Let $V = F \oplus F^* \oplus \mathrm{Span}\{w\}$ and $\Sp = \bigwedge^\bullet F$ as above. The Clifford algebra $\Cl(V)$ is isomorphic to $\End(\Sp) \oplus \End(\Sp^\vee)$ as an $\so_{2n+1}$-module and an algebra, obtained by extending the map
\begin{align*}
    V = F \oplus F^* \oplus \mathrm{Span}\{ w \} &\rightarrow \End\left(\Sp\right) \oplus \End\left(\Sp^\vee\right)\, , \text{ where} \\ 
    v \in F &\mapsto \sqrt{2}\left( \epsilon_v \oplus \iota_v\right), \\
    x \in F^* &\mapsto \sqrt{2}\left(\iota_x \oplus \epsilon_x\right), \\
    w &\mapsto \left(v_I \mapsto (-1)^{|I|}v_I ,\, x_J \mapsto (-1)^{1+|J|} x_J\right).
\end{align*}
\end{proposition}

\begin{corollary}\label{c:evenclifftoend}
    The above isomorphism in Proposition \ref{p:cliffisend} induces an isomorphism 
    \[ \Cl^{\mathrm{even}}(V) \cong \End(\Sp),\]
    between the even subalgebra of the Clifford algebra and endomorphisms of the Spin module, as both an $\so_{2n+1}$-module and an algebra.
\end{corollary}

\begin{proof}
    Note that the isomorphism in Proposition \ref{p:cliffisend} defines a homomorphism 
    \[
        \Cl^{}(V) \to \End(\Sp), \quad v\in F \mapsto \sqrt{2}\epsilon_v \, , \, x\in F^* \mapsto \sqrt{2}\iota_x \, , \, w \mapsto (v_I \mapsto (-1)^{|I|}v_I).
    \]
    We then restrict this to the even subalgebra $\Cl^{\rm even}(V)$.
    It is well known that over the complex numbers, $\Cl^{\mathrm{even}}(V) \cong \Cl(W)$ where $\dim(W) = \dim(V) - 1$ \cite[pg. 79 (3.17)]{Meinrenken:2013}.
    Furthermore, every even rank Clifford algebra is isomorphic to an endomorphism algebra \cite[Lemma 20.9]{Fulton+Harris:2004} which has no non-trivial proper ideals.
    Therefore the kernel of the homomorphism above must be either $\{0\}$ or $\Cl^{\rm even}(V)$.
    Note that any product of two basis elements does not map to zero, thus the homomorphism is injective.
    Finally, both $\Cl^{\rm even}(V)$ and $\End(\Sp)$ are of dimension $2^{2n}$, hence the map is an isomorphism. 
\end{proof}

The element $v_{[n]} = v_1 \wedge \ldots \wedge v_n$ is (up to scalar) the unique element of top degree, i.e.\ degree~$n$, in $\twedge^\bullet F$.
For any element in $a$ in $\twedge^\bullet F$, let $a_{\rm [top]}$ denote the coefficient of $v_{[n]}$ in the expansion of $a$. Define a form on $\Sp$, 
\[
    (-,-)_\Sp: \Sp \times \Sp \to \mathbb{C}, \quad (v_I, v_J)_\Sp = (v_{\rev{I}} \wedge v_J)_{\rm [top]}.
\]
The following lemma  describes an $\mathfrak{so}_{2n+1}$ invariant form on $\Sp$. This form is also used in the proof of \cite[82, Theorem 3.12]{Meinrenken:2013}, but for a different model of $\Sp$, hence we reprove it for our model.

\begin{lemma}[$\mathfrak{so}_{2n+1}$-invariant forms on $\Sp = \twedge^\bullet F$] \label{l:nondegformonsp} 
The form 
\[
    (-,-)'_S: \Sp \times \Sp \to \mathbb{C}, \quad (v_I,v_J)'_S = (-1)^{n|J|} (v_I,v_J)_S
\]
 is a $\mathfrak{so}_{2n+1}$ invariant non-degenerate form on $\Sp$.
\end{lemma}
\begin{proof}
    This form is clearly non-degenerate. 
    To show it is $\mathfrak{so}_{2n+1}$-invariant, we check for each $X$ in a set of generators of $\mathfrak{so}_{2n+1}$ that 
    \begin{equation}
        \label{e:forminginv}
        (X\cdot \alpha,\beta)^{'}_S + (\alpha, X\cdot \beta)^{'}_S = 0 .
    \end{equation}
We need to check  \cref{e:forminginv} for $X \in \{\half(v_ix_i-x_iv_i), v_iv_j, wv_i, v_ix_j, x_ix_j,wx_i\}$. We do this case by case. 
We will let $\Par(x)$ for $x \in \ZZ$ denote the parity of $x$.
In particular, we make heavy use of the fact that $\Par( + x) = \Par( - x)$.

For $\alpha = v_{i_1} \wedge \ldots \wedge v_{i_k} \in \Sp$ we let $A = \{i_1, \ldots, i_k\}$ be its associated ordered set and similarly for $\beta$ we let $B$ be its associated ordered set.
For ease of notation, we will always let $E_L = (X\cdot \alpha, \beta) = x_L + \lngS{L} + n\order{B}$ and $E_R = (\alpha, X\cdot \beta) = x_R + \lngS{R} + n\order{B}$, where $x_L, x_R \in \ZZ$ and $\sigma_L, \sigma_R \in S_n$. 
We abbreviate by $A\cat i$ the concatenated ordered set $A\cat \{i\}$.
In all the cases below, we aim to show that $\Par(E_L) \neq \Par(E_R)$.
We now handle our six cases:
\begin{itemize}
    \item[Case 1] when $X=\half(v_ix_i-x_iv_i)$.
        We first note the following:
        \begin{align*}
            x_i v_i \alpha &= \begin{cases}
            \alpha &\text{ if } i \text{ does not appear in } \alpha\\
            0 & \text{otherwise} \, ;
            \end{cases}\\
            v_i x_i \alpha &= \begin{cases}
            \alpha &\text{ if } i \text{ appears in } \alpha\\
            0 & \text{otherwise} \, .
            \end{cases}
        \end{align*}
        Then without loss of generality we assume $i$ appears in $\alpha$ and not in $\beta$ giving
        \begin{align*}
            ((x_i v_i - v_i x_i) \alpha, \beta) &= (-\alpha, \beta)\text{ and}\\
            (\alpha, (x_i v_i - v_i x_i)\beta) &= (\alpha, \beta)\, ,
        \end{align*}
        and we get our desired result.
    \item[Case 2] when $X = v_iv_j$, $i \neq j$.
        Note that $X \cdot \alpha = 0$ whenever $i \in A$ or $j\in A$.
        In other words, we may assume $i, j \notin A \cup B$.
        Then 
        $
            E_L = (v_i v_j \alpha, \beta) = \lng{\rev{A} \cat j \cat i \cat B}  + n\order{B}
        $
        and $(\alpha, v_i v_j \beta) = \lng{\rev{A}\cat i \cat j \cat B}  + n\order{B}$.
        Then $\Par(E_L) \neq \Par(E_R)$ since the only difference is the transposition of $ij$ in the middle which changes the parity by $1$.
    \item[Case 3] when $X = wv_i$. 
        As before, we may assume that $i\notin A \cup B$.
        Then 
        \begin{align*}
            E_L &= (w v_i \alpha, \beta) = \order{A \cat i} + \lng{\rev{A} \cat  i\cat  B}  + n\order{B}\text{ and }\\
            E_R &= (\alpha, w v_i \beta) = \order{B\cat i} + \lng{\rev{A} \cat i \cat  B} + n\order{B\cat i}\, .
        \end{align*}
        As the lengths are the same, it suffices to look at the other two terms.
        Then notice that $\order{A \cat i} + n \order{B}$ and $\order{B \cat i} + n \order{B\cat i}$ must have different parity:
        \begin{align*}
            \order{A\cat i} + n \order{B} &= n + (n-1)\order{B}\text{ and }\\
            \order{B\cat i} + n \order{B\cat i} &= (n + 1) \order{B\cat i} = n + 1 + (n+1)\order{B}\, .
        \end{align*}
        But $\Par(n-1) = \Par(n+1)$ and we have $\Par(E_L) \neq \Par(E_R)$ as desired.

    \item[Case 4] when $X = v_ix_j$, $i \neq j$.
        As before, we may assume that $j\notin A \cup B$ and that $i \in A \cap B$. Then
        \begin{align*}
            E_L &= (v_i x_j \alpha, \beta) = A(j) - 1 + \lng{\rev{\left(\remS{A}{j}\right)} \cat i \cat B}  + n\order{B}\text{ and}\\
            E_R &= (\alpha, v_i x_j \beta) = B(j) - 1 + \lng{\rev{A} \cat i \cat  \left(\remS{B}{j}\right)} + n\order{B}\, .
        \end{align*}
        Since $\Par(\lng{A \cat i}) = \Par(\order{A} + \lng{A})$ we have
        \begin{align*}
           \Par(E_L) &= \Par\left(A(j) - 1 + \lng{A} + B(j) - 1 + \lng{{\rev{\left(\remS{A}{j}\right)}} \cat i \cat \left(\remS{B}{j}\right)}  + n\order{B}\right)\text{ and}\\
           \Par(E_R) &= \Par\left(B(j) - 1 + \lng{A} - A(j) + \lng{{\rev{\left(\remS{A}{j}\right)}} \cat i\cat  \left(\remS{B}{j}\right)} + n\order{B}\right)
        \end{align*}
        and our inequality holds as desired.
    \item[Case 5] when $X = x_ix_j$.
        As before, we assume that $i, j \in A \cap B$.
        We may assume that $i < j$ in both cases, with the $j < i$ case having a similar proof.
        Then
        \begin{align*}
            E_L &= (x_i x_j \alpha, \beta) = A(i) - 1 + A(j) - 1 + \lng{\rev{\left(\rem{A}{i, j}\right)} \cat  B}  + n\order{B}\text{ and}\\
            E_R &= (\alpha, x_i x_j \beta) = B(i) - 1  + B(j) - 1 + \lng{\rev{A} \cat \left(\rem{B}{i, j}\right)} + n\order{\rem{B}{i, j}}\, .
        \end{align*}
        Since $\Par\left(\lng{A \cat i}\right) = \Par\left(\order{A} + \lng{A}\right)$ we have
        \begin{align*}
           \Par(E_L) &= \Par\left(
                \begin{array}{@{}c@{}}
                    A(i) + A(j) - 2 + B(i) + B(j) - 2 + 2(\order{A} - 2)\\
                    + \lng{\rev{\left(\rem{A}{i,j}\right)} \cat \left(\rem{B}{i, j}\right)} + n\order{B}
                \end{array}
            \right)\\
            &= \Par\left(
                \begin{array}{@{}c@{}}
                    A(i) + A(j) + B(i) + B(j) + 2\order{A}  - 8\\
                    + \lng{\rev{\left(\rem{A}{i,j}\right)} \cat \left(\rem{B}{i, j}\right)}  + n\order{B}
                \end{array}
            \right)\text{ and}\\
           \Par(E_R) &= \Par\left(
                \begin{array}{@{}c@{}}
                    B(i)  + B(j) - 2 + \left(\order{A} - 1 - A(i)\right) + \left(\order{A} - 2 - A(j)\right)\\
                    + \lng{\rev{\left(\rem{A}{i,j}\right)} \cat \left(\rem{B}{i, j}\right)} + n\order{B^{(i, j)}}
                \end{array}
            \right)\\
           &= \Par\left(
                \begin{array}{@{}c@{}}
                    -A(i) - A(j) + B(i) + B(j) + 2 \order{A} - 5\\
                    + \lng{\rev{\left(\rem{A}{i,j}\right)} \cat \left(\rem{B}{i, j}\right)} + n\order{\rem{B}{i, j}}
                \end{array}
            \right)
        \end{align*}
        and our inequality holds as desired.
    \item[Case 6] when $X = wx_i$. 
        As before, we may assume that $i\in A \cap B$.
        Then
        \begin{align*}
            E_L &= (w x_i \alpha, \beta) = A(i) - 1 + \order{A} - 1 + \lng{\rev{\left(\remS{A}{i}\right)} \cat B }  + n\order{B}\text{ and}\\
            E_R &= (\alpha, w x_i \beta) = B(i) - 1  + \order{B} - 1 + \lng{\rev{A} \cat \left(\remS{B}{i}\right)} + n\order{\remS{B}{i}}\, .
        \end{align*}
        Since $\Par(\lng{A \cat i}) = \Par(\order{A} + \lng{A})$ we have
        \begin{align*}
           \Par(E_L) &= \Par\left(A(i) + \order{A} - 2 + B(i) -1 + \order{A} - 1 + \lng{\rev{\left(\remS{A}{i}\right)} \cat \left(\remS{B}{i}\right)} + n\order{B}\right)\text{ and}\\
           \Par(E_R) &= \Par\left(B(i) + \order{B} - 2 + \order{A} - A(i) + \lng{\rev{\left(\remS{A}{i}\right)} \cat \left(\remS{B}{i}\right)} + n\order{\remS{B}{i}}\right)\, .
        \end{align*}
        As most terms are the same, it suffices to show that the parity of $\order{A}  + n\order{B}$ is different to the parity of $\order{B} + n\order{B^{(i)}}$. These are
        \begin{align*}
            \order{A}  + n\order{B} &= n + (n-1)\order{B} + 1\text{ and}\\
            \order{B}  + n\order{\rem{B}{i}} &= \order{B} + n(\order{B} - 1) = (n + 1)\order{B} - n
        \end{align*}
        as $\order{A} + \order{B} = n + 1$. \qedhere
\end{itemize}
\end{proof}

For the next isomorphism, recall the notation $I^c = [n] \setminus I$ for the complement of the set $I$.
\begin{corollary}\label{l:Spisotodual}
   Recall the models $\Sp = \bigwedge^\bullet F$ and $\Sp^\vee = \bigwedge^\bullet F^*$.
    Then $\Sp$ and $\Sp^\vee$ are isomorphic as $\mathfrak{so}_{2n+1}$ modules via the map
    \begin{equation}\label{eq:spintodualmap}
    \Sp \rightarrow  \Sp^\vee \, , \quad  v_I \mapsto (-1)^{n|I^c| + \lng{\rev{I} \cat I^c}}x_{I^c} .
    \end{equation}
\end{corollary}

\begin{proof}
\Cref{l:nondegformonsp} proves that $(-,-)^{'}_S$ is an $\mathfrak{so}_{2n+1}$-invariant non-degenerate form on $\Sp$.
Analogously to \eqref{eq:B+form+iso}, this form induces an isomorphism 
\begin{equation*}
    \Sp \rightarrow \Sp^\vee, \quad v \mapsto (v,-)^{'}_S.
\end{equation*}
Note that $(v_I,-)^{'}_S$ takes the value zero on all basis vectors aside from $v_{I^c}$ where by definition it takes the value $(-1)^{n|I^c| + \lng{\rev{I} \cat I^c}}$.
Hence the linear map \eqref{eq:spintodualmap} is precisely this isomorphism.
\end{proof}

\begin{lemma} 
\label{l:end_iso_S_tensor_Svee}
    The algebra $\End(\Sp)$ is isomorphic to $\Sp \otimes \Sp^\vee$. 
\end{lemma} 
\begin{proof} 
It is a general fact about vector spaces that if $U$ is a vector space with basis ${b_i}$ and dual basis $d_i \in U^\vee$, then $\End(U) \cong U \otimes U^\vee$ via the map
\[ 
\phi \mapsto \sum_{i,j}d_i(\phi(b_j)) b_i \otimes d_j \, .
\]
By \Cref{r:Bdetsharp} and its preceding discussion, we have a basis $\set{v_J}{J \subseteq [n]}$ for $\Sp$ and an associated dual basis $\set{x_J}{J \subseteq [n]}$ for $\Sp^\vee$.
Hence, the explicit isomorphism is given by 
\begin{align}\label{eq:homtoend}
\End(\Sp) \rightarrow \Sp \otimes \Sp^\vee \, , \quad \phi \mapsto \sum_{I,J \subseteq [n]} x_{I}(\phi(v_J)) v_I \otimes x_{J} \, .
\end{align}
\end{proof}

We now have all the isomorphisms to prove the main result of this section:
\begin{proposition}\label{p:Ssquaredecomp}
The following are isomorphic as $\mathfrak{so}_{2n+1}$ modules:
\[
    \Sp^\vee \otimes \Sp^\vee \cong \Sp \otimes \Sp^\vee \cong \End(\Sp)   \cong \Cl^{\mathrm{even}} V \cong \twedge^{\mathrm{even}}V = \bigoplus_{k=0}^{n} \twedge^{2k} V\, .
\] 
Viewing $S^2(\Sp^\vee)$ as a submodule of $\Sp^\vee \otimes \Sp^\vee$, 
this induces the following isomorphism:
\[
    S^2 (\Sp^\vee) \cong \begin{dcases}
        \bigoplus_{k=0}^{\floor{\frac{2n+1}{4}}} \twedge^{4k} V & n = 3 , 4 \bmod 4\, , \\
        \bigoplus_{k}^{\floor{\frac{2n-1}{4}}} \twedge^{4k+2} V & n = 1,2 \bmod 4\, .\end{dcases}
\] 
\end{proposition}

\begin{proof}
The first isomorphism is described in Lemma \ref{l:Spisotodual}, the second by \Cref{l:end_iso_S_tensor_Svee}, the third by Corollary \ref{c:evenclifftoend} and the fourth by \Cref{lem:quant+iso} restricted to the even submodules.

To understand which of the submodules of $\twedge^{\mathrm{even}} V$ are isomorphic to $S^2(\Sp)$, 
we consider the involution $\tau:\Sp^\vee \otimes \Sp^\vee \to \Sp^\vee \otimes \Sp^\vee$ defined by $\tau(x_I \otimes x_J) = x_J \otimes x_I$.
Observe that $S^2(\Sp^\vee)$ is the $+1$-eigenspace of $\tau$.
By composing $\tau$ with the relevant isomorphisms, we will derive involutions on $\Sp^\vee \otimes \Sp^\vee$, $\End(\Sp)$, $\Cl^{\mathrm{even}}(V)$, and $\twedge^{\mathrm{even}} V$ whose $+1$-eigenspaces are isomorphic to $S^2(\Sp)$.
By abuse of notation, we will also denote these involutions by $\tau$.

We first track the element $x_I \otimes x_J$ through our chain of isomorphisms.
Under the inverse of the isomorphism defined in Lemma \ref{l:Spisotodual}, we find $x_I \otimes x_J$ gets mapped to 
\begin{equation} \label{eq:tau+first}
    (-1)^{n|I^c| + \lng{\rev{I} \cat I^c}} v_{I^c} \otimes x_J \in \Sp \otimes \Sp^\vee \, .
\end{equation}
To see where this is mapped to in $\End(\Sp)$ under the second isomorphism in Lemma \ref{l:end_iso_S_tensor_Svee}, observe that from Lemma \ref{lem:howepsandiotawork} we have $\iota_{x_{I}}\epsilon_{v_{[n]}}\iota_{x_{\rev{J}}}(v_J) = (-1)^{\lng{\rev{I} \cat I^c}}v_{I^c}$.
It follows that \eqref{eq:tau+first} gets mapped to 
\begin{equation} \label{eq:tau+second}
  (-1)^{n|I^c|}\iota_{x_{I}}\epsilon_{v_{[n]}}\iota_{x_{\rev{J}}} =(-1)^{n|I^c| + \binom{|J|}{2}}\iota_{x_{I}}\epsilon_{v_{[n]}}\iota_{x_J} \in \End(\Sp) \, .
\end{equation}
Finally, under the inverse of the map described in Corollary \ref{c:evenclifftoend}, we find \eqref{eq:tau+second} is mapped to the following element in $\Cl^{\mathrm{even}}(V)$:
\begin{align*}
     &\phantom=\begin{cases} (-1)^{n|I^c| + {\binom{|J|}{2}}}x_Iv_{[n]}x_J & n + |I| + |J| \equiv 0 \mod 2 \\ (-1)^{(n+1)|I^c| + {\binom{|J|}{2}}}wx_Iv_{[n]}x_J  & n + |I| + |J| \equiv 1 \mod 2\end{cases}
    \\
    &= (-1)^{(|I|+|J|)|I^c| + {\binom{|J|}{2}}}w^{|I|+|J|+n}x_Iv_{[n]}x_J\, .
\end{align*}

Recall that $\tau(x_I \otimes x_J) = x_J \otimes x_I$ in $\Sp^\vee \otimes \Sp^\vee$.
Hence $\tau$ induces an involution on $\Cl(V)$ by
\[
    \tau\left((-1)^{(|I|+|J|)|I^c| + {\binom{|J|}{2}}}w^{|I|+|J|+n}x_Iv_{[n]}x_J\right) = (-1)^{(|I|+|J|)|J^c| + {\binom{|I|}{2}}}w^{|I|+|J|+n}x_Jv_{[n]}x_I .
\]
Thus 
\begin{equation}\label{eq:taueq}
    \tau \left(w^{|I|+|J|+n}x_Iv_{[n]}x_J\right) = (-1)^{(|I|+|J|)(|I^c|+|J^c|) + {\binom{|I|}{2}}+{\binom{|J|}{2}}}w^{|I|+|J|+n}x_Jv_{[n]}x_I.
\end{equation} 

We next compare $\tau$ with another involution $\bullet^t$ on the Clifford algebra $\Cl(V)$, often called transposition.
This is defined as the unique antiautomorphism extending the identity on $V$. Explicitly 
\[
    v^t = v \quad \text{ for all } v \in V \quad \text{ and } \quad (ab)^t = b^ta^t \quad \text{ for all } a, b \in \Cl(V) \, .
\]
Let us record that 
\begin{align}\label{eq:smalltrans1}
    (x_{I}v_{[n]}x_J)^t &= (x_J)^tv_{[n]}^tx_{I}^t = (-1)^{{\binom{|J|}{2}} + {\binom{n}{2}} + {\binom{|I|}{2}}} x_J v_{[n]}x_{I}\, , \\
\label{eq:smalltrans2}
    (wx_{I}v_{[n]}x_J)^t &= (x_J)^tv_{[n]}^tx_{I}^tw = (-1)^{\binom{|J|}{2} + \binom{n}{2} + \binom{|I|}{2} + |I|+|J|+n} wx_J v_{[n]}x_{I}\, .
\end{align}
As with $\tau$, we can view $\bullet^t$ as an involution of $\twedge^\bullet V$ by passing through the (inverse of the) quantisation map.
By an abuse of notation, we denote the involution on both spaces by $\bullet^t$.
It is shown in \cite[Section 2.2.6]{Meinrenken:2013} that quantisation intertwines with $\bullet^t$, i.e., $q(a^t) = q(a)^t$ for all $a \in \twedge^\bullet V$.
It follows that $\bullet^t$ acts by $a^t = (-1)^{\binom{d}{2}} a$ for $a \in \twedge^d V$, hence $\twedge^d V$ is in the $+1$ eigenspace of $\bullet^t$ when $d \equiv 0, 1 \mod 4$ and in the $-1$ eigenspace when $d \equiv 2,3 \mod 4$.

From Equations \eqref{eq:smalltrans1} and \eqref{eq:smalltrans2} we get
\begin{equation}\label{eq:smalltrans3}(w^{|I|+|J|+n}x_{I}v_{[n]}x_J)^t= (-1)^{{\binom{|J|}{2}} + {\binom{n}{2}} + {\binom{|I|}{2}} + (|I|+|J|+n)(|I|+|J|+n)} w^{|I|+|J|+n}x_J v_{[n]}x_{I}\, .\end{equation}
Comparing \eqref{eq:taueq} to \eqref{eq:smalltrans3} and noting that  $|I|+|J| \equiv\ |I^c| + |J^c| \mod 2$ and $(m+n)^2 \equiv m^2 + n^2\mod 2$,  we see that 
\[  \tau (w^{|I|+|J|+n}x_Iv_{[n]}x_J)  =  (-1)^{n^2 + {\binom{n}{2}}}(w^{|I|+|J|+n}x_Iv_{[n]}x_J)^t.\]

Since $n^2 + {\binom{n}{2}}$ and ${\binom{n+1}{2}}$ are equal modulo $2$, 
the involutions $\bullet^t$ and $\tau$ differ by $(-1)^{\binom{n+1}{2}}$ and hence are equal when $n \equiv 3,4 \mod 4$ and differ by a sign when $n \equiv 1,2 \mod 4$.
Viewed as an involution of $\twedge^{\rm even} V$, the $+1$ eigenspace of $\bullet^t$ is $\bigoplus \twedge^{4k} V$ and the $-1$ eigenspace is $\bigoplus \twedge^{4k+2} V$.
Hence when $n \equiv 3,4 \mod 4$ we find $S^2(\Sp) \cong \bigoplus \twedge^{4k} V$ and when $n \equiv 1,2\mod 4$, the $+1$ eigenspace of $\tau$ corresponds to the $-1$ eigenspace of $\bullet^t$, hence $S^2(\Sp) \cong \bigoplus \twedge^{4k+2} V$.
\end{proof}

\subsection{Explicit equations: Type $B_n$}\label{App:B.2}

This section is divided into two parts.
In the previous section we obtained an explicit description of the decomposition of $S^2(\Sp^\vee)$ into irreducible $\mathfrak{so}_{2n+1}$-modules via the four isomorphisms described in Lemma \ref{l:Spisotodual}, \Cref{l:end_iso_S_tensor_Svee}, Corollary \ref{c:evenclifftoend} and \Cref{lem:quant+iso}. 
\Cref{App:BspinEqs} is dedicated to the slightly tedious process of carefully pulling a basis for $\twedge^\bullet V$ back through the three isomorphisms,
culminating in \Cref{t:BspinEqs} below.
Then, in \Cref{App:BnSpinmain}, we prove \Cref{t:BnSpinmain} by confirming that its equations have the same span as those in \Cref{t:BspinEqs}.

\subsubsection{A first basis of equations}\label{App:BspinEqs}

We prove \Cref{t:BspinEqs} in this section.
Recall the notation $\lng{-}$ for length of ordered multisets from \Cref{sssec:Symmetric-group}.
We let $\cM_{N,L}$ denote the set
\begin{equation}\label{eq:M_NL}
    \cM_{N,L} \coloneqq \set{4k + 2 \ceil{\frac{n}{2}} - |N| - |L| - \varepsilon}{k\in\ZZ \, , \, \varepsilon \in \{0,1\}} \setminus \Big\{ m-2,\,m-1,\,m,\,m+1,\,m+2\Big\}
\end{equation}
where $m= n -|N|-|L|$.

For disjoint $N, L \subseteq [n]$, define the set $Q_{N,L} = \set{ K \subset [n]}{N \subset K, L \cap K = \emptyset}$  of dimension $n - |N| - |L|$. 
This is the set of subsets indexing a face of the $n$-cube $Q_n$; we'll generally speak of it as a face or a (sub)cube itself.
The full $n$-cube $Q_n$ is given by $Q_{\emptyset,\emptyset}$.
Given a set $I$, write $I^c = [n] \setminus I$ for the complement.
Any pair $(I,J)$ of subsets of~$[n]$ is an antipode in a unique cube, namely $Q_{N,L}$, where $N = I \cap J$ and $L = (I \cup J)^c$ (this may be a $0$-cube).
So we will call such a pair of subsets an \defn{ordered antipode}.
Observe that we can write $J = (I^c \setminus L) \cup N$.
As such, for any $K \in Q_{N,L}$ we let $\overline{K}_{N,L} := (K^c \setminus L) \cup N$ denote the set such that $(K, \overline{K}_{N,L})$ is the unique antipode in $Q_{N,L}$ containing $K$.
If $N$ and $L$ are clear from context we will drop them and simplify $\overline{K}_{N,L}$ to $\overline{K}$. Letting $S = (N \cup L)^c$, we have $\overline{K}_{N,L} = K \symDiff S$.

\begin{restatable}{theorem}{thBspineqs}\label{t:BspinEqs}
    The set of equations 
    \[
    \set{p^M_{N,L} \in S^2(\Sp^\vee)}{M,N,L \text{ pairwise disjoint} \, , \, 2|M| \in \cM_{N,L}}\]
    where
    \[
        p^M_{N,L} = \sum_{\substack{(K ,\overline{K}) \\\text{ antipodes in } Q_{N,L}}} (-1)^{|K^c \cap M| + |\overline{K}^c|(|N|+|L|) + \lng{\rev{\overline{K}^c} \cat \overline{K}} + n|\overline{K}|+\lng{L \cat \left(\remS{K}{N}\right)} + \lng{\rev{N} \cat \remS{K}{N}}}x_{K}x_{{\overline{K}}}  
    \]
     forms a basis for the Lichtenstein embedding equations.
\end{restatable}


\begin{remark}
It will be shown in the proof of \Cref{l:leadingquadeqisingrassmmann} that due to the mod~$4$ conditions, $2|M|$ cannot be $m-2$ or $m+2$. Thus it makes no difference to the theorem whether those are listed when defining $\cM_{N,L}$.
However, explicitly removing $\{m-2,m-1,m,m+1,m+2\}$ from $\cM_{N,L}$ clearly demonstrates that that there are no equations supported on antipodes of $0$-, $1$- and $2$-cubes.
\end{remark}

We emphasise that this theorem is not new, but rather just a calculation from a result of Lichtenstein \cite{Lichtenstein:1982}.
However, we include it in this appendix as we could not find an explicit basis of polynomials spanning $S^2(V_\lambda) / V_{2 \lambda}$ in the literature. 

To prove \Cref{t:BspinEqs}, it is sufficient to prove that the set of quadrics $\{ p^M_{N,L}\}$ span  the summand $S^2(\Sp^\vee) / V_{2\lambda}$. 
Note that the Casimir $\Omega$ has different eigenvalues on each irreducible summand of $S^2(\Sp^\vee)$.
The lemma below calculates each of these eigenvalues, producing a basis for $S^2(\Sp^\vee)$ which splits into bases for each irreducible summand.
We then finish the proof of Theorem \ref{t:BspinEqs} by deducing which subset of this basis is a basis for $S^2(\Sp^\vee) /V_{2\lambda}$. 

\begin{definition}\label{d:antipodesigns}
    Given an ordered antipode $(K, \overline{K})$, we define
    \[
        [K,\overline{K}] := |\overline{K}^c|\left(|N|+|L|\right) + \lng{\rev{\overline{K}^c} \cat \overline{K}} + n|\overline{K}|+\lng{L \cat \left(\remS{K}{N}\right)} + \lng{\rev{N} \cat \remS{K}{N}}\, . 
    \]
    Recall that $N$ and $L$ are determined by $(K,\overline{K})$.
\end{definition}

\begin{lemma}\label{l:allpolysBn}
The set
\begin{align}\label{e:cliffbasisSsq2} 
   \begin{split}
        \Bigg\{  \sum_{\substack{(K,\overline{K})\\\text{ antipode of } Q_{N,L}}} (-1)^{|K^c \cap M| + [K, \overline{K}]}x_{K}  x_{\overline{K}}\st   &L,M,N \text{ pairwise disjoint,}\\
    &|N|+2|M| +|L| \equiv 2\ceil{ \frac{n}{2}}, 2\ceil{ \frac{n}{2}}-1 \mod 4\Bigg\} 
    \end{split}
\end{align}
is a basis of $\Omega$-eigenvectors for $S^2(\Sp^\vee)$.

\end{lemma}

\begin{proof}
\Cref{p:Ssquaredecomp} gives an explicit isomorphism between $S^2(\Sp^\vee)$ and an appropriate submodule of $\twedge^\bullet V$ using the four isomorphisms described in Lemma \ref{l:Spisotodual}, \Cref{l:end_iso_S_tensor_Svee}, Corollary \ref{c:evenclifftoend} and \Cref{lem:quant+iso}. 
We choose a basis of this submodule and then chase this basis through the four isomorphisms to give an explicit basis for $S^2(\Sp^\vee)$.

Recall $V = F \oplus F^* \oplus \mathrm{span}(w)$, and let $v_i$ be a basis of $F$, $u_j$ a basis of $F^*$, and $x_i$ a basis of $F^\vee$. Here we introduce $u_j$ to distinguish $F^*$ from $F^\vee$ for clarity.
Consider a general pure tensor in $\bigwedge^{2l}(V)$, where $2l = 4k$ if $n \equiv 3,4 \mod 4$ and $2l=4k+2$ if $n \equiv 1,2 \mod 4$. 
A general pure tensor is either of the form $z = v_I \wedge u_J$ or $z'=v_{I'} \wedge u_{J'} \wedge w$ where $|I| + |J| = 2l$ or $|I'| + |J'| = 2l-1$.
The quantisation of $z$ and $z'$ is 
\begin{align}\label{eq:qz} 
    q(z) &={\left(\half\right)^{|I\cap J|}(-1)^{\bullet}} \prod_{s \in I \cap J} (v_su_s-u_sv_s) \prod_{i \in I \setminus J} v_i \prod_{j \in J \setminus I} u_j. \\
\label{eq:qz'}
    q(z') &={\left(\half\right)^{|I'\cap J'|}(-1)^{\bullet} } w\prod_{s \in I' \cap J'} (v_su_s-u_sv_s) \prod_{i \in I' \setminus J'} v_i \prod_{j \in J' \setminus I'} u_j  .
\end{align}
From now on we drop the global factor $(\half)^{|I\cap J|}(-1)^{\bullet}$ and its equivalent for $I'$ and $J'$ as this will only affect the resulting basis up to scalars.
 
We now relabel using pairwise disjoint sets $L,M,N \subset [n]$ (resp.\ $L',M',N' \subset [n]$) where $L = I \setminus J$, $M = I \cap J$ and $N = J \setminus I$ (resp.\ $L' = I' \setminus J'$, $M = I' \cap J'$ and $N' = J' \setminus I'$).
The conditions on the cardinality of $I,J,I',J'$ imply that $|L| + 2|M| + |N| = 2l$ and $|L'| + 2|M'| + |N'| = 2l-1$, which can be further refined to
\begin{equation}
    \label{eq:NLcond}
    |L| + 2|M| + |N| = |L'| + 2|M'| + |N'| + 1 =
    \begin{cases}
         4k & \text{if }n \equiv 3,4 \bmod 4\, , \\
         4k+2 & \text{if }n \equiv 1,2 \bmod 4\, .
    \end{cases}
\end{equation}
We rewrite the elements $q(z)$ and $q(z')$ from Equations \eqref{eq:qz} and \eqref{eq:qz'} in $\End(\Sp)$: 
\begin{align*}
    \left(\prod_{i \in L}v_i\right)(v_K) &= \epsilon_{v_L}(v_K)\, , &
    \left(\prod_{j \in N}u_j\right)(v_K) &= \iota_{x_N}(v_K)\, , \\
    w(v_K) &= (-1)^{|K|}v_K\, , &
    \left(\prod_{s \in M}v_su_s -u_sv_s\right)(v_K) &= (-1)^{|K^c \cap M|}v_K\, .
\end{align*}
Thus we can rewrite $q(z)$ and $q(z')$ as products of elements in $\End(\Sp)$ to give a basis for $S^2(\Sp^\vee)$ inside $\End(\Sp)$:
\begin{align*}
    q(z) (v_K) &=(-1)^{\order*{((K\setminus N)\cup L)^c \cap M}} \epsilon_{v_L} \iota_{x_N}(v_K)\, ,\\
    q(z')(v_K) &= (-1)^{\order*{((K\setminus N')\cup L')^c \cap M'} + \order*{(K\setminus N')\cup L')}} \epsilon_{v_{L'}} \iota_{x_{N'}}(v_K)\, .
\end{align*}
Note that as $M$ is disjoint from $N$ and $L$, we have $((K\setminus N) \cup L)^c \cap M = K^c \cap M$.
Using \Cref{lem:howepsandiotawork}, we can conclude that 
\begin{align*}
    q(z)v_K &= \begin{cases}
            (-1)^{|K^c \cap M| + \lng{L \cat (K \setminus N)}   + \lng{\rev{N} \cat (\remS{K}{N})}}v_{(K \setminus N) \cup L} & N \subset K, L \cap K = \varnothing,\\
            0 & \text{ otherwise; and}
        \end{cases} \\
    q(z')v_K &= \begin{cases}
                (-1)^{|(K \setminus N') \cap L'| + |K^c \cap M'| + \lng{L' \cat (K \setminus N')} + \lng{\rev{N'} \cat (\remS{K}{N'})}} v_{(K \setminus N') \cup L'} & N' \subset K, L' \cap K = \varnothing,\\
                0 & \text{ otherwise.}
            \end{cases}
\end{align*}
Note that the definition of $\iota_{x_{L}}$ and $\epsilon_{v_{N}}$ ensures that that output is non-zero only when $L \cap K = \emptyset$ and $N \subseteq K$.

Now, we need to understand $q(z) \in \End(\Sp)$ in terms of the basis $v_I \otimes x_{J}$ of $\Sp \otimes \Sp^\vee$.
Via the isomorphism in \Cref{l:end_iso_S_tensor_Svee}, we can write $q(z)$ as
\begin{equation} \label{eq:qz+sp+spvee}
    \sum_{K \in Q_{N,L}} (-1)^{|K^c \cap M| + \lng{L \cat (K \setminus N)}   + \lng{\rev{N} \cat (\remS{K}{N})}}v_{(K \setminus N )\cup L} \otimes x_{K} \in \Sp \otimes \Sp^\vee \, ,
\end{equation}
and similarly for $q(z')$.
Mapping \eqref{eq:qz+sp+spvee} through the isomorphism in \Cref{l:Spisotodual}
we obtain
\begin{align*} \label{eq:finalz}
q(z) &\mapsto \sum_{K \in Q_{N,L}} (-1)^{|K^c \cap M| + \lng{\rev{J} \cat J^c} + n|J^c|+\lng{L \cat (K \setminus N)} + \lng{\rev{N} \cat (\remS{K}{N})}}x_{{J^c}} \otimes x_{K} \in \Sp^\vee \otimes \Sp^\vee, \\
    q(z') &\mapsto \sum_{K \in Q_{N',L'}} (-1)^{|K^c \cap M'| + |J'| + \lng{\rev{J'} \cat J'^c} + n|J'^c|+ \lng{L' \cat (K \setminus N')} + \lng{\rev{N'} \cat (\remS{K}{N'})} }x_{{J'^c}} \otimes x_{K} \in \Sp^\vee \otimes \Sp^\vee
\end{align*}
where $J = (K\setminus N) \cup L$, and $J' = (K\setminus N') \cup L'$,
subject to the conditions: 
\[
|L| + 2|M| + |N| = |L'| + 2|M'| + |N'| + 1 =
    \begin{cases}
         4k & \text{if }n \equiv 3,4 \bmod 4\, , \\
         4k+2 & \text{if }n \equiv 1,2 \bmod 4\, .
    \end{cases}
\]
One can combine this into a single set of equations
\begin{align}\label{eq:Bnbedding}
   \begin{split}
    \Bigg\{  \sum_{K \in Q_{N,L}} &(-1)^{|K^c \cap M| + |J|(|N|+|L|) + \lng{\rev{J} \cat J^c} + n|J^c|+\lng{L \cat (K \setminus N)}   + \lng{\rev{N} \cat (\remS{K}{N})} }x_{{J^c}} \otimes x_{K}\\
    & \st L,M,N \text{ pairwise disjoint, }
    |N|+2|M| +|L| \equiv 2\ceil{ \frac{n}{2}}, 2\ceil{ \frac{n}{2}}-1 \mod 4\Bigg\} 
    \end{split}
\end{align}
where $J = (K \setminus N) \cup L$.
By construction, \eqref{eq:Bnbedding} is a basis of $S^2(V_\lambda)$ consisting of eigenvectors of the Casimir.
Note that $J^c = (K^c \setminus L) \cup N = \overline{K}$, hence our equations are supported on antipodes of $Q_{N,L}$, namely $x_K \otimes x_{\overline{K}}$ and $x_{\overline{K}} \otimes x_K$.
As \eqref{eq:Bnbedding} is a basis of the symmetric square $S^2(\Sp^\vee)$, 
we can sum only over the antipodes of $Q_{N,L}$ and replace $x_{\overline{K}} \otimes x_K$ with the quadratic monomial $x_Kx_{\overline{K}}$.
This finishes the proof.
\end{proof}

\begin{proof}[Proof of \Cref{t:BspinEqs}]
Lemma \ref{l:allpolysBn} gives a basis of $\Omega$-eigenvectors for $S^2(\Sp^\vee)$, and this basis splits into a basis for each irreducible submodule.
To prove Theorem \ref{t:BspinEqs}, it only remains to remove the basis elements that span $V_{2\lambda}$.
Note that $2\lambda = (1,\ldots,1)$ which is the highest weight for $\twedge^n V \cong \twedge^{n+1} V$.
One can see this by noticing that all of the weights of $V$ are $(0,\dots,\pm 1,\dots,0)$ or $(0,\ldots,0)$ with multiplicity one.
The weights of $\twedge^k V$ are all sums of $k$ different weights of $V$.
Notice, one must use either $n$ or $n+1$ weights from $V$ to sum to $2\lambda$.
Therefore, we must remove the basis elements given in Lemma \ref{l:allpolysBn} that originate from $\twedge^n V $ or $\twedge^{n+1} V$, only one of which can occur for each n.
The $z$ and $z'$ that originate in $\twedge^n V$ are such that $|N| + 2|M| + |L| = n$ and $|N'| + 2|M'| + |L'| = n-1$.
Likewise, $z$ and $z'$ that originate in $\twedge^{n+1} V$ are such that $|N| + 2|M| + |L| = n+1$ and $|N'| + 2|M'| + |L'| = n$.
Our basis of $S^2(\Sp)$ constructed in Lemma \ref{l:allpolysBn} are the $p^M_{N,L}$ such that 
\[
    |N| + 2|M| + |L| \equiv  2\ceil{ \frac{n}{2}},  2\ceil{ \frac{n}{2}}-1 \mod 4\, .
\]
To obtain a basis for the subspace of $S^2(\Sp^\vee)$ orthogonal to $V_{2\lambda}$, we remove the $p^M_{N,L}$ such that $|N| + 2|M| + |L| \in \{n-1,n,n+1\}$. 
By the mod $4$ conditions, $|N| + 2|M| + |L|$ is never equal to $n-2$ or $n+2$: this is shown in the proof of Lemma \ref{l:leadingquadeqisingrassmmann}.
Thus we find a basis for the Lichtenstein embedding equations is the set of $p^M_{N,L}$ such that \begin{equation}\label{eq:conditions1}
    |N| + 2|M| + |L| \in \set{ 4k+ 2\ceil{ \frac{n}{2}}, 4k+  2\ceil{ \frac{n}{2}}-1 }{ k \in \mathbb{Z} }\setminus \{ n-2,n-1,n,n+1,n+2\}  \, .
\end{equation}
Subtracting $|N| + |L|$ from both sides of the conditions \eqref{eq:conditions1}, we find that $p^M_{N,L}$ is in our basis set precisely when 
\[
    2|M| \in   \set{4k + 2 \ceil{\frac{n}{2}} - |N| - |L| - \varepsilon }{ \varepsilon \in \{0, 1\},\, k\in\ZZ} \setminus \left\{ m-2,\,m-1,\,m,\,m+1,\,m+2\right\}\, ,
\]
or equivalently $2|M| \in \cM_{N,L}$.
\end{proof}

\subsubsection{Minimising the support}\label{App:BnSpinmain}

We now produce the type~$B$ quadratic embedding equations $F_{I,J}^{(B)}$ of \Cref{t:BnSpinmain}.
To work with these equations it will be convenient to rephrase and reindex them.
Recall the notation for subcubes $Q_{N,L}$ and ordered antipodes $(K,\overline{K})$ from \Cref{App:BspinEqs}. For such an ordered antipode $(K,\overline{K})$ with $\overline{K}=K\symDiff S$
\begin{equation}\label{eq:strong+ex}
  \mathrm{ex}_{(K,\overline{K})} =\mathrm{ex}_{(K,K \symDiff S)} = \sum_{i \in S} (-1)^{\lngS{i,K} + \lngS{i,K \symDiff S}}x_{K \symDiff i} x_{K \symDiff S \symDiff i} - \big [ |S| \equiv 1 \mod 2 \big ] x_K x_{K \symDiff S}\, .
\end{equation} 
Note that $\mathrm{ex}_{(K,K \symDiff S)} =0$ if and only if $|S| \leq 2$. Note that the quadratic embedding equation $F_{I,J}^{(B)}$ is equal via this reindexing to $\mathrm{ex}_{(I,\overline{I}_{N,L})}$ where $N = I \cap J$ and $L = (I \cup J)^c$.

In \Cref{l:allpolysBn} we presented the Lichtenstein embedding equations as
\[
    p^M_{N,L} = \sum_{\substack{(K,\overline{K})\\ \text{antipode of } Q_{N,L}}} (-1)^{|K^c \cap M| + [K,\overline{K}]} x_K x_{\overline{K}}
\]
for all $M$ with $2|M| \in \cM_{N,L}$.
The proof of \Cref{t:BspinEqs} shows that $p^M_{N,L}$ is symmetric under exchanging $K$ and $\overline{K}$, but in general $[K,\overline{K}]$ differs from $[\overline{K},K]$, and the two can even have different parities.

To prove \Cref{t:BnSpinmain} it suffices to show that the span of the quadratic embedding equations $\{\mathrm{ex}_{(K,K\symDiff S)}\}$ is equal to the span of the Lichtenstein embedding equations $\{p^M_{N,L}\}$. 
The rest of this subsection is dedicated to proving the equality of these subspaces of quadrics.
Corollary \ref{cor:quadsareingrassmann} is the forward inclusion and  \Cref{l:psareinex} is the reverse.

To compare the quadratic embedding equations with the Lichtenstein embedding equations, we must rewrite them with coefficients in terms of $[K,\overline{K}]$.
This is the content of Corollary \ref{c:signverificationofexeqs}.
We first give a technical lemma.

\begin{lemma} \label{l:antipode+coefficient+difference}
Let $(K, \overline{K})$ be an antipode in $Q_{N, L}$.
For all $i \in K \symDiff \overline{K} = [n] \setminus (N \cup L)$,
\[
[K,\overline{K}] \equiv [K \symDiff i, \overline{K} \symDiff i] + \lngS{i,K\symDiff i} + \lngS{i,\overline{K}\symDiff i} + n \mod 2 \, .
\]
\end{lemma}

\begin{proof}
   We use Definition~\ref{d:antipodesigns} to calculate the difference $[K,\overline{K}] - [K \symDiff i, \overline{K} \symDiff i]\mod 2$ term by term.
   We first observe that 
   \[
   |\overline{K}^c|(|N| + |L|) + n|\overline{K}| - |(\overline{K}\symDiff i)^c|(|N| + |L|) - n|\overline{K}\symDiff i| \equiv n + |N| + |L| \mod 2 \, .
   \]

    Next, let $\tau$ be the permutation where $\sigma_{(\overline{K}\symDiff i)^{c^T} \cat (\overline{K}\symDiff i)}$ is composed with the inverse of $\sigma_{\overline{K}^{c^T} \cat \overline{K}}$.
    Then $\tau$ transposes $i$ past all elements $j$ of $\overline{K}$ and $\overline{K}^c$ with $j < i$.
    Therefore, it follows that
   \[
    \lng{\overline{K}^{c^T} \cat \overline{K}} - \lng{\left(\overline{K}\symDiff i \right)^{c^T} \cat \left(\overline{K}\symDiff i \right)} \equiv \lng{\tau} \equiv \left|\set{j \in \overline{K}}{j < i}\right| + \left|\set{j \in \overline{K}^c}{j < i}\right| \mod 2\, .
   \]
    Repeating this argument for the other permutations in Definition~\ref{d:antipodesigns}, we get
   \begin{align*}
       \lng{L \cat \left(K \setminus N\right)} - \lng{L \cat \left((K\symDiff i) \setminus N\right)} &\equiv \left| \set{j \in L}{j > i}\right|\mod 2 \, , \\
       \lng{N^T \cat \left(K \setminus N\right)} - \lng{N^T \cat \left((K\symDiff i) \setminus N\right)} &\equiv \left| \set{j \in N}{j > i}\right| \mod 2\, .
   \end{align*}
    Before summing these, we make a number of observations.
    Firstly, as $i \notin N \cup L$, we have $\set{j \in L}{j > i} + \set{j \in L}{j < i} = L$ and $\set{j \in N}{j > i} = [n] \setminus \set{j \in N}{j < i}$.
    Furthermore, we can write $\overline{K}^c = (K\setminus N) \cup L$.
    Combining these facts, we have
    \begin{align*}
    [K,\overline{K}] - [K \symDiff i, \overline{K} \symDiff i] &\equiv n + |N| + |L| + \left|\set{j \in \overline{K}}{j < i}\right| + \left|\set{j \in \overline{K}^c}{j < i}\right| \\
    &\qquad  + \left| \set{j \in L}{j > i}\right| + \left| \set{j \in N}{j > i}\right| \mod 2 \\
    &\equiv n + |N| + \left|\set{j \in \overline{K}}{j < i}\right| + \left|\set{j \in K \setminus N}{j < i}\right| + \left| \set{j \in N}{j > i}\right| \mod 2 \\
    &\equiv n + \left|\set{j \in \overline{K}}{j < i}\right| + \left|\set{j \in K}{j < i}\right| \mod 2
    \end{align*}
    As $\lngS{i,K\symDiff i} = \left|\set{j \in K}{j < i}\right|$, the lemma follows.
\end{proof}

\begin{corollary}\label{c:signverificationofexeqs}
Up to a global sign, we have
\[
    \mathrm{ex}_{(K,K \symDiff S)} = \sum_{i \in S} (-1)^{[K \symDiff i, K \symDiff S \symDiff i]}x_{K \symDiff i} x_{K \symDiff S \symDiff i} + \big [ |S| \equiv 1 \mod 2 \big ](-1)^{n-1 + [K,K \symDiff S]} x_K x_{K \symDiff S} \, .
\]
\end{corollary}
\begin{proof}
If we multiply $\mathrm{ex}_{(K,K \symDiff S)}$ from Definition \ref{d:quadraticembeddingeqs} by the sign $(-1)^{n + [K,K \symDiff S]}$, the claim immediately follows from Lemma \ref{l:antipode+coefficient+difference}.
\end{proof}

We next introduce the variables $y_{(K,\overline{K})} = (-1)^{[K,\overline{K}]}x_Kx_{\overline{K}}$. 
We emphasise that as $(-1)^{[K,\overline{K}]}$ may differ from $(-1)^{[\overline{K},K]}$, the variables $y_{(K,\overline{K})}$ and $y_{(\overline{K},K)}$ may differ.
By Corollary \ref{c:signverificationofexeqs}, in these variables,  the quadratic embedding equations become
\[
    \mathrm{ex}_{(H,\overline{H})}(y) = \sum_{\substack{(K,\overline{K})\\|H \Delta K = 1|}} y_{(K,\overline{K})} + \big[k \equiv 1 \mod 2 \big] (-1)^{n-1}y_{(H,\overline{H})} \, ,
\]
and the Lichtenstein embedding equations become
\begin{equation}\label{eq:p^M_NL(y)}    
    p^M_{N,L}(y) = \sum_{\substack{(K ,\overline{K}) \\\text{ antipode in } Q_{N,L}}} (-1)^{|K^c \cap M|}y_{(K,\overline{K})}\, .
\end{equation}
This change of variables helps us prove the following.

\begin{lemma}\label{l:leadingquadeqisingrassmmann}
    For any $S \subseteq [n]$, the equation $\mathrm{ex}_{(\emptyset,S)}$ is in the subspace spanned by \[ \set{p^M_{\emptyset,S^c}}{2|M| \in \cM_{\emptyset,S^c}}.\]
\end{lemma}

\begin{proof}
    Our proof will proceed as follows.
    Fix $S$ of cardinality $m$.
    We first give an explicit description of which values of $|M|$ give rise to a Lichtenstein embedding equation $p_{\varnothing,S^c}^M$\,.
    Using this description we define intermediary linear combinations $\chi_b$ of the $p_{\varnothing,S^c}^M$ indexed by these possible values of~$|M|$.
    Finally we present $\mathrm{ex}_{(\emptyset,S)}$ as a linear combination of the $\chi_b$,
    and prove the claimed equality by induction after rewriting to introduce some auxiliary integer parameters.
    
    So, we begin with describing which $|M|$ are allowed in Lichtenstein embedding equations.
    Recall from \eqref{eq:M_NL} that
    \[
    \cM_{\varnothing,S^c} = \set{4k + 2\ceil{\frac{n}{2}} + (m-n) - \varepsilon}{k \in \ZZ\, , \, \varepsilon \in \{0,1\}} \setminus \left\{m-2, m-1,m,m+1,m+2\right\} \, .
    \]
    If $p^{M}_{\varnothing,S^c}$ is a valid Lichtenstein embedding equation, we necessarily have $M \in Q_{\varnothing,S^c}$, i.e.\ $M \subseteq S$.
    Hence we can restrict attention to the values in $\cM_{\varnothing,S^c}$ between $0$ and~$2m$.
    The further requirement that $2|M| \in \cM_{\varnothing,S^c}$ implies
    \begin{align*}
        |M| \equiv \ceil{\frac{n}{2}} + \floor{\frac{m-n}{2}} \mod 2 \, .
    \end{align*}
    We let $\gamma = \ceil{\frac{n}{2}} + \floor{\frac{m-n}{2}}$.
    To see which values of $|M|$ are excluded, observe that
    \begin{equation}\label{eq:gamma}
        \gamma =\begin{cases}
    \ceil{\frac{n}{2}} + \frac{m}{2} + \floor{\frac{-n}{2}} = \frac{m}{2} & m \equiv 0 \mod 2\, ,\\
    \ceil{\frac{n}{2}} + \frac{m-1}{2} + \floor{\frac{1-n}{2}} = \frac{m-1}{2} + (n \mod 2) & m \equiv 1 \mod 2\, .
\end{cases}
    \end{equation}
    In particular, if $|M| = \gamma$ then $2|M| \in \{m-1, m, m+1\}$ and hence $2|M| \notin \cM_{\varnothing,S^c}$.
    We further note that $m+2, m-2 \not\equiv \gamma \mod 4$ for any values of $m,n$, and hence these values have been `trivially' removed from $\cM_{\varnothing, S^c}$.
    Putting all this together, we see that
    \begin{equation} \label{eq:altMdescription}
    \set{p^M_{\emptyset,S^c}}{2|M| \in \cM_{\emptyset,S^c}} = \set{p^M_{\emptyset,S^c}}{0 \leq |M| \neq \gamma \leq m \, , \, |M| \equiv \gamma \mod 2 } \, .
    \end{equation}
    
    We next define our intermediary polynomials $\chi_b$, as follows:
    \[
    \chi_b =\sum_{|M| = b } p^M_{\varnothing,S^c}(y) \, . 
    \]
    We will actually be working with sums involving the polynomial $\chi_b$ for general~$b$,
    using \eqref{eq:p^M_NL(y)} to give meaning to the symbol $p^M_{\varnothing,S^c}$ in all cases.
    It is important that we require $2b \in \cM_{\varnothing,S^c}$ if $\chi_b$ is to be a linear combination of Lichtenstein embedding equations.
    It follows from \eqref{eq:altMdescription} that this is the case when $b \equiv \gamma \mod 2$ and $b \in [0,2m] \setminus \{\gamma\}$.
    
    Each polynomial $p^M_{\varnothing,S^c}$ has only terms $y_{(K,\overline{K})}$ in its support where $K \subseteq S$.
    Writing $\chi_b[y_{(K,\overline{K})}]$ for the coefficient of $y_{(K,\overline{K})}$ in $\chi_b$, we claim that
   \begin{equation} \label{eq:chi-b-n-a}
   \chi_b[y_{(K,\overline{K})}] = \sum_{|M| = b} (-1)^{|\overline{K} \cap M|} = \sum_k (-1)^k \binom{a}{k}\binom{m-a}{b-k}
   \end{equation}
   where $a = |\overline{K}|$.
   To see this, observe that $\chi_b[y_{(K,\overline{K})}] = \sum_{|M|=b} p^M_{\emptyset,S^c}[y_{(K,\overline{K})}] = \sum_{|M|=b} (-1)^{|K^c \cap M|}$.
   As $M \subseteq S$, we have $|K^c \cap M| = |\overline{K}\cap M|$ and hence $\chi_b[y_{(K,\overline{K})}]= \sum_k \sum_{|M| = b, |\overline{K} \cap M| = k} (-1)^{k}$.
    Counting the sets $M$ of size $b$ such that $\overline{K} \cap M$ is of size $k$ involves choosing $k$ elements from $\overline{K}$ followed by $b-k$ elements from $S \setminus \overline{K}$.
   Thus the number of $M$ of size $b$ such that $|\overline{K} \cap M| = k$ is $\binom{a}{k} \binom{m-a}{ b-k}$.

   We now show that $\ex_{(\varnothing,S)}$ can be written as a linear combination of those polynomials $\chi_b$ which are sums of Lichtenstein embedding equations. Proving this completes the proof.
   Explicitly, we claim
    \[
        \ex_{(\varnothing,S)} = \frac{1}{2^{m-2}}\sum_{\substack{b=0,\\ \, b \equiv \gamma \bmod 2}}^{m} \left(\gamma - b\right)  \chi_b \, .
    \]
    Note that the coefficient of~$\chi_\gamma$ is zero and hence the right hand side is indeed a linear combination of Lichtenstein embedding equations.
    By comparing coefficients of $y_{(K,\overline{K})}$, we can restate our above claim as the following relation for each $K \subseteq S$:
    \begin{equation} \label{eq:ex-chi-coefficients}
    \frac{1}{2^{m-3}}\sum_{\substack{b=0,\\ \, b \equiv \gamma \bmod 2}}^m \half\left(\gamma - b\right) \chi_b [y_{(K,\overline{K})}] = \begin{cases}  
    0 & |K| \in\{ 0, m\} \text{ and } m \text{ even}, \\ 
    (-1)^{n-1} & |K| \in\{0,m\}  \text{ and }  m\text{ odd}, \\ 
    1 & |K| \in \{1,m-1\}, \\ 
    0 & \text{ otherwise}. 
    \end{cases}
    \end{equation}

We prove \eqref{eq:ex-chi-coefficients} by instead considering a more general function
\[
S(m,a,c) \coloneqq \sum_{b\equiv c\bmod2}\frac12(c-b)\chi_b(m,a)
\]
where $c \in \ZZ$ and
\[
\chi_b(m,a) := \sum_k (-1)^k \binom{a}{k}\binom{m-a}{b-k} \, .
\]
Note the deliberate abuse of notation, as $\chi_b(m,a)$ is the coefficient $\chi_b[y_{(K,\overline{K})}]$ where $|\overline{K}| = a$ by \eqref{eq:chi-b-n-a}.
We show that for $m\ge3$,
\begin{equation}\label{eq:S(m,a,c)}
S(m,a,c) = \begin{cases}
    (c-\frac m2)2^{m-2} & a=0\\
    2^{m-3} & a=1\\
    0 & a,m-a\ge2.
\end{cases}
\end{equation}
If this holds, then specialising to $S(m,a,\gamma)$ and using \eqref{eq:gamma} gives \eqref{eq:ex-chi-coefficients} directly and completes the proof.

We have $\chi_b(m,0)=\binom mb$.
Write $S^+(m,c)=\sum_b(c-b)\binom mb$ and $S^-(m,c)=\sum_b(-1)^b(c-b)\binom mb$, both without parity conditions on~$b$,
so that 
\[S(m,0,c)=\frac14(S^+(m,c)+(-1)^cS^-(m,c)).\]
The sums $S^+(m,c)$ and $S^-(m,c)$ can be evaluated using the binomial theorem at $1$ and~$-1$ respectively for the linear term in~$c$
and the first derivative of the binomial theorem for the constant term:
the result is that $S^+(m,c)=c2^m-m2^{m-1}$ and $S^-(m,c)=0$.
This proves \eqref{eq:S(m,a,c)} for $a=0$.

When $a>0$ we have 
\begin{align*}
\chi_b(m,a)&=\sum_k(-1)^k\left(\binom{a-1}k+\binom{a-1}{k-1}\right)\binom{m-a}{b-k}
\\&=\sum_k(-1)^k\binom{a-1}k\binom{(m-1)-(a-1)}{b-k} - \sum_k(-1)^{k-1}\binom{a-1}{k-1}\binom{(m-1)-(a-1)}{(b-1)-(k-1)}
\\&=\chi_b(m-1,a-1)-\chi_{b-1}(m-1,a-1)
\end{align*}
implying that $S(m,a,c)=S(m-1,a-1,c)-S(m-1,a-1,c-1)$. 
Since the right hand side of \eqref{eq:S(m,a,c)} satisfies the same recurrence,
the proof is complete after manually checking the base case $m=3$.

We finish by doing this check. When $c$ is even,
\begin{align*}
S(3,a,c) &= \frac c2\chi_0(3,a)+\frac{c-2}2\chi_2(3,a)
\\&= \frac c2\binom a0\binom{3-a}0
+ \frac{c-2}2\left(\binom a2\binom{3-a}0-\binom a1\binom{3-a}1+\binom a0\binom{3-a}2\right)
\\&= \frac c2 + \frac{c-2}2(2a^2-6a+3).
\end{align*}
The evaluations of~$a$ appearing in~\eqref{eq:S(m,a,c)} are
\begin{align*}
S(3,0,c) &= 2c-3 = 2^{3-2}(c-\textstyle\frac32),\\
S(3,1,c) &= 1 = 2^{3-3}.
\end{align*}
In the same way, when $c$ is odd,
\begin{align*}
S(3,a,c) &= \frac{c-1}2\chi_0(3,a)+\frac{c-3}2\chi_2(3,a)
\\&= \frac{c-1}2(-2a+3)
+ \frac{c-3}2(-\frac43a^3+6a^2-\frac{20}3a+1),
\end{align*}
and the evaluations in~\eqref{eq:S(m,a,c)} are again
\begin{align*}
S(3,0,c) &= 2c-3,\\
S(3,1,c) &= 1
\end{align*}
as desired.
\end{proof}

Having shown that every $\ex_{(\emptyset,S)}$ is in the span of the Lichtenstein embedding equations, 
the next step is to generalise to all quadratic embedding equations.
Every antipode of every subcube of the $n$-cube is in the orbit of an antipode $\{\emptyset,S\}$ under reflections in the coordinate hyperplanes,
so our strategy is essentially to apply these reflections.
But we cannot literally act by symmetries of the cube;
the symmetries we have are those of the spin representation $\mathrm{Spin}(2n+1)$.
To introduce these, we resume the notation relating to the vector representation $V =\CC^{2n+1}$ of $\mathfrak{so}_{2n+1}$ from Appendix~\ref{App:B}.

The group $\mathrm{Pin}(2n+1)$ is defined as the subset of the Clifford algebra $\Cl(V)$ multiplicatively generated by every $v \in V$ such that $B(v,v) =1$.
The Spin group $\mathrm{Spin}(2n+1)$ is the subset $\mathrm{Pin}(V) \cap \Cl^{\mathrm{even}}(V)$ of even degree elements.
Since $x_i$ and $v_i$ are isotropic and $B(x_i,v_i) = 1$, we have that $\frac{1}{\sqrt{2}}(x_i + v_i)$ is in $\mathrm{Pin}(V)$. Also $B(w,w) = 1$, whence
\[
    \frac{1}{\sqrt{2}}(x_i + v_i)w \in \mathrm{Spin}(2n+1)\, .
\]
Recall (Lemma \ref{p:cliffisend}) that $x_i$ acts on $\Sp^\vee$ by left multiplication by $\sqrt{2}x_i$ (denoted $\sqrt{2} \epsilon_{x_i}$) and $v_i$ acts on $\Sp^\vee$ by contraction by $\sqrt{2}v_i$ (denoted $\sqrt{2}\iota_{v_i}$). Hence the element $g_i = \frac{1}{\sqrt{2}} (x_i + v_i) w$ is such that
\[
    g_i(x_A) = \begin{cases}
        (-1)^{|A|+\lngS{i,A}} x_{A \setminus i}  & i \in A\, ,\\
        (-1)^{|A|} x_{i \cat A}  & i \notin A\, .
        \end{cases}
\]
Thus in both cases, in the notation of \Cref{sssec:Symmetric-group}, we have 
\[
    g_i(x_A) = (-1)^{|A| + \lngS{i,A}} x_{\{i\}\symDiff A}\, .
\]
Since $g_i$ is a group element, we have
\[
    g_i(x_Ax_B) = g_i(x_A)g_i(x_B)\, .
\]

\begin{lemma}\label{l:quadembcomefrombasicantipode}
    If $H=\{h_1,\ldots,h_{k}\}$ with $h_1<\cdots<h_{k}$,
    let $g_H = g_{h_1}\cdots g_{h_{k}}$.
    Then the quadratic embedding equation $\mathrm{ex}_{(H, \overline{H})}$ satisfies
\[
    \ex_{(H, \overline{H})} = (-1)^{|H||H\symDiff\overline{H}|+\lngS{H,H\symDiff\overline{H}}}\,
    g_H\ex_{(\emptyset,H\symDiff\overline{H})}\, .
\]
\end{lemma}

\begin{proof}
For any three sets $A,B,C$, we have the associativity relation
\[
    (-1)^{\lngS{A,B}+\lngS{A\symDiff B,C}} = 
    (-1)^{\lngS{B,C}+\lngS{A,B\symDiff C}}\, .
\]
We see that both of these are equal to the sign of the multipermutation sorting $A\cat B\cat C$ into weakly increasing order,
once we observe that replacing the intermediate weakly increasing ordered multiset $A\circ B$ 
with the increasingly ordered set $A\symDiff B$
only deletes pairs of adjacent letters,
which removes an even number of transpositions from the subsequent sorting.

Thus, for any $A\subseteq[n]$, we have
\begin{align*}
    g_H(x_A) &= (-1)^{|H||A| + \lngS{h_1,A} + \lngS{h_2,A\symDiff h_1} + \cdots + \lngS{h_{k},A\symDiff(H\setminus h_{k})}}\, x_{H\symDiff A}\\
    &= (-1)^{|H||A| + \lngS{H,A}}\, x_{H\symDiff A}\, ,
\end{align*}
where the second equality can be seen as repeated use of the above associativity and the fact that $h_1,\ldots,h_{k}$ is increasing.
Write $S=H\symDiff \overline{H}$, so $\overline{H}=H\symDiff S$.
Then 
\[
    g_H(\ex_{(\emptyset,S)}) = 
    (-1)^{|H||S|}
    \left(\sum_{i\in S}
        (-1)^{\lngS{i,S\setminus i}+\lngS{H,i}+\lngS{H,S\setminus i}}x_{H\symDiff i}x_{H\symDiff S\symDiff i}
        - [|S|\equiv1(2)](-1)^{\lngS{H,S}}x_Hx_{H\symDiff S}
    \right)\, .
\]
We multiply through by the sign $(-1)^{\lngS{H,S}}$.
This makes the sign inside the sum
\begin{align*}
    (-1)^{\lngS{i,S\setminus i}+\lngS{H,S}+\lngS{H,i}+\lngS{H,S\setminus i}}
    &= (-1)^{\lngS{H,i}+\lngS{i\symDiff H,S\setminus i}+\lngS{H,i}+\lngS{H,S\setminus i}}
    \\&= (-1)^{\lngS{i\symDiff H,S\setminus i}+\lngS{H,S\setminus i}}
    \\&= (-1)^{\lngS{i,H}+\lngS{i,H\symDiff (S\setminus i)}}
    \\&= (-1)^{\lngS{i,H}+\lngS{i,H\symDiff S}}
\end{align*}
where the first equality is associativity on the sets $H$, $\{i\}$, $S\setminus i$,
the third is associativity on the sets $\{i\}$, $H$, $S\setminus i$ with the terms collected differently,
and the last is the observation that adding or removing an $i$ from $H\symDiff S$ does not change the position that another $i$ on the left would sort into.
\end{proof}

\begin{corollary}\label{cor:quadsareingrassmann}
    Every quadratic embedding equation $\mathrm{ex}_{(H, \overline{H})}$ lies in the span of the Lichtenstein embedding equations $\{ p^M_{N,L}\}$.
  
\end{corollary}

\begin{proof}
    This follows from Lemma~\ref{l:quadembcomefrombasicantipode} because the span of this set of $p^M_{N,L}$ is closed under the action of the $g_h$.
\end{proof}

\begin{lemma}\label{l:psareinex}
    Fix $N$ and $L$. Then every $p^M_{N,L}$ for $2|M| \in \cM_{N,L}$ is in the span of the quadratic embedding equations $\{\mathrm{ex}_{(H,\overline{H})}: (H,\overline{H}) \text{ an antipode of } Q_{N,L}\}$.
\end{lemma}

\begin{proof}
    Recall $m = n-|N|-|L| = |K \symDiff \overline{K}|$. Let $\gamma_{|M|} =  -(-1)^n(m-2|M|)+(-1)[m \equiv 1\ \mod 2]$.
    To prove the lemma we claim, for $M$ such that $2|M| \in\cM_{N,L}$, that

    \begin{equation}\label{eq:pexpressedasex}
         p^M_{N,L} = \frac{1}{\gamma_{|M|}}\sum_{\substack{(K,\overline{K}) \\ \text{ an antipode of }Q_{N,L}}} p^M_{N,L}(v_Kv_{\overline{K}}) \mathrm{ex}_{(K,K^c)}\, .
    \end{equation}
    Note that $2|M| \in \cM_{N,L}$ precludes $\gamma_{|M|} = 0$.
    To show that Equation (\ref{eq:pexpressedasex}) holds, we compare the coefficient of $x_Ix_{\overline{I}}$ on each side. Note that
    \[ p^M_{N,L}[x_Ix_{\overline{I}}] = (-1)^{|I^c \cap M|+[I,\overline{I}]}.
    \] 
    Thus we must show that the coefficient of $x_Ix_{\overline{I}}$ in the right hand side of (\ref{eq:pexpressedasex}) is also $(-1)^{|I^c \cap M|+[I,\overline{I}]}$.

    Note that \[\mathrm{ex}_{(K,\overline{K})}[x_Ix_{\overline{I}}] =\begin{cases} -1[m \equiv 1 \mod 2] & \{K,\overline{K}\} = \{I,\overline{I}\} \\
    (-1)^{\lngS{i,K} + \lngS{i,\overline{K}}} & I \symDiff K =\{i\} \text{ or } I \symDiff \overline{K} = \{i\} \\
    0 & \text{otherwise.}
    \end{cases}
    \]
    Equivalently, 
\[\mathrm{ex}_{(K,\overline{K})}[x_Ix_{\overline{I}}] =\begin{cases} -1[m \equiv 1 \mod 2] & \{K,\overline{K}\} = \{I,\overline{I}\} \\
    (-1)^{\lngS{i,I \symDiff i} + \lngS{i,\overline{I}\symDiff i}} & I \symDiff K =\{i\} \text{ or } I \symDiff \overline{K} = \{i\} \\
    0 & \text{otherwise.}
    \end{cases}
    \]
    
     Hence, we find 
    \begin{align*}
     \frac{1}{\gamma_{|M|}}&\sum_{\substack{(K,\overline{K}) \\\text{an antipode of }Q_{N,L} }} p^M_{N,L}(v_Kv_{K^c}) \ex_{(K,K^c)}[x_Ix_{\overline{I}}]\\ 
    &= \frac{1}{\gamma_{|M|}}\left(-p^M_{N,L}(v_Iv_{I^c})[m \equiv 1 \bmod 2] +\sum_{i\in [n] \setminus (N \cup L)} p^M_{N,L}(v_{I \symDiff i}v_{\overline{I}\symDiff i})(-1)^{\lngS{i,I \symDiff i} + \lngS{i,\overline{I}\symDiff i}} \right)\\
    &= \frac{1}{\gamma_{|M|}}\left(-(-1)^{|I^c \cap M|+[I,\overline I]}[m \equiv 1 \bmod 2] +\sum_{i\in [n] \setminus (N \cup L)} (-1)^{|(I\symDiff i)^c \cap M|+[I \symDiff i,\overline I \symDiff i]+\lngS{i,I \symDiff i} + \lngS{i,\overline{I}\symDiff i}}\right).
   \end{align*}

   Note that, if $i \notin M$, then adding or subtracting $i$ from $I$ does not affect the intersection with $M$. If $i$ is in~$M$, this increases or decreases the intersection by $1$. Thus 
     \[
        |(I\symDiff i)^c \cap M| \equiv \begin{cases}
            |I^c \cap M| \mod 2 & i \notin M\, , \\
            |I^c \cap M| + 1 \mod 2 & i \in M\, .
        \end{cases}
    \]

    We also recall from Lemma \ref{l:antipode+coefficient+difference} that
   \[
       [I ,\overline{I}]\equiv [I \symDiff i,\overline{I} \symDiff i]   +\lngS{i,I \symDiff i} + \lngS{i,\overline{I}\symDiff i} +n \mod 2\, .
   \]
   Using these two facts we have that the coefficient of the right-hand side of (\ref{eq:pexpressedasex}) is 
   \begin{align*}\frac{1}{\gamma_{|M|}}\left(-(-1)^{|I^c \cap M| + [I,\overline{I}]}[m \equiv 1 \mod 2] - \sum_{i \in M}(-1)^{|I^c \cap M| + [I,\overline{I}]  +n}+ \sum_{i \in [n]\setminus (N \cup L \cup M)}(-1)^{|I^c \cap M| + [I,\overline{I}] +n}\right)\\
   = (-1)^{|I^c \cap M| + [I,\overline{I}]}\frac{- [m \equiv 1 \mod 2]-  (-1)^n(|M|+ m-|M|)}{\gamma_{|M|}}\\
   =(-1)^{|I^c \cap M| + [I,\overline{I}]}\,.
   \end{align*}
   Hence the coefficient of the two sides of (\ref{eq:pexpressedasex}) are equal for every $x_I x_{\overline{I}}$ and the result follows. 
\end{proof}

\begin{proof}[Proof of Theorem~\ref{t:BnSpinmain}]
    Corollary \ref{cor:quadsareingrassmann} and Lemma \ref{l:psareinex} prove that the span of the quadratic embedding equations $\{\mathrm{ex}_{(K,K\symDiff S)}\}$ is equal to the span of the Lichtenstein embedding equations $\{p^M_{N,L}\}$.
    As $\{p^M_{N,L}\}$ cut out $\bbG/\PP \subset \proj(\Sp)$ by Theorem~\ref{t:BspinEqs}, our theorem is proved.
\end{proof}

\subsection{Equations from Lichtenstein: Type $D_n$}\label{App:D}

We now proceed to find a basis for the Lichtenstein embedding equations for type $D_n$ with its minuscule representation $\Sp_+$. The organisation and procedure is identical to type $B_n$ above.
We focus on $\Sp_+$ since the two varieties associated to $\Sp_+$ and $\Sp_-$ are the two connected components of the isotropic Grassmannian and are diffeomorphic as varieties, and furthermore the weights of $\Sp_+$ and $\Sp_-$ are isomorphic as (transitive) sets with a $W$-action. So for our purposes these cases are equivalent.

We can model $\Sp_+$ as follows.
Recall that $V = \CC^{2n}$ is an even-dimensional vector space with non-degenerate symmetric bilinear form $B$.
Following~\cite[Section 1]{Meinrenken:2013}, there exists maximal isotropic subspaces $F$ with respect to $B$ such that $V = F \oplus F^*$. 
Moreover, we can fix bases $\{v_1, \dots, v_n\}$ for $F$ and $\{x_1, \dots, x_n\}$ for $F^*$ such that $B(v_i, x_j) = \delta_{ij}$.
We model $\Sp_+$ by writing $\Sp_+ = \bigwedge^{\mathrm{even}} F$ with inherited basis
\[
    \set{v_I \coloneqq v_{i_1} \wedge \cdots \wedge v_{i_{2k}} }{ I = \{i_1, \dots, i_{2k}\} \subseteq [n]} \, .
\]
Here we implicitly assume $I$ is ordered $i_1 < \cdots < i_{2k}$ for this basis and $|I|$ is even. Similarly we model the module $\Sp_-$ with $\bigwedge^{\mathrm{odd}}F$.
Note that $v_I$ is a basis for the weight space with $+1/2$ in the $i^{\rm th}$ position if $i \in I$ and $-1/2$ otherwise.
Also, $\Sp_+$ has weights with an even number of positive and negative entries and $\Sp_-$ has an odd number of positive and negative entries.
Let $\Sp = \Sp_+ \oplus \Sp_-$.
Then the $\mathfrak{so}_{2n}$-representation is precisely the restriction of the $\mathfrak{so}_{2n+1}$-representation $\Sp$ considered in Section \ref{App:B}. Therefore the isomorphism of $\mathfrak{so}_{2n+1}$-representations considered induces an isomorphism of $\mathfrak{so}_{2n}$-representations. Since $\Sp =\Sp_+ \oplus \Sp_-$ we have 
\[
    S^2(\Sp^\vee) = S^2(\Sp_+^\vee) \oplus S^2(\Sp_-^\vee) \oplus \Sp_+^\vee\otimes \Sp_-^\vee\, .
\]
Theorem \ref{t:BspinEqs} shows that $S^2(\Sp^\vee)$ is spanned by equations $p^M_{N,L}$. Thus to find the Lichtenstein embedding equations, we need to find the subspace of the span of these equations that spans $S^2(\Sp_+^\vee)$.

\thDneqs*

\begin{proof}
Note that $\Sp$ has a basis given by $\cQ_n$ and a basis for $\Sp_+$ is given by the even demicube $\cD_{n}^+$. The space $S^2(\Sp^\vee)$ can be considered as the space of symmetric $2$-forms on $\Sp$. The space $S^2(\Sp_+^\vee)$ is the space of symmetric $2$-forms on $\Sp_+$ and can be described as all symmetric $2$-forms on $\Sp$ that vanish on $\Sp_-$, equivalently vanishing on $\cD_{n,-}$. Thus the space of Lichtenstein embedding equations of $\bbG/\PP_{\alpha_{n}}$ for type $D_n$ is precisely the subspace of Lichtenstein embedding equations of type $B_n$ that are supported only on the even demicube and not the odd demicube. An identical argument works with $\bbG/\PP_{\alpha_{n-1}}$ and~$\Sp_-$.
\end{proof}

As with the odd-dimensional case, the Spin module $\Sp = \twedge^{\bullet} F$ is a $\Cl(V)$ module via the action 
   \[
        \Cl^{}(V) \to \End(\Sp), \quad v\in F \mapsto \sqrt{2}\epsilon_v \, , \, x\in F^* \mapsto \sqrt{2}\iota_x\, .
    \]
Note that $\epsilon_v$ and $\iota_x$ exchange $\Sp_+=\twedge^{\mathrm{even}}F$ and $\Sp_-=\twedge^{\mathrm{odd}}F$, hence $\Sp_+$ and $\Sp_-$ are $\Cl^{\mathrm{even}}(V)$ submodules. 
To end, we utilise this Clifford structure to demonstrate that the spin representation of $\Spin(2n+1)$ is isomorphic to either of the half spin representations of $\Spin(2n+2)$.
As discussed in \Cref{rem:B+D+Gr+iso}, this gives us an isomorphism between the type $B_n$ embedding equations and the type $D_{n+1}$ embedding equations.

Let $V_{2n+2}$ be of dimension $2n+2$ with maximal isotropic spaces $F_{n+1}$ and $F^*_{n+1}$ such that we have $V_{2n+2} = F_{n+1} \oplus F^*_{n+1}$.
As described earlier, the spin module $\Sp_{2n+2} = \twedge^\bullet F_{n+1}$ splits into the two irreducible 
half-spin representations of $\Spin(2n+2)$, namely $\Sp_{2n+2,+} = \twedge^{\mathrm{even}} F_{n+1}$ and $\Sp_{2n+2,-} = \twedge^{\mathrm{odd}} F_{n+1}$, and these are both $\Cl^{\mathrm{even}}(V_{2n+2})$ modules.
Let $v_1,\ldots,v_{n+1}$ be a basis for $F_{n+1}$, with dual basis $x_1,\ldots, x_{n+1}$ in $F^*_{n+1}$.
Let $F_{n}$ and $F^*_{n}$ be spaces spanned by $v_1,\ldots,v_{n}$ and $x_1,\ldots,x_{n}$ respectively and define $V_{2n+1} = F_{n} \oplus F^*_{n} \oplus \textrm{Span}\{w\}$ where $w = \sqrt{\frac{-1}{2}}(v_{n+1} - x_{n+1})$.
Recall that the spin representation $\Sp_{2n+1}$ of $\Spin(2n+1)$ can be modelled as $\Sp_{2n+1} = \twedge^\bullet F_n$.
Moreover, it is a $\Cl(V_{2n+1})$ module where $v_i$ acts by $\epsilon_{v_i}$, $x_i$ acts by $\iota_{x_i}$ and $w$ acts by $w \cdot v_I = (-1)^{|I|} v_I$.

Set $e_{2n+2}$ to be $\frac{1}{\sqrt{2}}(x_{n+1} + v_{n+1})$. 
There is an algebra isomorphism \cite[pg 79 (3.17)]{Meinrenken:2013}
\[ 
\Cl(V_{2n+1}) \cong \Cl^{\mathrm{even}}(V_{2n+2}), \quad v \mapsto \sqrt{-1}v e_{2n+2}\, .
\]
This allows us to consider $\Sp_{2n+2,+}$ and $\Sp_{2n+2,-}$ also as $\Cl(V_{2n+1})$ modules.

\begin{lemma}\label{lem:isospinmod}
There is an isomorphism of $\Cl(V_{2n+1})$ modules
\[ \phi:\twedge^\bullet F_n \cong \twedge ^{\mathrm{even}} F_{n+1}, \quad
v_I \mapsto \begin{cases} v_I & |I| \text{ even}, \\\sqrt{-1}v_{I \cup \{n+1\}} & |I| \text{ odd}. \end{cases}
\]    
\end{lemma}

\begin{proof}
We verify the property that for all $v_I \in \twedge ^\bullet F_n$ and $u$ in the basis $\set{v_{i},x_i, w}{i \in \{ 1,\ldots,n\}}$ that 
\[
\sqrt{-1}ue_{2n+2} \cdot \phi(v_I) = \phi(u \cdot v_I) ,
\]
where $\cdot$ denotes the Clifford algebra actions on the respective spin modules. 
Note that $w$ acts on $\twedge^\bullet F_{n+1}$ by $\sqrt{-1}(  \epsilon_{x_{n+1}}-\iota_{v_{n+1}} )$ and $e_{2n+2}$ acts by $\epsilon_{v_{n+1}} +\iota_{v_{n+1}}$. 

First, suppose $|I|$ is even. Note that we always have $e_{2n+2}\cdot \phi(v_I) = \epsilon_{v_{n+1}}v_I = v_{I \cup \{n+1\}}$, hence
     \begin{align*}
    \sqrt{-1}v_i e_{2n+2} \cdot \phi( v_{I}) &=\sqrt{-1}v_i \cdot  v_{I\cup \{n+1\}}= \sqrt{-1}v_i \wedge  v_{I\cup \{n+1\}}
    =\phi (v_i \cdot v_I)    
   \\ 
    \sqrt{-1}x_i e_{2n+2} \cdot \phi(v_{I}) &= \sqrt{-1}x_i \cdot  v_{I\cup \{n+1\}} = \phi(x_i \cdot v_I)
    \\
    \sqrt{-1} w e_{2n+2} \cdot \phi(v_I) &= \sqrt{-1} w \cdot  v_{I\cup \{n+1\}}  = \sqrt{-1}( -\sqrt{-1}\iota_{x_{n+1}}) v_{I \cup \{n+1\}} = v_I   
    =\phi (w \cdot v_I).
    \end{align*}
    Now suppose $|I|$ is odd. Note that we always have $e_{2n+2}\cdot \phi(v_I) = \sqrt{-1}\iota_{v_{n+1}}v_{I\cup\{n+1\}} = -\sqrt{-1}v_{I}$, hence
   \begin{align*}
    \sqrt{-1}v_i e_{2n+2} \cdot \phi( v_{I}) &= v_i \cdot v_{I} =\phi (v_i \cdot v_I)    
   \\ 
    \sqrt{-1}x_i e_{2n+2} \cdot \phi(v_{I}) &= x_i \cdot v_{I} = \phi(x_i \cdot v_I)
    \\
    \sqrt{-1} w e_{2n+2} \cdot \phi(v_I) &= w \cdot v_{I} =\sqrt{-1} \epsilon_{v_{n+1}} v_I= -\sqrt{-1}v_{I\cup \{n+1\}}  
    =\phi (w \cdot v_I).
    \end{align*}
Thus $\phi$ is a morphism of $\Cl(V_{2n+1})$ modules, and hence $\mathrm{Spin}(2n+1)$ modules.
A similar isomorphism exists from $\twedge^\bullet F_n$ to $\twedge^{\mathrm{odd}}F_{n+1}$.
\end{proof}

\section{Code}\label{app:code}

The following is a condensed version of the code we used for our proofs.
The full version of the code can be found at \cite{githubrepo}.

\subsection{Proof of Lichtenstein basis}
\label{ss:grassmanianBasis}
The following code is the proofs for \Cref{p:mainthmE6} and \Cref{p:mainthmE7}.
We used \texttt{GAP 4.13.1} \cite{GAP4}.
\begin{lstlisting}[language=GAP]
ty := 6; # 6 or 7

ExtRep := function(elt)
    return ExtRepOfObj(ExtRepOfObj(elt));
end;;

sBasisWt := function(elt)
    local extElt;
    extElt:= ExtRep(elt);
    return ExtRep(extElt[1][1])[1][3] + ExtRep(extElt[1][2])[1][3];
end;;

if ty = 6 then
    hwv:= [0,0,0,0,0,1];
    owv:=[1,0,0,0,0,0];
else
    hwv:= [0,0,0,0,0,0,1];
    owv:=[1,0,0,0,0,0,0];
fi;

L:= SimpleLieAlgebra( "E", ty, Rationals );;
V:=HighestWeightModule(L, hwv);
S:=SymmetricPowerOfAlgebraModule(V, 2);
SBasis:=Basis(S);

sOWVData:=[];
for s in SBasis do
    if sBasisWt(s) = owv then
        Append(sOWVData, [s]);
    fi;
od;

csOWVData:=[1*sOWVData[1], -1*sOWVData[2], -1*sOWVData[3], -1*sOWVData[4], 1*sOWVData[5]];
if not ty = 6 then
    csOWVData[5]:=-1*sOWVData[5];
    csOWVData[6]:=1*sOWVData[6];
fi;
linCombo:=Sum(csOWVData);

subAlgModule:=SubAlgebraModule(S, [linCombo]);
dim:=Dimension(subAlgModule);
deg:=DimensionOfHighestWeightModule(L, owv);
if dim = deg then
    Print("Dimension is as expected - ", dim, ".\n\n");
else
    Print("Dimension is NOT as expected.\n\n");
fi;
Print("The basis of the subalgebra module is given by:\n");
for sb in Basis(subAlgModule) do
    Print(sb, "\n");
od;
\end{lstlisting}


\end{document}